\documentclass[12pt,a4paper,reqno]{amsart}
\usepackage[utf8]{inputenc}
\usepackage[T2A,T1]{fontenc}
\usepackage[english]{babel}
\usepackage{amsmath,amssymb,amsthm,amsfonts,wasysym,calc,verbatim,enumitem,url,hyperref,mathrsfs,bbm,cite,fullpage}
\usepackage{bm}
\usepackage{mathtools}
\usepackage{lmodern}
\usepackage{caption}
\usepackage{subcaption}
\usepackage{floatrow}
\usepackage{pgf,tikz}

\usetikzlibrary{shapes.misc,calc,intersections,patterns,decorations.pathreplacing}
\hypersetup{colorlinks=true,citecolor=blue,linktoc=page}

\usetikzlibrary{decorations.markings}
\usetikzlibrary{calc,positioning,decorations.pathmorphing,decorations.pathreplacing}
\usetikzlibrary{arrows} 

\newtheorem{thm}{Theorem}
\newtheorem{cor}[thm]{Corollary}
\newtheorem{lem}[thm]{Lemma}
\newtheorem{prop}[thm]{Proposition}
\newtheorem{conj}[thm]{Conjecture}

\newtheorem*{AL}{The Aizenman--Lebowitz lemma}
\newtheorem*{vBKlemma}{The van den Berg--Kesten Lemma}
\newtheorem*{Reimer}{Reimer's Theorem}
\theoremstyle{definition}
\newtheorem{defn}[thm]{Definition}

\newtheorem{alg}[thm]{Algorithm}
\newtheorem{ex}[thm]{Example}
\newtheorem{obs}[thm]{Observation}

\numberwithin{thm}{section}
\def\cE{\mathcal{E}}
\def\cF{\mathcal{F}}

\def\cH{\mathcal{H}}
\def\cI{\mathcal{I}}
\def\cJ{\mathcal{J}}

\def\cU{\mathcal{U}}

\def\cY{\mathcal{Y}}

\def\N{\mathbb{N}}
\def\Pr{\mathbb{P}}

\def\Z{\mathbb{Z}}

\def\d{\mathbf{d}}

\def\x{\mathbf{x}}

\def\1{\mathbbm{1}}

\def\<{\langle}
\def\>{\rangle}
\def\ds{\displaystyle}

\renewcommand{\leq}{\leqslant}

\renewcommand{\le}{\leqslant}
\renewcommand{\ge}{\geqslant}
\renewcommand{\to}{\rightarrow}

\def\Sfr{S_\square^\x}
\def\Sbar{S_\blacksquare^\x}
\def\<{\langle}
\def\>{\rangle}
\def\0{\mathbf{0}}

\newcommand{\md}{\mathrm{\;\!d}}
\newcommand{\sh}{\mathrm{\,\!short}}
\newcommand{\lgs}{\mathrm{\,\!long}}
\newcommand{\tr}{\mathrm{\,\!tr}}
\newcommand{\utr}{\mathrm{\,\!up}}
\title{The second term for two-neighbour \\ bootstrap percolation in two dimensions}
\author[I. Hartarsky \and R. Morris]{Ivailo Hartarsky \and Robert Morris}

\address{D\'epartement de Math\'ematiques et Applications, \'Ecole Normale Sup\'erieure, CNRS, PSL Research University, Sorbonne Universités, 45 rue d'Ulm, Paris, France}
\email{ivailo.hartarsky@ens.fr}
\address{IMPA, Estrada Dona Castorina 110, Jardim Bot\^anico, Rio de Janeiro, 22460-320, Brazil}
\email{rob@impa.br}
\thanks{Both authors were partially supported by ERC Starting Grant 680275 MALIG. The research of the second author was also partially supported by CNPq (Proc.~303275/2013-8), by FAPERJ (Proc.~201.598/2014), and by JSPS}

\begin{document}

\begin{abstract}
In the $r$-neighbour bootstrap process on a graph $G$, vertices are infected (in each time step) if they have at least $r$ already-infected neighbours. Motivated by its close connections to models from statistical physics, such as the Ising model of ferromagnetism, and kinetically constrained spin models of the liquid-glass transition, the most extensively-studied case is the two-neighbour bootstrap process on the two-dimensional grid $[n]^2$. Around 15 years ago, in a major breakthrough, Holroyd determined the sharp threshold for percolation in this model, and his bounds were subsequently sharpened further by Gravner and Holroyd, and by Gravner, Holroyd and Morris. 

In this paper we strengthen the lower bound of Gravner, Holroyd and Morris by proving that the critical probability $p_c\big( [n]^2,2 \big)$ for percolation in the two-neighbour model on $[n]^2$ satisfies
\[p_c\big( [n]^2,2 \big) = \frac{\pi^2}{18\log n} - \frac{\Theta(1)}{(\log n)^{3/2}}\,.\]
The proof of this result requires a very precise understanding of the typical growth of a critical droplet, and involves a number of technical innovations. We expect these to have other applications, for example, to the study of more general two-dimensional cellular automata, and to the $r$-neighbour process in higher dimensions.
\end{abstract}

\maketitle

\section{Introduction}\label{sec:intro}

In this paper we study the phase transition of the two-neighbour bootstrap percolation process on the two-dimensional grid $[n]^2$. This model has been extensively studied over the past 30 years, most notably by Aizenman and Lebowitz~\cite{Aizenman88} and by Holroyd~\cite{Holroyd03}, who determined the sharp threshold for percolation. In this paper we will improve the best known lower bound on the critical probability, proving a bound which matches the upper bound of Gravner and Holroyd~\cite{Gravner08} up to a constant factor in the second term. 

The $r$-neighbour bootstrap percolation process, which was first introduced in 1979 by Chalupa, Leath and Reich~\cite{Chalupa79}, is a deterministic, monotone cellular automaton, defined on a graph $G$. Given a set $A \subset V(G)$ of initially `infected' vertices, new vertices are infected at each time step if they have at least $r$ infected neighbours, and infected vertices remain infected forever. Thus, writing $A_t$ for the set of infected vertices at time $t$, we have $A_0 = A$, and
\[A_{t+1} = A_t \cup \big\{ v \in V(G) : |N(v) \cap A_t| \ge r \big\}\]
for each $t \ge 0$. We write $[A] := \bigcup_{t \ge 0} A_t$ for the \emph{closure} of $A$ under the bootstrap process (that is, set of eventually infected sites), and say that $A$ \emph{percolates} if $[A] = V(G)$. 

Motivated by applications to statistical physics, for example the Ising model~\cite{Fontes02,Morris11}, kinetically constrained spin models of the liquid-glass transition~\cite{Martinelli18,Morris17}, and the abelian sandpile~\cite{Fey10,Morris17}, the bootstrap process has been most extensively studied on finite subsets of the lattice $\Z^d$, with the initial set $A$ chosen randomly. More precisely, let $G$ be the finite grid $[n]^d = \{1,\ldots n\}^d$ with graph structure inherited from $\Z^d$, write $\Pr_p$ for the probability measure on subsets $A \subset [n]^d$ obtained by including each vertex in $A$ independently at random with probability $p$ (we call such a set \emph{$p$-random}), and define the \emph{critical probability} for percolation in the $r$-neighbour model on $[n]^d$ to be  
\[p_c\big( [n]^d, r \big) := \inf\Big\{ p \in (0,1) \, : \, \Pr_p\big( [A] = [n]^d \big) \ge 1/2 \Big\}\,.\]
Foundational work on this problem was done by Aizenman and Lebowitz~\cite{Aizenman88} in 1988, who determined $p_c( [n]^d, 2 )$ up to a constant factor for all fixed $d \ge 2$. However, the problem of determining a sharp threshold for $p_c( [n]^2, 2 )$ remained open until 2003, when Holroyd~\cite{Holroyd03}, in an important breakthrough, proved that
\begin{equation}\label{eq:Holroyd}
p_c\big( [n]^2, 2 \big) = \bigg( \frac{\pi^2}{18} + o(1) \bigg) \frac{1}{\log n}\,.
\end{equation}
Holroyd's proof was particularly significant for its introduction of the so-called `method of hierarchies' (see Section~\ref{sec:hier}), which has played a crucial role in much of the subsequent progress in the area.

Building on Holroyd's proof, Gravner and Holroyd~\cite{Gravner08} and Gravner, Holroyd and Morris~\cite{Gravner12} improved (respectively) the upper and lower bounds obtained in~\cite{Holroyd03}, proving that
\begin{equation}\label{eq:GHM}
\frac{\pi^2}{18\log n} - \frac{C (\log\log n)^3}{(\log n)^{3/2}} \, \le \, p_c\big( [n]^2,2 \big) \, \le \, \frac{\pi^2}{18\log n} - \frac{c}{(\log n)^{3/2}}
\end{equation}
for some constants $C > c > 0$. The upper bound was obtained by considering a larger family of possible growth mechanisms than in~\cite{Holroyd03}, by allowing the shape of the so-called `critical droplet' to vary a little (rather than being exactly square). The lower bound was obtained by repeating Holroyd's proof, but using a more refined notion of hierarchy, and counting these hierarchies more carefully (see Section~\ref{outline:sec} for a more detailed discussion). 

In this paper we will prove a stronger lower bound on $p_c( [n]^2,2 )$, which removes the $(\log\log n)^3$ term, and hence matches the upper bound up to a constant factor in the second term. To be precise, we will prove the following theorem.

\begin{thm}\label{thm:sharp}
There exist constants $C > c > 0$ such that
\[\frac{\pi^2}{18\log n} - \frac{C}{(\log n)^{3/2}} \, \le \, p_c\big( [n]^2,2 \big) \, \le \, \frac{\pi^2}{18\log n} - \frac{c}{(\log n)^{3/2}}\,.\]
\end{thm}

Note that the upper bound follows from the theorem of Gravner and Holroyd~\cite{Gravner08}. We remark that, unlike the result of Gravner, Holroyd and Morris~\cite{Gravner12}, our lower bound does not follow via a relatively minor modification of Holroyd's proof, and we will have to work much harder, and introduce a number of significant technical innovations. In particular, we will require a much finer understanding of the typical growth of a critical droplet, which we obtain using hierarchies that are chosen carefully in order to encode much more information about the growth of the droplet, and extremely precise bounds on the probability that each step of these hierarchies is `satisfied' by the set $A$ of initially infected sites. We remark that we will need to use non-monotone events in our hierarchies, which creates additional technical difficulties in the proof. A more detailed sketch of the proof of Theorem~\ref{thm:sharp} is given in Section~\ref{outline:sec}. 

Before embarking on the proof of Theorem~\ref{thm:sharp}, let us briefly discuss the state of knowledge in higher dimensions, and some exciting recent developments regarding more general models in two dimensions. As noted above, Aizenman and Lebowitz~\cite{Aizenman88} determined $p_c\big( [n]^d, 2 \big)$ up to a constant factor for all $d \ge 2$. However, it took more than 10 years until a corresponding result was proved for all $d \ge r \ge 2$, by Cerf and Cirillo~\cite{Cerf99} (in the case $d = r = 3$) and Cerf and Manzo~\cite{Cerf02}. The sharp threshold for the $r$-neighbour process on $[n]^d$ was finally determined by Balogh, Bollob\'as, Duminil-Copin and Morris~\cite{Balogh09b,Balogh12}, who determined, for each $d \ge r \ge 2$, an explicit constant $\lambda(d,r) > 0$ such that\footnote{Here $\log_{(r)}$ is an $r$-times iterated natural logarithm.} 
\[p_c\big( [n]^d, r \big) = \bigg( \frac{\lambda(d,r) + o(1)}{\log_{(r-1)} n} \bigg)^{d-r+1}\,.\]
More recently, Uzzell~\cite{Uzzell12} extended the upper bound of Gravner and Holroyd~\cite{Gravner08} to arbitrary $d \ge r \ge 2$, providing some hope that a result corresponding to Theorem~\ref{thm:sharp} could eventually be proved in this more general setting.

In a different direction, an extremely general family of monotone cellular automata, called \emph{$\cU$-bootstrap percolation}, was recently introduced by Bollob\'as, Smith and Uzzell~\cite{Bollobas15}, and subsequently was also studied by Balister, Bollob\'as, Przykucki and Smith~\cite{Balister16} and by Bollob\'as, Duminil-Copin, Morris and Smith~\cite{Bollobas14}. The model is as follows: given an arbitrary finite collection $\cU = \{ X_1,\ldots,X_m \}$ of finite subsets of $\Z^2 \setminus \{ \0 \}$, and a set $A \subset \Z_n^2$ of initially infected sites, set $A_0 = A$ and define
\[A_{t+1} = A_t \cup \big\{ v \in \Z_n^2 \,:\, v + X \subset A_t \text{ for some } X \in \cU \big\}\]
for each $t \ge 0$. Confirming a conjecture of Bollob\'as, Smith and Uzzell, the main results of~\cite{Bollobas15} and~\cite{Balister16} together characterize for which of these families the critical probability is polynomial in $n$, for which it is polylogarithmic, and for which it is bounded away from zero (rather surprisingly, there are no other possibilities). In~\cite{Bollobas14} the critical probability was determined up to a constant factor for families with polylogarithmic critical probability. The sharp threshold is only known for a certain class of centrally symmetric families~\cite{Duminil12}, and for two specific non-symmetric models~\cite{Duminil13,Bollobas17}. For more details, and a discussion of the problem in higher dimensions, we refer the interested reader to~\cite{Bollobas14} or~\cite{Morris17}. We also remark that a sharper threshold, along the lines of Theorem~\ref{thm:sharp}, for a so-called `unbalanced' model known as `anisotropic bootstrap percolation' was recently proved by Duminil-Copin, van Enter and Hulshof~\cite{Duminil17c}.

The rest of the paper is organised as follows. In Section~\ref{outline:sec} we give a detailed outline of the proof of Theorem~\ref{thm:sharp}, and in Section~\ref{sec:tools} we recall some basic tools and facts that we will need later, and set up some useful notation and conventions used throughout the paper. In Section~\ref{sec:key} we state (and give an extended sketch of the proof of) our key bounds on the probability that a rectangle is internally filled by $A$ together with a sub-rectangle (the full details of the proof are postponed to the Appendix.) In Section~\ref{sec:hier} we introduce the hierarchies we will use in the proof, prove some standard facts about the family of hierarchies, and describe a partition of this family which plays an important role in the analysis. Finally, in Section~\ref{sec:proof}, we prove Theorem~\ref{thm:sharp}. We finish the paper, in Section~\ref{sec:open}, by mentioning a couple of natural open problems.

\section{An outline of the proof}\label{outline:sec}

The proof of Theorem~\ref{thm:sharp} is very technical, so in this section we will attempt to give the reader an easily-digestible outline of the main ideas behind the proof. The main step will be to bound the probability that a `critical droplet' $R$ (a rectangle with sides of length between $1/p$ and $(1/p) \log(1/p)$) is `internally filled' by the $p$-random set $A$. The claimed lower bound on $p_c( [n]^2,2 )$ will follow easily from this bound via a standard argument (using a lemma due to Aizenman and Lebowitz~\cite{Aizenman88} and the union bound). In order to state this theorem precisely, we will need to introduce a little notation. 

A \emph{rectangle} is a non-empty set $R \subset \Z^2$ of the form $[a,b] \times [c,d]$; we write $\dim(R) = (b-a+1,d-c+1)$ for the \emph{dimensions} of $R$. We say that a rectangle $R$ is \emph{internally filled} by $A$ if $[ A \cap R ] = R$. We also need the function
\begin{equation}\label{def:g}
g(z) := - \log\Big( \beta\big( 1 - e^{-z} \big) \Big)
\end{equation}
where $\beta(u) := \frac{1}{2} \big( u + \sqrt{ u ( 4 - 3u ) } \big)$, which was defined by Holroyd~\cite{Holroyd03}, who also proved that 
\begin{equation}\label{def:lambda}
\int_0^{\infty} g(z)\md z \, = \lambda \/ := \, \frac{\pi^2}{18}\,.
\end{equation}
Finally, set $q := - \log(1 - p)$, and note that $q \ge p$, and that $q \sim p$ as $p \to 0$. (This notation is convenient, because the probability that a set of size $a$ contains no element of the $p$-random set $A$ is $e^{-aq}$. We will assume throughout that $p \to 0$.) We can now state our main bound on the probability that a critical droplet is internally filled.

\begin{thm}\label{thm:droplet}
There exists a constant $C > 0$ such that the following holds. Let $R$ be a rectangle with dimensions $\dim(R) = (a,b)$, and suppose that $a \le b$, and 
\begin{equation}\label{eq:droplet:conditions}
\frac{C}{q} \, \le \, b \, \le \, \frac{1}{2q} \log \frac{1}{q}\,.
\end{equation}
Then
\[\Pr_p\big( [A \cap R] = R \big) \le \exp\left( - \min\left\{ \frac{2\lambda}{q} + \frac{1}{q^{3/4}}, \, ( b - a ) g( aq ) + \frac{2}{q} \int_0^{aq} g(z)\md z - \frac{C}{\sqrt{q}} \right\} \right)\,.\]
\end{thm}

We remark that the first term in the minimum is easily large enough for our purposes, and is only needed for technical reasons; the reader should therefore focus her attention on the second term. Let us write $\lgs(R)$ and $\sh(R)$ for the maximum and minimum (respectively) of the dimensions of $R$. In order to deduce Theorem~\ref{thm:sharp} from Theorem~\ref{thm:droplet}, we will need the following fundamental lemma of Aizenman and Lebowitz~\cite{Aizenman88}. 
 
\begin{AL}
If $[A] = [n]^2$, then for each $1 \le k \le n$ there exists a rectangle $R$ with 
\[k \le \lgs(R) \le 2k\]
that is internally filled by $A$.
\end{AL}

To deduce a lower bound on $p_c( [n]^2,2 )$, we simply apply the Aizenman--Lebowitz lemma with $k = (1/(4q)) \log (1/q)$, and take a union bound over choices of $R$, using Theorem~\ref{thm:droplet} to bound the probability that $R$ is internally filled, and the (straightforward) fact that
\[( b - a ) g( aq ) + \frac{2}{q} \int_0^{aq} g(z) \md z \, \ge \, \frac{2\lambda}{q} - \frac{O(1)}{\sqrt{q}}\]
if $a \le b$ and $b \ge (1/(4q)) \log(1/q)$, see Lemma~\ref{obs:lambda}. 

Our main challenge will therefore be to prove Theorem~\ref{thm:droplet}. As has become standard in the area since their introduction by Holroyd~\cite{Holroyd03}, we will do so using hierarchies; however, our definition will differ in various important ways from that used in~\cite{Holroyd03}, and also from the various notions of hierarchy used in, for example,~\cite{Bollobas14,Bollobas17,Duminil13,Gravner12}.

In order to motivate the definition, let us begin by recalling the hierarchies used by Holroyd~\cite{Holroyd03} to prove~\eqref{eq:Holroyd} and by Gravner, Holroyd and Morris~\cite{Gravner12} to prove~\eqref{eq:GHM}. Roughly speaking, the basic idea of the proof in~\cite{Holroyd03} is that, given an internally filled rectangle $R$, we would like to associate with $R$ a constant-size rooted tree of sub-rectangles that `encodes' the way in which the set $A \cap R$ grows to infect the rest of $R$. The leaves of this tree correspond to small internally filled rectangles (`seeds'), a vertex with two children corresponds to two (not too small) rectangles merging to form a larger rectangle, and a vertex with one child corresponds to a rectangle `growing on its sides' to fill a slightly larger rectangle. Crucially, we would like all of these (increasing) events to occur disjointly, so that we can apply the van der Berg--Kesten inequality (see Section~\ref{sec:correlation}) to bound the probability of their intersection. Since there are few such hierarchies, and each is (roughly speaking) at least as as unlikely as a single `seed' growing to fill $R$, one can deduce a sufficiently strong bound on the probability that $R$ is internally filled.

Gravner, Holroyd and Morris~\cite{Gravner12} required two additional ideas in order to prove the lower bound in~\eqref{eq:GHM}. First, they needed their seeds to be much smaller (of size $q^{-1/2}$, rather than $o(1/q)$), and to grow geometrically (rather than linearly) as a function of their height in the tree. As a result, the number of possible hierarchies became very large (too large to use a naive union bound), and to deal with this issue they partitioned the family of hierarchies according to the number of `big' seeds. We will use refinements of both of these ideas in the proof of Theorem~\ref{thm:droplet}.

In order to prove Theorem~\ref{thm:droplet} we can only afford to lose a factor of $\exp\left( O(q^{-1/2}) \right)$ (in the expected number of `satisfied' hierarchies), and since our hierarchies will typically have height $\Theta(q^{-1/2})$, this means that we can only allow ourselves a constant number of choices at each step, unless we `pay' for extra choices via some unlikely event occurring. Fortunately, this is intuitively possible: the only things that could prevent us from choosing the next rectangle in an almost unique way are: $(a)$ the existence of a `double gap' of consecutive empty rows or columns blocking the growth of the critical droplet, or $(b)$ the merging of two reasonably large internally filled rectangles. Our challenge will be to show that we gain enough from these events to compensate for the extra choices we are forced to make. 

To do so, we will need to encode the existence of double gaps in our hierarchies, which causes two immediate problems: the events cease to be increasing, and cease to occur disjointly. To avoid these issues we only use the fact that the double gaps are empty in a single path through the hierarchy (which we call the `trunk'); outside the trunk we use increasing events defined on the complement of the double gaps. In Section~\ref{sec:key} we will state (and sketch the proofs of) a pair of technical `crossing' lemmas which provide sufficiently strong bounds on the probabilities of these events. We remark that we gain from the existence of these double gaps in two distinct ways: they force us to find either two infected sites close together, or one infected site in a relatively small region, and when the rectangle is very large they are themselves unlikely to exist.  

Bounding the expected number of `satisfied' hierarchies with height $O(1/\sqrt{q})$ will then be relatively straightforward; unfortunately, however, this is not always the case. In Section~\ref{sec:proof} we will have to deal with various other types of hierarchy: those with too many vertices, with too many (or too large) seeds, and those whose growth deviates from the diagonal by a macroscopic amount (see Lemma~\ref{lem:height:or:vertex}). One additional innovation that we will need in order to deal with this last case is Lemma~\ref{lem:pods}, which provides us with two `pods', instead of the single pod required by Holroyd.

\section{Basic facts and definitions}\label{sec:tools}

In this section we will recall a few basic facts about two-neighbour bootstrap percolation on $[n]^2$, state a few simple properties of the function $g(z)$, and introduce some further notation. For convenience, let us fix  (for the rest of the paper) sufficiently large constants $B > 0$ and $C = C(B) > 0$, and a sufficiently small constant $\delta = \delta(B,C) > 0$.

\subsection{Preliminaries}
\label{sec:preliminaries}

To begin, recall the following simple and well-known fact (see, e.g.,~\cite[Problem~34]{Bollobas06}). We write $\phi(R)$ for the semi-perimeter of a rectangle $R$, so $\phi(R) = \lgs(R) + \sh(R)$. 

\begin{lem}\label{lem:choco}
If $[A \cap R] = R$, then $|A \cap R| \ge \ds\frac{\phi(R)}{2}$.
\end{lem}

Now, recall from~\eqref{def:g} the definition of the function $g(z)$. The next lemma, which bounds the probability that a sufficiently small rectangle is internally filled, follows easily from Lemma~\ref{lem:choco} (see, e.g.,~\cite[Lemma~2]{Gravner12}).

\begin{lem}\label{lem:seeds}
There exists $\delta > 0$ such that for any $p > 0$ and any rectangle $R$ with $\dim(R) = (a,b)$, where $a \le b$ and $ap \le \delta$,
\[\Pr_p\big( [A \cap R] = R \big) \le \, 3^{\phi(R)} \exp\Big( - \phi(R) g(aq) \Big)\,.\]
\end{lem}

In order to control the growth of a droplet, we will need to bound various probabilities relating to the existence of double gaps. To be precise, let us say that a rectangle $R = [a,b] \times [c,d]$ has a \emph{vertical double gap} if there exists $j \in [a,b-1]$ such that 
\[A \cap \big( [j,j+1] \times [c,d] \big) = \emptyset\,,\]
and similarly for a \emph{horizontal double gap}. (We will say that $R$ has a double gap if it has a horizontal or vertical double gap.) We will say that $R$ is \emph{crossed from left to right}\footnote{We define similarly the notions of being crossed from right to left, bottom to top, and top to bottom.} if it has no vertical double gap and the rightmost column $\{b\} \times [c,d]$ is \emph{occupied}, that is, has non-empty intersection with $A$. Note that if the column to the left of $R$ is already infected, and $R$ is crossed from left to right, then $R$ will also be infected by the process. The following simple estimates were proved in~\cite[Lemma~8]{Holroyd03}.

\begin{lem}\label{lem:doublegaps}
If $R$ is a rectangle with $\dim(R) = (a,b)$, then
\begin{equation*}
\Pr_p\big( R \textnormal{ has no vertical double gap} \big) \le e^{-(a-1)g(bq)}
\end{equation*}
and
\begin{equation*}
\Pr_p\big( R \textnormal{ is crossed from left to right} \big) \le e^{-ag(bq)}\,.
\end{equation*}
\end{lem}

We remark that the function $g$ is positive, decreasing, convex and differentiable on $(0,\infty)$, that $g(z) \sim e^{-2z}$ as $z \to \infty$, that 
\begin{equation}\label{obs:g3}
- \frac{1}{2} \log z - \sqrt{z} \, \le \, g(z) \, \le \, - \, \frac{1}{2} \log z + z
\end{equation}
for all sufficiently small $z > 0$ (see~\cite[Observation~4]{Gravner12}), that
\begin{equation}\label{obs:g5}
e^{2g(z)}\le\frac{C}{z}
\end{equation}
for  all $0 < z \le 3e^{2B}$ (see~\cite[Observation~10]{Gravner12}), and that 
\begin{equation}\label{obs:g6}
- g'(z) \le \left\{ 
\begin{array}{cl}
B/z & \textup{ if } z \le B \smallskip \\
3 e^{-2z} & \textup{ if } z \ge B/2 
\end{array} \right.
\end{equation}
since $B$ and $C = C(B)$ were chosen sufficiently large.

\subsection{Analytic estimates}
\label{sec:analytic}

We will use the following definition from~\cite{Holroyd03} to control the growth of a droplet. 

\begin{defn}\label{def:U}\leavevmode
For each $\mathbf a \le \mathbf b \in \mathbb{R}_+^2$, define 
\begin{equation}\label{def:W}
W(\mathbf a,\mathbf b) =\inf_{\gamma \, : \, \mathbf a \rightarrow \mathbf b} \,\int_\gamma \, \big( g(y) \md x + g(x) \md y \big)\,,
\end{equation}
where the infimum is taken over all piecewise linear increasing paths from $\mathbf a$ to $\mathbf b$ in $\mathbb R^2$.

Now, for any pair $S \subset R$ of rectangles, define
\begin{equation}\label{eq:def:W}
U(S,R) := W\big( q \dim(S), q \dim(R) \big)\,.
\end{equation}
\end{defn}

One of the key lemmas from~\cite{Holroyd03} states that the integral in~\eqref{def:W} is minimized when the path $\gamma$ is chosen as close to the diagonal as possible. We will use the following immediate consequence of this fact.

\begin{lem}[Lemma~16 of~\cite{Holroyd03}]\label{lem:diagonal}
Let $S \subset R$ be rectangles with $\lgs(S) \le \sh(R)$. Then 
\[\frac{U(S,R)}{q} \, = \, (d - c) g(dq) + \frac{2}{q} \int^{aq}_{dq} g(z) \md z + (b - a) g(aq)\,,\]
where $a = \sh(R)$, $b = \lgs(R)$, $c = \sh(S)$ and $d = \lgs(S)$. 
\end{lem}

When $\lgs(S) > \sh(R)$ we will use the following easy consequence of the fact that  $g(z)$ is decreasing (it also follows immediately from~\cite[Lemma 16]{Holroyd03}).

\begin{lem}\label{lem:offdiagonal}
Let $S \subset R$ be rectangles with $\lgs(S) \ge \sh(R)$. Then 
\[\frac{U(S,R)}{q} \, \ge \, (b - d) g(aq)\,,\]
where $a = \sh(R)$, $b = \lgs(R)$ and $d = \lgs(S)$.
\end{lem}

We will also need the following straightforward bound from~\cite{Gravner12}.

\begin{lem}[Lemma~14 of~\cite{Gravner12}]\label{lem:Ulowerbound}
Let $S \subset R$ be rectangles, with $\lgs(S) \le \sh(R)$. Then
\[\frac{U(S,R)}{q} \, \ge \, \frac{2}{q} \int_0^{aq} g(z) \md z + (b - a) g(aq) - \frac{\phi(S)}{2} \log\left( 1 + \frac{1}{\phi(S)q} \right) - O\big( \phi(S) \big)\,,\]
where $a = \sh(R)$ and $b = \lgs(R)$.
\end{lem}

In order to transition between $U(S,R)$ and the bounds proved in Section~\ref{sec:key}, below, we will also need the following simple upper bound. If $S$ and $R$ are rectangles with dimensions $\dim(R) = (a,b)$ and $\dim S=(a-s,b-t)$, then set 
\begin{equation}\label{def:QRS}
Q(S,R) := s g\big( (b-t)q \big) + t g\big( (a - s ) q \big)\,,
\end{equation}
The following lemma follows immediately from the fact that $g(z)$ is decreasing.  

\begin{lem}[Proposition~13 of~\cite{Holroyd03}]\label{lem:U:upbound}
Let $S \subset R$ be rectangles. Then
\[\frac{U(S,R)}{q} \, \le \, Q(S,R)\,.\]
\end{lem}

We will also need a couple of additional technical lemmas, each of which follows easily from simple properties of the function $g$. The first is a variant of~\cite[Observation~19]{Gravner12}, with slightly weaker assumptions and conclusion.

\begin{lem}\label{obs:lambda}
If $a \le b$ and $b \ge (1/(4q)) \log(1/q)$, then
\[\frac{2}{q} \int^{aq}_{0} g(z) \md z + (b - a) g(aq) \, \ge \, \frac{2\lambda}{q} - \frac{4e^{4}}{\sqrt{q}}\,.\]
\end{lem}

\begin{proof}
Recall that $B > 0$ is a sufficiently large constant, and note that if $a \le B/q$ then
\[(b - a) g(aq) \, \ge \, \frac{g(B)}{5q} \log \frac{1}{q} \, \ge \, \frac{2\lambda}{q}\,,\]
since $g$ is decreasing and $q \to 0$. Let us therefore assume that $a \ge B/q$, and observe that therefore 
\begin{equation}\label{obs:g7}
\int_{aq}^\infty g(z) \md z \, \le \, g(aq)\,,
\end{equation}
since $g(z) \sim e^{-2z}$ as $z \to \infty$, and hence, recalling the definition~\eqref{def:lambda} of $\lambda$, 
\[\frac{2}{q} \int^{aq}_{0} g(z) \md z + (b - a) g(aq) \, \ge \, \frac{2}{q} \int_{0}^\infty g(z) \md z \, = \, \frac{2\lambda}{q}\]
if $b - a \ge 2/q$. Finally, if $b - a \le 2/q$, then $a \ge (1/(4q)) \log(1/q) - 2/q$, and so
\[\frac{2}{q} \int^{aq}_{0} g(z) \md z \, \ge \, \frac{2\lambda}{q} - \frac{2g(aq)}{q} \, \ge \, \frac{2\lambda}{q} - \frac{4e^{4}}{\sqrt{q}}\,,\]
by~\eqref{obs:g7}, and since $g(z) \le 2e^{-2z}$ if $z \ge B$.
\end{proof}

The next lemma quantifies how much harder it is for a droplet to grow far from the diagonal. To state it, we need to introduce a further large constant $L_1 = L_1(B,C,\delta) > 0$.

\begin{lem}\label{lem:leaving:the:diagonal}
If $L_1 a \le b \le B/q$, then
\[\frac{2}{q}\int_{aq}^{bq} g(z) \md z \, \le \, (b-a)\big( g(aq) + g(bq) \big) - 4Cb\,.\]
\end{lem}

\begin{proof}
We claim first that $- B / z \le g'(z) \le - \delta / z$ for every $0 < z < B$. Indeed, this follows since $g'(z) \sim - 1/(2z)$ as $z \to 0$ and $g'(z) \sim - 2 e^{-2z}$ as $z \to \infty$, and since $B$ was chosen sufficiently large, and $\delta = \delta(B)$ sufficiently small. Now, integrating by parts, we obtain
\[\frac{2}{q}\int_{aq}^{bq} g(z) \md z \, \le \,  2\big( b g(bq) - a g(aq) \big) + 2B (b - a)\,.\]
It follows that
\[\frac{2}{q}\int_{aq}^{bq} g(z) \md z - (b-a)\big( g(aq) + g(bq) \big) \le (a+b)\big( g(bq) - g(aq) + 2B \big)\,.\]
Now, since $g'(z) \le - \delta / z$ for every $z < B$, and $b/a \ge L_1$, we have
\[g(aq) - g(bq) \, = \, - \int_{aq}^{bq} g'(z) \md z \, \ge \, \delta \log L_1 \, \ge \, 5C\,,\]
and so the claimed bound follows.
\end{proof}

We will also need some larger constants, which we will denote by $L_2, L_3, \ldots$, where each $L_i$ is chosen to be sufficiently large depending on $B$, $C$, $\delta$, and $L_1,\ldots,L_{i-1}$. We will use $O(\cdot)$ to denote the existence of an absolute constant, that is, a constant that does \emph{not} depend on any of the aforementioned ones. 

\subsection{Correlation inequalities}\label{sec:correlation}

To finish this section, we will state the fundamental inequalities of van den Berg and Kesten~\cite{BK85} and Reimer~\cite{Reimer00}, which we will use in Section~\ref{sec:hier} to bound the probability that a hierarchy is `satisfied' by $A$, the $p$-random set of infected sites, see Definition~\ref{def:hier:sat} and Lemma~\ref{lem:basic:bound}. 

In our setting, an \emph{event} $\cE$ is simply a family of subsets of $[n]^2$, and the event $\cE$ is said to \emph{occur} if $A \in \cE$. Two events $\cE$ and $\cF$ are said to \emph{occur disjointly} for $A$ if there exist disjoint sets $X,Y \subset [n]^2$ depending on $A$ such that $S \in \cE$ for any $S$ such that $S \cap X = A \cap X$, and $T \in \cF$ for any $T$ such that $T \cap Y = A \cap Y$. We write $\cE \circ \cF$ for the event that $\cE$ and $\cF$ occur disjointly. 

Recall that we write $\Pr_p$ to indicate that $A$ is a $p$-random subset of $[n]^2$. The following fundamental lemma was proved in 1985 by van den Berg and Kesten~\cite{BK85}. 

\begin{vBKlemma}
Let $\cE$ and $\cF$ be any two increasing events and let $p \in (0,1)$. Then
\[\Pr_p(\cE \circ \cF) \, \le \, \Pr_p(\cE)\,\Pr_p(\cF)\,.\]
\end{vBKlemma}

The authors of~\cite{BK85} also conjectured that their inequality holds in the following more general setting; this was proved 15 years later by Reimer~\cite{Reimer00}.

\begin{Reimer}
Let $\cE$ and $\cF$ be any two events and let $p \in (0,1)$. Then
\[\Pr_p(\cE \circ \cF) \, \le \, \Pr_p(\cE)\,\Pr_p(\cF)\,.\]
\end{Reimer}

We remark that the events which we will need to consider will not all be increasing (or decreasing); however, they will all be obtained by intersecting an increasing event with a decreasing event. For such events the conclusion of Reimer's theorem was proved earlier, by van den Berg and Fiebig~\cite{Berg87}, and the proof is significantly simpler.

\section{The key lemmas}\label{sec:key}

In this section we will state our key bounds (Lemmas~\ref{lem:key:small} and~\ref{lem:key:big}, below) on the probability that a rectangle $R$ is internally filled by the union of $A$ (chosen according to $\Pr_p$) and a rectangle $S \subset R$. In order to simplify the statement somewhat, we will begin by giving some rather technical definitions, which are illustrated in Figure~\ref{fig:frame}.

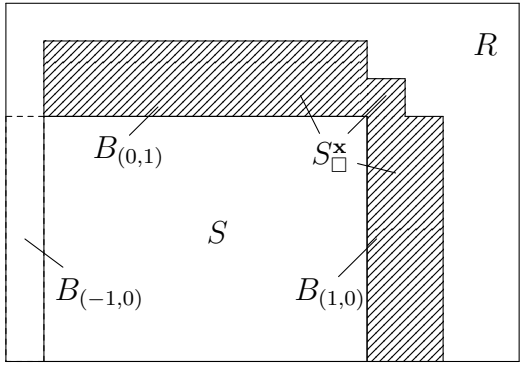
\begin{figure}[h]
\floatbox[{\capbeside\thisfloatsetup{capbesideposition={right,top}}}]{figure}[\FBwidth]
{\begin{tikzpicture}[line cap=round,line join=round,>=triangle 45,x=0.25cm,y=0.25cm]
\clip(-20.5,-13.5) rectangle (7.5,6.5);
\fill[fill=black,pattern=north east lines] (-18,0) -- (-18,4) -- (-1,4) -- (-1,0) -- cycle;
\fill[fill=black,pattern=north east lines] (-1,2) -- (-1,0) -- (1,0) -- (1,2) -- cycle;
\fill[fill=black,pattern=north east lines] (3,0) -- (3,-13) -- (-1,-13) -- (-1,0) -- cycle;
\draw (-20,-13)-- (7,-13);
\draw (7,-13)-- (7,6);
\draw (7,6)-- (-20,6);
\draw (-20,6)-- (-20,-13);
\draw (-18,-13)-- (-1,-13);
\draw (-1,-13)-- (-1,0);
\draw (-1,0)-- (-18,0);
\draw (-18,0)-- (-18,-13);
\draw [dash pattern=on 2pt off 2pt] (-18,0)-- (-20,0);
\draw [dash pattern=on 2pt off 2pt] (-20,0)-- (-20,-13);
\draw [dash pattern=on 2pt off 2pt] (-20,-13)-- (-18,-13);
\draw [dash pattern=on 2pt off 2pt] (-18,-13)-- (-18,0);
\draw (-18,0)-- (-18,4);
\draw (-18,4)-- (-1,4);
\draw (-1,0)-- (-18,0);
\draw (-1,2)-- (-1,4);
\draw (1,0)-- (1,2);
\draw (1,2)-- (-1,2);
\draw (3,0)-- (3,-13);
\draw (3,-13)-- (-1,-13);
\draw (-1,-13)-- (-1,0);
\draw (1,0)-- (3,0);
\draw (-10,-5) node[anchor=north west] {$S$};
\draw (4,5) node[anchor=north west] {$R$};
\draw (-18,-8) node[anchor=north west] {$B_{(-1,0)}$};
\draw (-17,-8)-- (-19,-6.5);
\draw (-16,-0.5) node[anchor=north west] {$B_{(0,1)}$};
\draw (-13.5,-1)-- (-12,0.5);
\draw (-5.4,-8) node[anchor=north west] {$B_{(1,0)}$};
\draw (-2.5,-8.5)-- (-0.5,-6.5);
\draw (-4.5,-0.8) node[anchor=north west] {$S_{\square}^\x$};
\draw (-3.5,-0.9)-- (-4.5,1.1);
\draw (-2,-0.9)-- (0,1.1);
\draw (-1.5,-2.5)-- (0.5,-3);
\end{tikzpicture}}
{\caption{An example of a frame. The non-empty buffers and frame $S_{\square}^\x$ (hatched) of $S$ in $R$ with $x_{1,0}=x_{0,1}=1$, $x_{-1,0}=x_{0,-1}=0$. Note that the buffers may have width $1$.}
\label{fig:frame}}
\end{figure}

Throughout this section, we will assume that $S \subset R$ are rectangles with $\sh(S) \ge 2$. 

\begin{defn}\label{def:buffers}
The \emph{buffers} of $S$ in $R$ are the sets
\[B_{(i,j)}(S,R) := \big\{ v \in R \setminus S \,:\, v - (2i,2j) \in S \big\}\,,\]
where $(i,j) \in \cI := \big\{ (1,0), (0,1), (-1,0), (0,-1) \big\}$. We call the elements of $\cI$ \emph{directions}, define
\[Z(S,R) := \big\{ \d \in \cI \,:\, B_\d (S,R) \ne \emptyset \big\}\]
to be the collection of non-empty buffers of $S$ in $R$, and set $z(S,R) = |Z(S,R)|$. 
\end{defn}

Given $\x = ( x_\d )_{\d \in \cI} \in \{0,1\}^\cI$, define the \emph{$\x$-buffer} of $S$ in $R$ to be
\[B^\x(S,R) := \bigcup_{\d \,\in\, \cI \,:\, x_\d \,=\, 1} B_\d(S,R)\,,\]
and the \emph{$\x$-frame} of $S$ in $R$ to be the set
\[S_{\square}^\x \, := \, B^\x(S,R) \cup \big\{ v \in R \setminus S  \,:\, |N(v) \cap B^\x(S,R)| \ge 2 \big\}\,.\]
and set $S_\blacksquare^\x := S \cup S_{\square}^\x$. Thus $\x$ encodes the inclusion in $S_{\square}^\x$ of some of the (non-empty) buffers, and also the `corner' site in between two selected buffers. We will write $x$ and $y$ for the number of non-empty horizontal and vertical buffers included in $B^\x(S,R)$, i.e., 
\begin{equation}\label{def:xy}
x := x'_{(1,0)}  + x'_{(-1,0)} \qquad \text{and} \qquad y := x'_{(0,1)} + x'_{(0,-1)}\,,
\end{equation}
where $\x'=\x \cdot \mathbbm{1}_{Z(S,R)}$ (i.e., $x_\d' := x_\d$ if $\d \in Z(S,R)$, and $x_\d' := 0$ otherwise).

We are now ready to define our key technical events, which will appear in our hierarchies (see Section~\ref{sec:hier}, below), and are designed to be sufficiently unlikely, and to occur disjointly.

\begin{defn}
Let the rectangles $S \subset R$, and $\x \in \{0,1\}^\cI$, be as described above.
\begin{enumerate}[label=(\textbf{\alph*})]
\item $D_1^{\x}(S,R)$ denotes the event that
\[\big[ S \cup \big( A \cap R \setminus \Sbar \big) \big] = R\,.\]
\item $D_2^{\x}(S,R)$ denotes the event
\[D_1^{\x}(S,R) \cap \big\{ A \cap \Sfr = \emptyset \big\}\,.\]
\end{enumerate}
\end{defn}

The main results of this section are the following two lemmas, which provide us with close to best possible upper bounds on the probabilities of the events $D_1^{\x}(S,R)$ and $D_2^{\x}(S,R)$. The statements are designed to facilitate a proof by induction. 

\begin{lem}\label{lem:key:small}
Let $S \subset R$ be rectangles with $\dim(R) = (a,b)$ and $\dim(S) = (a-s,b-t)$, let $\x \in \{0,1\}^\cI$ and set $z = z(S,R)$. If
\begin{equation}\label{eq:key:small:R}
L_1 \le \sh(R) \le \frac{B}{q} \qquad \text{and} \qquad \lgs(R) \le \frac{3e^{2B}}{q}\,,
\end{equation}
and $s,t \le 4\delta \sqrt{\sh(R)}$, then
\[\Pr_p\big( D_1^{\x}(S,R) \big) \le C^{z} \left( \frac{C}{\sqrt{a}} \right)^y \left( \frac{C}{\sqrt{b}} \right)^x \exp\big( - s g(bq) - t g(aq) \big)\,.\]
\end{lem}

\begin{lem}\label{lem:key:big}
Let $S \subset R$ be rectangles with $\dim(R) = (a,b)$ and $\dim(S) = (a-s,b-t)$, let $\x \in \{0,1\}^\cI$, and set $z = z(S,R)$. If
\begin{equation}\label{eq:key:big:R}
\sh(R) > \frac{B}{q} \qquad \text{and} \qquad \lgs(R) \le \frac{1}{2q} \log \frac{1}{q}
\end{equation}
and $s,t \le \ds\frac{4\delta}{\sqrt{q}} \cdot \exp\big( \sh(R) \cdot q \big)$, then
\[\Pr_p\big( D_2^{\x}(S,R) \big) \le \left( C e^{\sh(R) q} \right)^z \left( C \sqrt{q} e^{-aq} \right)^y \left( C\sqrt{q} e^{-bq} \right)^x \exp\big( - s g(bq) - t g(aq) \big)\,.\]
\end{lem}

It will be convenient later when applying these lemmas to combine them as follows. First, the following function encodes the upper bounds on $s$ and $t$:
\begin{equation}\label{def:f}
f(R) := \left\{ 
\begin{array}{cl}
\delta\sqrt{\sh(R)} & \text{ if } \, \sh(R) \le \ds\frac{B}{q}, \smallskip\\
\ds\frac{\delta}{\sqrt{q}}\exp\big( \sh(R) \cdot q \big) & \text{ otherwise.}
\end{array}
\right.
\end{equation}
Let us say that a rectangle $R$ is \emph{$1$-critical} if it satisfies the bounds in~\eqref{eq:key:small:R}, 
and \emph{$2$-critical} if it satisfies the bounds in~\eqref{eq:key:big:R}. Recall that if $S$ and $R$ have dimensions $\dim(R) = (a,b)$ and $\dim S=(a-s,b-t)$, then 
\[Q(S,R) = s g\big( (b-t)q \big) + t g\big( (a - s ) q \big)\,,\]
and if $\x \in \{0,1\}^\cI$ then write $\| \x \| := x + y = \sum_{\d \in Z(S,R)} x_\d$. The following corollary is an almost immediate consequence of Lemmas~\ref{lem:key:small} and~\ref{lem:key:big}. 

\begin{cor}\label{cor:key}
Let $S \subset R$ be rectangles with $\dim(R) = (a,b)$ and $\dim(S) = (a-s,b-t)$, let $\x \in \{0,1\}^\cI$. Let $j \in \{1,2\}$, and suppose that $R$ is $j$-critical. If $s,t \le 4 f(R)$, then
\begin{equation}\label{eq:cor:key}
\Pr_p\big( D_j^{\x}(S,R) \big) \le C^9 \left(\frac{\delta}{f(R)}\right)^{\|\x\|} \exp\big( - Q(S,R) + 4 \phi(R) q \big)\,.
\end{equation}
\end{cor}  

\begin{proof}
The claimed inequality follows from those given by Lemmas~\ref{lem:key:small} and~\ref{lem:key:big} using the bounds on $g'(z)$ given in~\eqref{obs:g6}, and noting that $x + y + z \le 8$ and $z \le 4$. To spell out the details, recall from~\eqref{obs:g6} that $g'(z) \ge - B/z$ if $z \le B$ and $g'(z) \ge - 3 e^{-2z}$ if $z \ge B/2$, and note that $g'(z)$ is increasing and that $f(R) \le \delta/q$. It follows that
\begin{align*}
\exp\big( - s g(bq) - t g(aq) \big) & \, \le \, \exp\Big( - Q(S,R) - 2stq \cdot g'\big( \big( \sh(R) - 4f(R) \big) \cdot q \big) \Big)\\
& \, \le \, \exp\big( - Q(S,R) + \delta \big)
\end{align*}
since $s,t \le 4 f(R)$ and $\delta = \delta(B)$ is sufficiently small. 
\end{proof}

Since the proofs of Lemmas~\ref{lem:key:small} and~\ref{lem:key:big} involve a significant amount of quite technical (and not especially illuminating) case analysis, we will give here only a sketch, and postpone the full details to the Appendix.
% of the arXiv version of this paper~\cite{HMarxiv}.

\begin{proof}[Sketch of the proof of Lemma \ref{lem:key:small}]
Let $R$ be a $1$-critical rectangle with dimensions $\dim(R) = (a,b)$, and for each $x,y,z$ and each $s,t \le 4f(R)$, set
\[F^{x,y,z}(s,t) := C^{z} \left( \frac{C}{\sqrt{a}} \right)^y \left( \frac{C}{\sqrt{b}} \right)^x \exp\big( - s g(bq) - t g(aq) \big)\,.\]
We will prove, by induction on the pair $(s+t,-(x+y))$, that 
\begin{equation}\label{ih:key:small}
\Pr_p\big( D_1^{\x}(S,R) \big) \le F^{x,y,z}(s,t)
\end{equation}
for every $0 \le s,t \le 4 f(R)$ and $\x \in \{0,1\}^\cI$, and every $S \subset R$ with $\dim(S) = (a-s,b-t)$, where $x$ and $y$ are as defined in~\eqref{def:xy}, and $z = z(S,R)$.

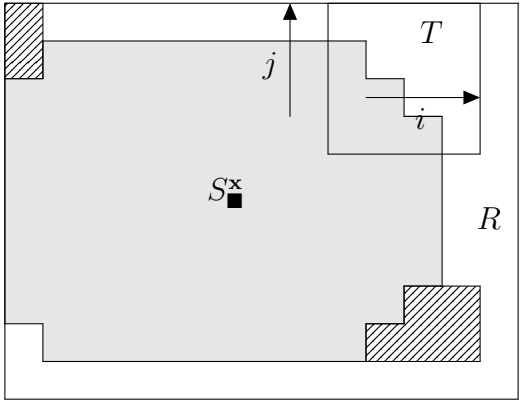
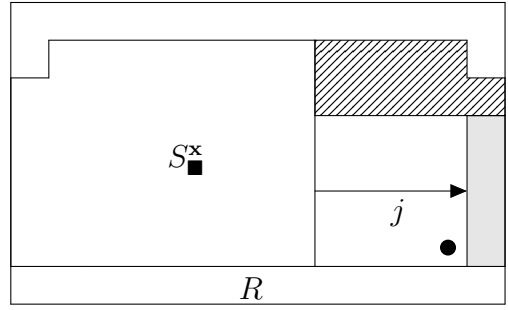
\begin{figure}
\centering
\begin{subfigure}{0.47\textwidth}
\begin{tikzpicture}[line cap=round,line join=round,>=triangle 45,x=0.25cm,y=0.25cm]
\clip(-20.5,-15.5) rectangle (7.5,6.5);
\fill[fill=black,fill opacity=0.1] (-18,4) -- (-18,2) -- (-20,2) -- (-20,-11) -- (-18,-11) -- (-18,-13) -- (-1,-13) -- (-1,-11) -- (1,-11) -- (1,-9) -- (3,-9) -- (3,0) -- (1,0) -- (1,2) -- (-1,2) -- (-1,4) -- cycle;
\fill[fill=black,pattern=north east lines] (-1,-11) -- (-1,-13) -- (5,-13) -- (5,-9) -- (1,-9) -- (1,-11) -- cycle;
\fill[fill=black,pattern=north east lines] (-20,2) -- (-18,2) -- (-18,6) -- (-20,6) -- cycle;
\draw (-20,-15)-- (7,-15);
\draw (7,-15)-- (7,6);
\draw (7.03,6)-- (-20,6);
\draw (-20,6)-- (-20,-15);
\draw (-9.93,-2.71) node[anchor=north west] {$\Sbar$};
\draw (4.25,-4.24) node[anchor=north west] {$R$};
\draw (-3,6)-- (-3,-2);
\draw (-3,-2)-- (5,-2);
\draw (5,-2)-- (5,6);
\draw (5,6)-- (-3,6);
\draw (1.27,5.64) node[anchor=north west] {$T$};
\draw (1,1) node[anchor=north west] {$i$};
\draw (-7,4) node[anchor=north west] {$j$};
\draw [->] (-1,1) -- (5,1);
\draw [->] (-5,0) -- (-5,6);
\draw (1,0)-- (3,0);
\draw (-18,4)-- (-18,2);
\draw (-18,2)-- (-20,2);
\draw (-20,2)-- (-20,-11);
\draw (-19.99,-11)-- (-18,-11);
\draw (-18,-11)-- (-18,-13);
\draw (-18,-13)-- (-1,-13);
\draw (-1,-13)-- (-1,-11);
\draw (-1,-11)-- (1,-11);
\draw (1,-11)-- (1,-9);
\draw (1,-9)-- (3,-9);
\draw (3,-9)-- (3,0);
\draw (3,0)-- (1,0);
\draw (1,0)-- (1,2);
\draw (1,2)-- (-1,2);
\draw (-1,2)-- (-1,4);
\draw (-1,4)-- (-18,4);
\draw (-1,-11)-- (-1,-13);
\draw (-1,-13)-- (5,-13);
\draw (5,-13)-- (5,-9);
\draw (5,-9)-- (1,-9);
\draw (1,-9)-- (1,-11);
\draw (1,-11)-- (-1,-11);
\draw (-20,2)-- (-18,2);
\draw (-18,2)-- (-18,6);
\draw (-18,6)-- (-20,6);
\draw (-20,6)-- (-20,2);
\end{tikzpicture}
\subcaption{Case 1: the rectangle $T$ is internally filled outside the shaded $\Sbar$ and allows $S$ to grow $i$ to the right and $j$ upwards.}
\label{fig:corner}
\end{subfigure}
\quad
\begin{subfigure}{0.47\textwidth}
\begin{tikzpicture}[line cap=round,line join=round,>=triangle 45,x=0.25cm,y=0.25cm]
\clip(-20.5,-10.5) rectangle (6.5,6.5);
\fill[fill=black,fill opacity=0.1] (4,0) -- (6,0) -- (6,-8) -- (4,-8) -- cycle;
\fill[fill=black,pattern=north east lines] (-4,4) -- (-4,0) -- (6,0) -- (6,2) -- (4,2) -- (4,4) -- cycle;
\draw (-20,-10)-- (6,-10);
\draw (6,-10)-- (6,6);
\draw (6,6)-- (-20,6);
\draw (-20,6)-- (-20,-10);
\draw (-16,-8)-- (-4,-8);
\draw (-4,-8)-- (-4,0);
\draw [->] (-4,-4) -- (4,-4);
\draw (-4,4)-- (-4,0);
\draw (-16,-8)-- (-20,-8);
\draw (-20,-8)-- (-20,2);
\draw (-20,2)-- (-18,2);
\draw (-18,2)-- (-18,4);
\draw (-18,4)-- (-4,4);
\draw (4,0)-- (6,0);
\draw (6,0)-- (6,-8);
\draw (6,-8)-- (4,-8);
\draw (4,-8)-- (4,0);
\draw (-4,4)-- (-4,0);
\draw (-4,0)-- (6,0);
\draw (6,0)-- (6,2);
\draw (6,2)-- (4,2);
\draw (4,2)-- (4,4);
\draw (4,4)-- (-4,4);
\draw (-4,-8)-- (4,-8);
\draw (-12.34,-0.99) node[anchor=north west] {$\Sbar$};
\draw (-0.58,-3.87) node[anchor=north west] {$j$};
\draw (-8.59,-8) node[anchor=north west] {$R$};
\begin{footnotesize}
\end{footnotesize}
\begin{scriptsize}
\fill [color=black] (3,-7) circle (3pt);
\end{scriptsize}
\end{tikzpicture}
\subcaption{Case 2: $S$ grows $j$ to the right until it reaches a double gap (shaded). The last column before that is necessarily occupied. In both figures, the hatched region is assumed (in the sketch proof) to be unoccupied.}
\label{fig:side}
\end{subfigure}
\caption{Two possible growth mechanisms.}
\end{figure}

The base of the induction is the case $\min\{ s, t \} = 0$. Without loss of generality suppose that $t = 0$, and note that this implies that $x = y = 0$, since otherwise $\Pr_p\big( D_1^{\x}(S,R) \big) = 0$. It follows that $R \setminus S$ consists of two rectangles (one of which may be empty), one of which is crossed from left to right, and the other of which is crossed from right to left. By Lemma \ref{lem:doublegaps}, it follows that
\[\Pr_p\big( D_1^{\x}(S,R) \big) \le \exp\big( - s g(bq) \big) \le F^{0,0,z}(s,0)\,,\]
as required. We remark that, since $\sh(R) \ge L_1$, the function $F^{x,y,z}(s,t)$ is increasing in $z$ and decreasing in $x$, $y$, $s$ and $t$. 

For the induction step, fix $\x \in \{0,1\}^\cI$ and $S \subset R$ with $\dim(S) = (a-s,b-t)$, and assume that~\eqref{ih:key:small} holds for all smaller values of the pair $(s+t,-(x+y))$ in lexicographical order. We partition into cases, depending on whether or not $z=x+y$. 

\pagebreak
\noindent \textbf{Case 1:} $z = x + y$, i.e., all of the non-empty buffers are included in $\Sfr$.
\medskip

The key observation in this case is that if the event $D_1^{\x}(S,R)$ holds, then there exists a rectangle $T \subset R$ such that 
\[\big[ A \cap T \setminus \Sbar \big] = T \qquad  \text{and} \qquad T \cap \Sbar \ne \emptyset\]
(see Figure \ref{fig:corner}). For simplicity, we will assume that $\phi(T) \le 36 f(R)$ (the other case is dealt with in the Appendix), which in particular implies that $\phi(T) \le \delta \cdot \sh(S)$.

We will sum over choices of $T$ the probability that 
\begin{equation}\label{eq:key:small:indep:events}
\big[ A \cap T \setminus \Sbar \big] = T \qquad \text{and} \qquad \big[ S \cup T \cup \big( A \cap R \setminus \Sbar \big) \big] = R\,.
\end{equation}
Note that these two events depend on disjoint sets of infected sites, and are therefore independent. To bound the probabilities of these events, we will partition according to $k := \phi(T)$, and the dimensions of $[S \cup T]$,
\[\dim\big( [S \cup T] \big) = (a - s + i, b - t + j)\,.\]
Note that $4 \le i + j \le k$, and therefore, by Lemma~\ref{lem:choco}, we have 
\[| A \cap T \setminus \Sbar | \, \ge \, \frac{k}{2} \, \ge \, \frac{i + j}{2}\,.\]
Note also that, given $i$, $j$ and $k$, we have at most $4k$ choices for the rectangle $T$ (at most $k$ per corner of $S$). Therefore, given $i$, $j$ and $k$, the expected number of rectangles $T$ satisfying the first event in~\eqref{eq:key:small:indep:events} is at most
\[4k {k^2 \choose \lceil k/2 \rceil} p^{\lceil k/2 \rceil} \le (24kp)^{k/2}\,.\]
To bound the probability of the second event in~\eqref{eq:key:small:indep:events}, we use the induction hypothesis. To do so, however, we need to split into cases according to whether or not the buffers of $[S \cup T]$ that are not adjacent to $T$ contain any elements of $A$. In this sketch we will assume that they do not; for the full details see the Appendix. 

By the induction hypothesis (under the assumption that no additional infections are found in the buffers), it follows that\footnote{Note that we used here the bound $z([S\cup T],R) \le z(S,R)$, and the fact that $F$ is increasing in $z$.}
\[\Pr_p \Big( \big[ S \cup T \cup \big( A \cap R \setminus \Sbar \big) \big] = R \Big) \le F^{x-1,y-1,z}(s-i,t-j)\,,\]
and hence, by the argument above, it will suffice (in this case) to show that
\begin{equation}\label{eq:key:small:first:need}
\sum_{i + j \ge 4} \sum_{k = i + j}^{36 f(R)} (24kp)^{k/2} \cdot F^{x-1,y-1,z}(s-i,t-j) \ll F^{x,y,z}(s,t)\,.
\end{equation}
To see this, note first that $24kp \le \delta$, since $k \le 36 f(R)$, and that we may therefore assume that $k = i + j$. Now, observe that
\begin{equation}\label{eq:key:small:first:step1}
\frac{F^{x-1,y-1,z}(s - i,t - j)}{F^{x,y,z}(s,t)} = \frac{\sqrt{ab}}{C^{2}} \exp\big( i g(bq) + j g(aq) \big) \le \frac{\sqrt{ab}}{C^{2}}\left(\frac{C}{bq}\right)^{i/2}\left(\frac{C}{aq}\right)^{j/2}
\end{equation}
by~\eqref{obs:g5}, since $\lgs(R) \le 3e^{2B} / q$. Since $i + j = k$ and $i,j \ge 1$, and recalling that $p \le q$, we have
\[\sum_{k = 4}^{36 f(R)} \sum_{i + j = k} (24kp)^{k/2} \cdot \frac{\sqrt{ab}}{C^{2}}\left(\frac{C}{bq}\right)^{i/2}\left(\frac{C}{aq}\right)^{j/2} \le \sum_{k = 4}^{36 f(R)} \frac{k \cdot (C^2k)^{k/2}}{\min\{a,b\}^{(k-2)/2}} \,\le\, \frac{C^5}{\sh(R)}\,,\]
since $\sh(R) \ge L_1$. Combining this with~\eqref{eq:key:small:first:step1}, we obtain~\eqref{eq:key:small:first:need}, as claimed.

\bigskip
\noindent \textbf{Case 2:} $z > x + y$, i.e., some non-empty buffer is not included in $\Sfr$.
\medskip

Without loss of generality, let $B_{(1,0)}(S,R)$ be a non-empty buffer that is not included in $\Sfr$, so $x_{(1,0)} = 0$. The idea is to `grow' $S$ to the right until we find a double gap, or reach the right-hand side of $R$, thus leading either to an increase in $x+y$, or a decrease in $s+t$. One significant complication is that before reaching a double gap we might find an infected site in one of the other buffers, which are growing with $S$ (see Figure~\ref{fig:side}). In this sketch we will assume that this does not occur, and also that we do not reach the right-hand side of $R$; the other cases are dealt with in the Appendix. 

Let $j$ be the distance to the first double gap to the right of $S$, that is
\[j := \min\big\{ i \ge 0 \,:\, A \cap R \cap \big( S + (i+2,0) \big) \setminus \big( S + (i,0) \big) = \emptyset \big\}\,,\]
and denote by $\hat{S} := \bigcup_{i = 0}^j \big( S + (i,0) \big)$ the rectangle formed by the growth of $S$ to the right, until it reaches that double gap. As noted above, we will assume in this sketch that 
\[B_{(1,0)}(\hat{S},R) \ne \emptyset  \qquad \text{and} \qquad A \cap \hat{S}_{\square}^{\hat{\x}} \setminus S_{\square}^\x = \emptyset\,.\]
where $\hat{\x} := \x + \mathbbm{1}_{(1,0)}$ (i.e., $\hat{x}_{(1,0)} = 1$ and $\hat{x}_\d = x_\d$ for each $(1,0) \ne \d \in \cI$). In other words, we found a double gap before reaching the right-hand side of $R$, and no new infected site was found along the way in any of the buffers. We will sum over choices of $j$ the probability that 
\begin{equation}\label{eq:key:small:indep:events:again}
\big[ S \cup \big( A \cap \hat{S} \big) \big] = \hat{S} \qquad \text{and} \qquad \big[ \hat{S} \cup \big( A \cap R \setminus \hat{S}_{\blacksquare}^{\hat{\x}} \big) \big] = R\,.
\end{equation}
Note that these two events depend on disjoint sets of infected sites, and are therefore independent; we will bound the first using Lemma~\ref{lem:doublegaps}, and the second using the induction hypothesis. Indeed, by Lemma~\ref{lem:doublegaps} (and since $g(z)$ is decreasing) we have
\[\Pr_p\Big( \big[ S \cup \big( A \cap \hat{S} \big) \big] = \hat{S} \Big) \le \exp\big( - j g(bq) \big)\,,\]
and by the induction hypothesis (under the assumption that no additional infections are found in the buffers and that the right-hand side of $R$ is not reached),
\[\Pr_p\Big( \big[ \hat{S} \cup \big( A \cap R \setminus \hat{S}_{\blacksquare}^{\hat{\x}} \big) \big] = R \Big) \le F^{x+1,y,z}(s - j,t)\,.\]
It follows that the probability that there exists $j \ge 0$ such that the events in~\eqref{eq:key:small:indep:events:again} both hold is at most
\[\sum_{j=0}^{s-1} \exp\big( - j g(bq) \big) F^{x+1,y,z}(s-j,t) = \frac{Cs}{\sqrt{b}} \cdot F^{x,y,z}(s,t) \le 4C\delta \cdot F^{x,y,z}(s,t)\]
since $s \le 4\delta\sqrt{\sh R}$. Since $\delta = \delta(C) > 0$ was chosen sufficiently small, this bound suffices in this case. For the full details of the proof, see the Appendix.
\end{proof}

The proof of Lemma~\ref{lem:key:big} is very similar to that of Lemma~\ref{lem:key:small}, and so we shall give here only a single calculation from the proof, which illustrates the main additional technicality that arises in this setting, and shows why the term $e^{\sh(R) qz}$ is needed in the statement of the lemma. The full details can once again be found in the Appendix.

\begin{proof}[Sketch of the proof of Lemma \ref{lem:key:big}]
Recall that $D_2^{\x}(S,R)$ denotes the event that 
\[\big[ S \cup \big( A \cap R \setminus \Sbar \big) \big] = R \qquad \text{and} \qquad A \cap \Sfr = \emptyset\,.\]
Let $R$ be a $2$-critical rectangle with dimensions $\dim(R) = (a,b)$; as in the proof of Lemma~\ref{lem:key:small}, we use induction on the pair $(s+t,-(x+y))$, this time to prove that 
\begin{equation*}\label{ih:key:big}
\Pr_p\big( D_2^{\x}(S,R) \big) \le \hat{F}^{x,y,z}(s,t)\,,
\end{equation*}
where
\[\hat{F}^{x,y,z}(s,t) := \left( C e^{\sh(R) q} \right)^z \left( C \sqrt{q} e^{-aq} \right)^y \left( C\sqrt{q} e^{-bq} \right)^x \exp\big( - s g(bq) - t g(aq) \big)\,,\]
for every $0 \le s,t \le 4 f(R)$ and $\x \in \{0,1\}^\cI$, and every $S \subset R$ with $\dim(S) = (a-s,b-t)$, where $x$ and $y$ are as defined in~\eqref{def:xy}, and $z = z(S,R)$.

In this sketch we will only consider one very particular (but instructive) configuration, which is illustrated  in Figure~\ref{fig:sideb}. In this example, the top buffer is of height 1, the left and bottom buffers are empty, and we attempt to grow $S$ to the right in search of a double gap. However, before finding one, we pass an infected site $u \in A$ in the (new part of the) top buffer, which instead causes us to grow upwards by one step.

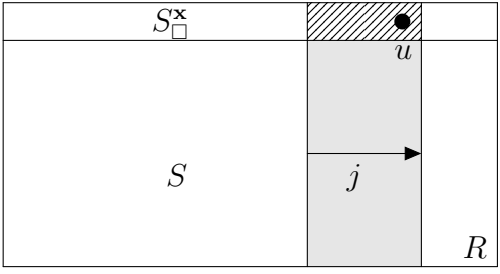
\begin{figure}
\floatbox[{\capbeside\thisfloatsetup{capbesideposition={right,top}}}]{figure}[\FBwidth]{\begin{tikzpicture}[line cap=round,line join=round,>=triangle 45,x=0.25cm,y=0.25cm]
\clip(-20.1,-12.1) rectangle (6.1,2.1);
\fill[fill=black,fill opacity=0.1] (-4,-12) -- (2,-12) -- (2,0) -- (-4,0) -- cycle;
\fill[fill=black,pattern=north east lines] (-4,2) -- (-4,0) -- (2,0) -- (2,2) -- cycle;
\draw (-20,-12)-- (6,-12);
\draw (6,-12)-- (6,2);
\draw (6,2)-- (-20,2);
\draw (-20,2)-- (-20,-12);
\draw (-4,-12)-- (-4,0);
\draw [->] (-4,-6) -- (2,-6);
\draw (-4,2)-- (-4,0);
\draw (-20,0)-- (-4,0);
\draw (2,-12)-- (2,0);
\draw (2,0)-- (2,2);
\draw (-4,0)-- (6,0);
\draw (-12.,-6.) node[anchor=north west] {$S$};
\draw (-12.75,2.25) node[anchor=north west] {$S_{\square}^{\x}$};
\draw (-2.5,-6.) node[anchor=north west] {$j$};
\draw (3.6,-9.8) node[anchor=north west] {$R$};
\draw (0.,0.3) node[anchor=north west] {$u$};
\begin{footnotesize}
\end{footnotesize}
\begin{scriptsize}
\fill [color=black] (1,1) circle (3pt);
\end{scriptsize}
\end{tikzpicture}}
{\caption{$S$ grows $j$ to the right and reaches the infected site $u$ in the hatched region before a double gap to the right. Thus, the shaded region has no vertical double gap.}
\label{fig:sideb}}
\end{figure}

Let $j$ denote the $\|\cdot\|_\infty$-distance of $u$ from $S$, and denote by 
\[\hat{S}' := \bigcup_{i = 0}^j \big( S + (i,0) \big) \qquad \text{and} \qquad \hat{S} := \hat{S}' \cup \big( \hat{S}' + (0,1) \big)\]
so $\hat{S}$ is the rectangle formed by the growth of $S$ to the right, and one step upwards (using $u$). As noted above, we will assume in this sketch that all of the buffers of $\hat{S}$ are empty except $B_{(1,0)}(\hat{S},R)$. We will sum over choices of $j$ the probability that $A \cap \Sfr = \emptyset$ (as in the definition of the event $D_2^{\x}(S,R)$), that $u \in A$, that there is no double gap to the right of $S$ before it reaches $u$, and that  
\[\big[ \hat{S} \cup \big( A \cap R \setminus \hat{S}_{\square}^{\hat{\x}} \big) \big] = R\,,\]
where $\hat{\x} = \x - \1_{(0,1)} = \0$. Note that these four events are independent, and moreover the probability of each is easy to bound. Indeed, note that $\Pr_p( u \in A) = p$, that
\[\Pr_p\big( A \cap \Sfr = \emptyset \big) = (1 - p)^{a-s} = \exp\big( - (a - s)q \big) \le 2 \cdot e^{- aq}\]
since $\lgs(R) \le (1/(2q))\log(1/q)$, and so $s \le (4\delta / \sqrt{q}) \cdot e^{\sh(R) q} \le 4\delta / q$, that  
\[\Pr_p\big( \hat{S}' \setminus S \text{ has no vertical double gap} \big) \le \exp\big( - (j - 1) g(bq) \big)\]
by Lemma~\ref{lem:doublegaps} (and since $g(z)$ is decreasing), and that 
\[\Pr_p\Big( \big[ \hat{S} \cup \big( A \cap R \setminus \hat{S}_{\square}^{\hat{\x}} \big) \big] = R \Big) \le \hat{F}^{0,0,1}(s - j,0) = \frac{e^{aq  - \sh(R) q + j g(bq) + g(aq) }}{C^2 \sqrt{q}} \cdot \hat{F}^{0,1,2}(s,1)\,,\]
by the induction hypothesis. It follows that the probability that there exists $j \ge 0$ such that the four events above all hold is at most
\begin{equation}\label{eq:key:big:indep:finalcalc}
\sum_{j=1}^{s-1} 2p \cdot \frac{e^{- \sh(R) q + g(bq) + g(aq) }}{C^2 \sqrt{q}} \cdot \hat{F}^{0,1,2}(s,1) \le 4s\sqrt{q} \cdot \frac{e^{- \sh(R) q}}{C^2} \cdot \hat{F}^{0,1,2}(s,1)\,,
\end{equation}
since $p \le q$, and since $a,b \ge B/q$ implies that $e^{g(aq) + g(bq)} \le 2$. Finally, recall that $s \sqrt{q} \le 4\delta \cdot e^{\sh(R) q}$, so the right-hand side of~\eqref{eq:key:big:indep:finalcalc} is at most $(16\delta / C^2 ) \cdot \hat{F}^{0,1,2}(s,1)$, as required. Once again, see the Appendix for the full details of the proof.
\end{proof}

\section{Hierarchies}\label{sec:hier}

In this section we will define precisely the family of hierarchies that we will use in the proof of Theorem~\ref{thm:droplet}.  Our definition is more complicated and restrictive than those used in~\cite{Gravner12,Holroyd03}, and is designed to take advantage of the bounds proved in Section~\ref{sec:key}, and to allow us to prove a sufficiently strong upper bound on the number of hierarchies.

\subsection{Good and satisfied hierarchies}

\begin{defn}\label{def:hier}
Let $R$ be a rectangle. A \emph{hierarchy} $\cH$ for $R$ is an oriented rooted tree $G_{\cH}$ with edges pointing away from the root (``downwards''), with edges $e$ labelled with vectors $\x(e)\in\{0,1\}^{\cI}$ and vertices $u$ labelled with rectangles $R_u \subset R$. Let $N_{G_\cH}(u)$ denote the out-neighbourhood of $u$ in $G_\cH$. We require them to satisfy the following conditions.
\begin{enumerate}[label=(\alph*)]
\item The label of the root is $R$.\smallskip
\item Each vertex has out-degree at most two.\smallskip
\item If $v\in N_{G_{\cH}}(u)$, then $R_v \subset R_u$.\smallskip
\item If $N_{G_\cH}(u)=\{v,w\}$, then $[R_v\cup R_w]=R_u$.
\end{enumerate}
\end{defn}

We will write $L(\cH)$ for the set of leaves of $G_{\cH}$, and refer to the rectangles associated with leaves $u \in L(\cH)$ as \emph{seeds} of the hierarchy. We will also refer to vertices with out-degree two as \emph{split vertices}. 

We next define a subclass of `good' hierarchies that is sufficiently small to allow us to use the union bound (see Lemma~\ref{lem:weighted:counting}), but sufficiently large so that every internally filled rectangle $R$ can be associated with a good hierarchy that encodes the growth of the infected sites inside $R$ (see Lemma~\ref{lem:hier:exists}). To do so, we will need one more piece of notation: if $S\subset R$ are rectangles with $R = [r_{(-1,0)},r_{(1,0)}] \times [r_{(0,-1)},r_{(0,1)}]$ and $S = [s_{(-1,0)},s_{(1,0)}] \times [s_{(0,-1)}, s_{(0,1)}]$, then we define
\[d_{\mathbf j}(S,R) := | r_{\mathbf j} - s_{\mathbf j}|\]
for each $\mathbf j\in\mathcal{I}$, and $d(S,R) := \max\big\{ d_{\mathbf j}(S,R) : \, \mathbf j\in\mathcal{I}\big\}$.

\pagebreak

\begin{defn}\label{def:hier:good}
We say that a hierarchy $\cH$ is \emph{good} if the following conditions hold for every $u \in V(G_{\cH})$:
\begin{enumerate}[label=(\alph*)]
\setcounter{enumi}{4}
\item \label{def:hier:f} $u$ is a leaf of $G_{\cH}$ if and only if $\sh(R_u) \le q^{-1/2}$;\smallskip
\item \label{def:hier:g} if $N_{G_\cH}(u) = \{v\}$, then 
\[d(R_v,R_u) \le 2f(R_u)\,;\]
\item \label{def:hier:h} if $v \in N_{G_\cH}(u)$ and either $|N_{G_\cH}(u)| = 2$ or $|N_{G_\cH}(v)| = 1$, then 
\[d(R_v,R_u) \ge f(R_u)\,;\]
\item \label{def:hier:j} if $N_{G_\cH}(u) = \{v\}$ and $|N_{G_\cH}(v)| = 1$, then $\x(uv) \leq \mathbbm{1}_{Z(R_v,R_u)}$, and moreover either\smallskip
\begin{enumerate}[label=(\Roman*)]
\item \label{def:hier:j:i} $\|\x(uv)\| = z(R_v,R_u)$, or\smallskip
\item \label{def:hier:j:ii} $\|\x(uv)\| = z(R_v,R_u) - 1$ and $d(R_v,R_u) \in \{ f(R_u), f(R_u) + 1\}$;\smallskip
\end{enumerate} 
\item \label{def:hier:i} if $v \in N_{G_\cH}(u)$ and either $|N_{G_\cH}(u)| \neq 1$ or $|N_{G_\cH}(v)| \neq 1$, then $\x(uv) = \0$;
\end{enumerate}
where the function $f(R)$ was defined in~\eqref{def:f}, and $\| \x(uv) \| = \sum_{\d \in Z(R_v,R_u)} x(uv)_\d$.
\end{defn}

Finally, we need to define the family of events that we require to occur disjointly. In order to do so, let us first choose a path (the \emph{trunk}) from the root to a leaf of a hierarchy $\cH$ by choosing at each split vertex the out-neighbour whose associated rectangle has larger short side (if they are equal, choose arbitrarily). We will write $\tr(\cH)$ for the set of edges of the trunk.

\begin{defn}\label{def:hier:sat}
A hierarchy $\cH$ is \emph{satisfied} by $A$ if the following events all occur \emph{disjointly}:
\begin{enumerate}[label=(\alph*)]
\setcounter{enumi}{9}
\item If $u \in L(\cH)$, then the rectangle $R_u$ is internally filled by $A$;\smallskip
\item If $N_{G_{\cH}}(u)=\{v\}$ and $uv \not\in \tr(\cH)$, then $D_1^{\x}(R_v,R_u)$ holds, where $\x=\x(uv)$;\smallskip
\item If $N_{G_{\cH}}(u)=\{v\}$ and $uv \in \tr(\cH)$, then $D_2^{\x}(R_v,R_u)$ holds, where $\x=\x(uv)$.
\end{enumerate}
\end{defn}

We remark that the purpose of the trunk is to guarantee that the unoccupied frames in the events $D_2^{\x}(R_v,R_u)$ occur disjointly. For the sake of brevity, we will often say that a rectangle is in the trunk of a hierarchy $\cH$, when we really mean that the associated vertex of $G_\cH$ is in the trunk, and trust that this will cause no confusion.

\subsection{Fundamental properties}

The following deterministic lemma implies that every internally filled rectangle that satisfies the condition~\eqref{eq:droplet:conditions} of Theorem~\ref{thm:droplet} has a good and satisfied hierarchy. 

\begin{lem}\label{lem:hier:exists}
Let $R$ be a rectangle that is internally filled by a set $A$, and suppose that
\begin{equation}\label{eq:longR:upper:bound}
\lgs(R) \/ \le \, \frac{1}{2q} \log \frac{1}{q}\,.
\end{equation}
Then there exists a good hierarchy $\cH$ for $R$ that is satisfied by $A$.
\end{lem}

A similar lemma was proved by Holroyd in~\cite{Holroyd03} using the following lemma, which is a straightforward consequence of the `rectangles process' of Aizenman and Lebowitz~\cite{Aizenman88}. 

\begin{lem}[Proposition~30 of~\cite{Holroyd03}]\label{lem:dj:span}
Let $R$ be a rectangle with $\lgs(R) \ge 2$. If $R$ is internally filled by $A$, then there exist rectangles $S_1,S_2 \subsetneq R$, with $[S_1 \cup S_2] = R$, that are disjointly internally filled by $A$.
\end{lem}

We will need the following slight (and straightforward) strengthening of this lemma.

\begin{lem}\label{lem:dj:span:stronger}
Let $R$ be a rectangle with $\lgs(R) \ge 2$. If $R$ is internally filled by $A$, then there exist rectangles $S_1,S_2 \subsetneq R$, with $[S_1 \cup S_2] = R$, such that 
\begin{equation}\label{eq:dj:span:stronger}
\big[ A \cap ( S_1 \setminus S_2 ) \big] = S_1 \qquad \text{and} \qquad \big[ A\cap ( S_2 \setminus S_1 ) \big] = S_2\,.
\end{equation}
\end{lem}

\begin{proof}
By taking a subset if necessary, we may assume that $A$ is a minimal percolating set for~$R$, i.e., that $A$ is minimal such that $[A] = R$. We claim that for such a set $A$, the rectangles $S_1$ and $S_2$ given by Lemma~\ref{lem:dj:span} in fact satisfy~\eqref{eq:dj:span:stronger}. Indeed, suppose that $A_1$ and $A_2$ are disjoint subsets of $A$ such that $S_1 = [A_1]$, $S_2 = [A_2]$ and $[S_1 \cup S_2] = R$, and observe that
\[\big[ A \setminus ( A_1 \cap S_2 ) \big] = R\,,\]
since $S_2 \subset [A \setminus A_1]$ and $[A] = R$. By the minimality of $A$, it follows that $A_1 \cap S_2 = \emptyset$, and similarly $A_2 \cap S_1 = \emptyset$ as required.
\end{proof}

Now we prove that any internally filled rectangle has a good and satisfied hierarchy.

\begin{proof}[Proof of Lemma~\ref{lem:hier:exists}]
The proof is similar to that of~\cite[Proposition~32]{Holroyd03}, but since there are several slightly subtle (and important) differences, we will give the details in full. 

We prove the statement by induction on $\phi(R)$. If $\sh(R) \le q^{-1/2}$ then we can let $\cH$ be the hierarchy with only one vertex, which is good by the bound on $\sh(R)$, and satisfied since $R$ is internally filled by $A$. So assume that $\sh(R) > q^{-1/2}$, and that the lemma holds for all rectangles with semi-perimeter strictly smaller than $\phi(R)$. 

We use Lemma~\ref{lem:dj:span:stronger} to construct a sequence of rectangles
\[R = T_0 \supsetneq T_1 \supsetneq \cdots \supsetneq T_m\]
for some $m \in \N$ as follows. For each $i \ge 0$, suppose that we have already constructed $T_i$, and let $T_{i+1}$ and $T'_i$ be the two rectangles  given by Lemma~\ref{lem:dj:span:stronger} applied to $T_i$, where $d(T_{i+1},R) \le d(T_i',R)$. Now let $m$ be minimal such that $d(T_m,R) \ge f(R)$, and note that $m$ exists because~\eqref{def:f} and~\eqref{eq:longR:upper:bound} imply that $\lgs(R) > 2f(R)$. We consider three cases:

\medskip
\noindent \textbf{Case 1:} $d(T_m,R) \le 2f(R)$.
\medskip

In this case, instead of applying the induction hypothesis to $T_m$ (as in, e.g.,~\cite{Gravner12,Holroyd03}), we let $T_m \subset S \subset R$ be a maximal internally filled rectangle with $d(S,R) \ge f(R)$, and apply the induction hypothesis to $S$. (We remark that this is a crucial step in our proof.) Observe that, by the maximality of $S$, one of the following two events holds:
\begin{itemize}
\item[(I)] There is no element of $A$ within distance two of $S$. In this case the event $D_2^{\x}(S,R)$ holds for $\x=\mathbbm{1}_{Z(S,R)}$, since $A \cap \Sfr = \emptyset$, and $R$ is internally filled by $A$.\smallskip
\item[(II)] $d( [S \cup \{u\}], R ) < f(R)$ for each $u \in A$ within distance two of $S$, and therefore $d(S,R) \in \{f(R), f(R)+1\}$. Choose $\d \in \cI$ such that $d_{\d}(S,R) \in \{f(R), f(R)+1\}$, and set $\x = \mathbbm{1}_{Z(S,R)\setminus\{\d\}}$. We claim that the event $D_2^{\x}(S,R)$ holds. Indeed, $R$ is internally filled by $A$, and if there exists an element $u \in A \cap \Sfr$, then we have $d( [S \cup \{u\}], R ) \ge d_{\d}(S,R) \ge f(R)$, contradicting the maximality of $S$. 
\end{itemize}

Now, let $\cH'$ be the good and satisfied hierarchy for $S$ given by the induction hypothesis, and form a hierarchy $\cH$ for $R$ by adding an edge from a vertex $u$ corresponding to $R$, to the (root) vertex $v$ of $\cH'$ corresponding to $S$. If $|N_{G_\cH}(v)| \neq 1$, then set $\x(uv) = 0$, and otherwise define $\x(uv)$ as above, i.e., $\x(uv) = \mathbbm{1}_{Z(S,R)}$ in (I) and $\x(uv)=\mathbbm{1}_{Z(S,R)\setminus\{\d\}}$ in (II), where $d_\d(S,R)\in\{f(R),f(R)+1\}$.

We claim that $\cH$ is good, and satisfied by $A$. To see that $\cH$ is good, recall that $\cH'$ is good, and note that $f(R) \le d(S,R) \le 2f(R)$, and that if $|N_{G_\cH}(v)| = 1$ then either $\|\x(uv)\| = z(S,R)$ (if there is no element of $A$ within distance two of $S$), or $\|\x(uv)\| = z(S,R) - 1$ and $d(S,R) \in \{f(R), f(R)+1\}$ (otherwise). 

To see that $\cH$ is satisfied by $A$, recall that $\cH'$ is satisfied by $A$, and note that the event $D_2^{\x(uv)}(S,R)$ occurs (by the observations above). Moreover, the event $D_2^{\x(uv)}(S,R)$ depends only on sites in $R \setminus S$, whereas the events involved in $\cH'$ depend only on sites inside $S$. The events involved in $\cH$ therefore occur disjointly, as required.

\medskip
\noindent \textbf{Case 2:} $d(T_1,R) > 2f(R)$.
\medskip

Let $\{ S_1, S_2 \} = \{ T_1, T_0' \}$, where the labelling is chosen so that $\sh(S_1) \ge \sh(S_2)$, and recall that $[S_1 \cup S_2] = R$, that~\eqref{eq:dj:span:stronger} holds, i.e.,  
\[\big[ A \cap ( S_1 \setminus S_2 ) \big] = S_1 \qquad \text{and} \qquad \big[ A \cap ( S_2 \setminus S_1 ) \big] = S_2\,,\]
and that $2f(R) < \min\{ d(S_1,R), d(S_2,R) \}$. Set $A_1 := A \cap S_1$ and $A_2 := A \cap ( S_2 \setminus S_1 )$ and, applying the induction hypothesis, let $\cH'_1$ and $\cH'_2$ be good hierarchies for $S_1$ and $S_2$ that are satisfied by $A_1$ and $A_2$, respectively. Form a hierarchy $\cH$ for $R$ by adding edges from a vertex $u$ corresponding to $R$, to the roots of $\cH'_1$ and $\cH'_2$, that is, the vertices $v_1$ and $v_2$ corresponding to $S_1$ and $S_2$ (respectively), and set $\x(uv_1) = \x(uv_2) = \0$. 

We claim that $\cH$ is good, and satisfied by $A$. To see that $\cH$ is good, recall that $\cH'_1$ and $\cH'_2$ are good, and that $\min\{ d(S_1,R), d(S_2,R) \} \ge 2f(R)$. To see that $\cH$ is satisfied by $A$, recall that $\cH'_1$ and $\cH'_2$ are satisfied by $A$, and note that the trunk of $\cH$ can be chosen to pass through $S_1$. Now, all of the increasing events involved in $\cH'_1$ and $\cH'_2$ are witnessed by disjoint subsets of $A_1$ and $A_2$, respectively, and $A_1$ and $A_2$ are disjoint sets (since $A_1 \subset S_1$ and $A_2 \cap S_1 = \emptyset$), so all of these events occur disjointly. Since $S_2$ is not in the trunk, the only remaining events are the decreasing events involved in $\cH'_1$ (that the frames of rectangles in the trunk are empty), which all depend only on sites in $S_1 \setminus A_1$, and therefore occur disjointly from those that depend on $A_1$ and $A_2$. The events involved in $\cH$ therefore occur disjointly, as required.

\medskip
\noindent \textbf{Case 3:} $d(T_m,R) > 2f(R)$, and $m > 1$.
\medskip

Set $S = T_{m-1}$, and let $\{ S_1, S_2 \} = \{ T_m, T_{m-1}' \}$, where the labelling is chosen so that $\sh(S_1) \ge \sh(S_2)$, and recall that $[S_1 \cup S_2] = S$, that~\eqref{eq:dj:span:stronger} holds, and that 
\begin{equation}\label{eq:distances:Sis}
d(S,R) < f(R) \qquad \text{and} \qquad d(S_i,S) \ge d(S_i,R) - d(S,R) > f(R)
\end{equation}
for each $i \in \{1,2\}$. As in Case~2, set $A_1 := A \cap S_1$ and $A_2 := A \cap ( S_2 \setminus S_1 )$ and let $\cH'_1$ and $\cH'_2$ be good hierarchies for $S_1$ and $S_2$, satisfied by $A_1$ and $A_2$, respectively, given by the induction hypothesis. Form a hierarchy $\cH$ for $R$ by adding an edge from a vertex $u$ corresponding to $R$, to a vertex $v$ corresponding to $S$, and edges from $v$ to the roots of $\cH'_1$ and $\cH'_2$, that is, the vertices $w_1$ and $w_2$ corresponding to $S_1$ and $S_2$ (respectively), and set $\x(uv) = \x(vw_1) = \x(vw_2) = \0$. 

It is again easy to see that $\cH$ is good, since $\cH'_1$ and $\cH'_2$ are good, and using the inequalities~\eqref{eq:distances:Sis}. We claim that moreover $\cH$ is satisfied by $A$; this follows almost exactly as in Case~2, but for completeness we will spell out the details. Recall that $\cH'_1$ and $\cH'_2$ are satisfied by $A$, observe that the event $D_2^{\x}(S,R)$ holds (since $R$ is internally filled by $A$, and $\x(uv) = \0$), and note that the trunk of $\cH$ can be chosen to pass through $S_1$.

Now, all of the increasing events involved in $\cH'_1$, $\cH'_2$ and $D_2^{\x}(S,R)$ are witnessed by disjoint subsets of $A_1$, $A_2$ and $A \setminus S$, respectively, and $A_1$, $A_2$ and $A \setminus S$ are disjoint sets (since $A_1 \subset S_1$, $A_2 \cap S_1 = \emptyset$, and $A_1 \cup A_2 \subset S$), so all of these events occur disjointly. Now, $S_2$ is not in the trunk, and $\x(uv) = \0$, so the only remaining events are the decreasing events involved in $\cH'_1$ (that the frames of rectangles in the trunk are empty), which all depend only on sites in $S_1 \setminus A_1$, and therefore occur disjointly from those that depend on $A_1$, $A_2$ and $A \setminus S$. The events involved in $\cH$ therefore occur disjointly, as required.
\end{proof}

We are now ready to deduce our fundamental bound on the probability that a rectangle is internally filled, cf.~\cite[Section~10]{Holroyd03} or~\cite[Lemma~7]{Gravner12}. Given a rectangle $R$, let us  write $\cH_R$ for the set of good hierarchies for $R$, and for each $\cH \in \cH_R$, set
\[G_{\cH}^{(2)} = \tr({G_\cH}) \qquad \text{and} \qquad G_\cH^{(1)} = E(G_\cH)\setminus G_\cH^{(2)}\,.\]
Recall also that $\Pr_p$ denotes the probability space obtained by choosing $A$ to be a $p$-random subset of $[n]^2$, and let us write $I(R)$ for the event that $R$ is internally filled by $A$.

\begin{lem}\label{lem:basic:bound}
If $R$ is a rectangle with $\lgs(R) \le (1/(2q)) \log (1/q)$, then
\begin{equation}\label{eq:basic:bound}
\Pr_p\big( I(R) \big) \le \sum_{\cH \in \cH_R} \bigg( \prod_{j=1}^{2} \prod_{\substack{uv\in G_\cH^{(j)} \\ N_{G_{\cH}}(u)=\{v\}}} \Pr_p\big( D_j^{\x(uv)}(R_v,R_u) \big) \bigg) \bigg(\prod_{u \in L(\cH)} \Pr_p\big( I(R_u) \big) \bigg)\,.
\end{equation}
\end{lem}

\begin{proof}
By Lemma \ref{lem:hier:exists}, if $R$ is internally filled by $A$ then there exists a hierarchy $\cH \in \cH_R$ that is satisfied by $A$. By the union bound (over $\cH_R$), it will therefore suffice to show that for each $\cH \in \cH_R$, the probability that $\cH$ is satisfied by $A$ is bounded above by the corresponding term of the right-hand side of~\eqref{eq:basic:bound}. But this follows immediately from Definition~\ref{def:hier:sat} by Reimer's Theorem, and hence~\eqref{eq:basic:bound} holds, as claimed.
\end{proof}

We remark that we did not actually need the full power of Reimer's Theorem in the proof above, since our events are particularly simple: each is the intersection of an increasing and a decreasing event, and the decreasing events are moreover primitive (i.e., a fixed set must be empty). For events of this form, the conclusion of Reimer's theorem is actually a straightforward consequence of the van den Berg--Kesten lemma. 

\subsection{Weighted counting}

Recall that we can bound the probabilities on the right-hand side of~\eqref{eq:basic:bound} using Lemma~\ref{lem:seeds} and Corollary~\ref{cor:key}. It therefore remains to control the `size' of the set $\cH_R$; however, since hierarchies with many empty buffers are more numerous and less likely to be satisfied, we would like to give them lower `weight' when measuring the size of $\cH_R$. Due to the form of the right-hand side of~\eqref{eq:cor:key} (in particular, its dependence on $\| \x \|$), we will find the following definition useful.

\begin{defn}\label{def:weight}
Given a rectangle $R$, the \emph{weight} of a hierarchy $\cH \in \cH_R$ is defined to be
\[w(\cH) := \prod_{N_{G_{\cH}}(u)=\{v\}} \left( \frac{1}{f(R_u)} \right)^{\|\x(uv)\|}\,.\]
\end{defn}

Given a hierarchy $\cH$, we will write $v(\cH)$ for the number of vertices of $G_\cH$, and $s(\cH) = |L(\cH)|$ for the number of seeds of $\cH$. (Note that $\cH$ has exactly $s(\cH) - 1$ split vertices.) Given a rectangle $R$, let us write
\begin{equation}\label{def:HNM}
\cH_R(N,M) := \big\{ \cH \in \cH_R \,:\, v(\cH) = N, \, s(\cH) = M \big\}\,.
\end{equation}
The following lemma bounds the total weight of $\cH_R(N,M)$.

\begin{lem}\label{lem:weighted:counting}
Let $R$ be a rectangle, and let $N,M \in \N$. Then
\[\sum_{\cH \in \cH_R(N,M)} w(\cH) \le \exp\Big( 16 \big( N + M \log \phi(R) \big) \Big)\,.\]
\end{lem}

\begin{proof}
Let us first fix the tree $G_\cH$ and the labels $\x(e)$ for each $e \in G_\cH$. There are at most $3^N$ oriented rooted trees on $N$ vertices with maximum out-degree at most two (and edges oriented away from the root), and at most $2^{4N}$ choices for the labels $\x(e) \in \{0,1\}^\cI$. We will choose the rectangles one by one, starting at the root and working our way down the tree, counting the number of choices (given the earlier choices) at each step.

Let $u \in V(G_\cH)$, and suppose that we have already chosen the rectangle $R_u$. Suppose first that $u$ is a split vertex, and let $N_{G_\cH}(u) = \{v,w\}$. We clearly have at most $\phi(R)^4$ choices for each of $R_v$ and $R_w$, and hence (recalling that there are $M - 1$ split vertices) the total number of choices for the rectangles associated with the out-neighbours of split vertices is at most $\phi(R)^{8M}$. Similarly, if $N_{G_\cH}(u) = \{v\}$ and $v$ is a split vertex or a seed, then we have at most $\phi(R)^4$ choices for $R_v$, so  the total number of choices for the rectangles associated with such vertices is also at most $\phi(R)^{8M}$.

Suppose now that $N_{G_\cH}(u) = \{v\}$ and $|N_{G_\cH}(v)| = 1$, and recall from Definition~\ref{def:hier:good} that $d(R_v,R_u) \le 2f(R_u)$, and that either $\|\x(uv)\| = z(R_v,R_u)$, or 
\[\|\x(uv)\| = z(R_v,R_u) - 1 \qquad \text{and} \qquad d(R_v,R_u) \in \{ f(R_u), f(R_u) + 1\}\,.\]
In either case we have at most $2^{10} f(R_u)^{\|\x(uv)\|}$ choices for $R_v$, and it follows that 
\[\sum_{\cH \in \cH_R(N,M)} w(\cH) \le ( 3 \cdot 2^{14} )^N \cdot \phi(R)^{16 M} \le \exp\Big( 16\big( N + M \log \phi(R) \big) \Big)\,,\]
as claimed.
\end{proof}

\subsection{The height of a hierarchy}

Let us write $h(\cH)$ for the \emph{height} of the hierarchy $\cH$, that is, the number of vertices in the longest path from the root to a leaf of $G_\cH$. In this subsection we will prove some straightforward (though sometimes slightly technical) properties of the height of a good hierarchy. 

Let us begin with a simple lower bound on the size of a seed in a good hierarchy.

\begin{obs}\label{obs:small:seeds}
Let $R$ be a rectangle, and suppose that $\lgs(R) \le (1/(2q)) \log (1/q)$ and $\sh(R) \ge q^{-1/2}$. If $\cH \in \cH_R$, and $v \in V(G_\cH)$, then 
\[\phi(R_v) \ge \frac{\delta}{q^{1/4}}\,.\]
\end{obs}

\begin{proof}
It suffices to prove the claimed bound for seeds of $\cH$, so assume that $v$ is a seed, and that $v \in N_{G_{\cH}}(u)$ (if $\cH$ has only one vertex then the result is trivial). Since $u$ is not a seed and $\cH$ is a good hierarchy (see Definition~\ref{def:hier:good}), we have $\sh(R_u) > q^{-1/2}$. Thus
\[\phi(R_v) \ge \min\big\{ \phi(R_u) - 8f(R_u), \, f(R_u) \big\} \ge \frac{\delta}{q^{1/4}}\,,\]
as required, since if $|N_{G_{\cH}}(u)| = 1$ then $d(R_v,R_u) \le 2f(R_u)$, and if $|N_{G_\cH}(u)| = 2$ then $d(R_v,R_u) \ge f(R_u)$. Note that in the first step we used the fact that if $[R_v \cup R_w] = R_u$, then $\phi(R_v) + \phi(R_w) \ge \phi(R_u)$, and in the second we used the definition~\eqref{def:f} of $f(R)$. 
\end{proof}

Next, let us recall a simple but key observation from~\cite{Gravner12}. Let us say that a seed $S$ is \emph{large} if $\lgs(S) \ge 1 / (3\sqrt{q})$, and denote by $m(\cH)$ the number of large seeds of a hierarchy~$\cH$. Observe (or recall from~\cite[Observation 17]{Gravner12}) that every non-leaf vertex of a good hierarchy lies above a large seed.

\begin{obs}\label{obs:height:large:seeds}
Let $R$ be a rectangle with $\sh(R) \ge q^{-1/2}$. If $\cH \in \cH_R$, then 
\[v(\cH) \le 2 \cdot h(\cH) \cdot m(\cH)\,.\]
\end{obs}

\begin{proof}
Since $\cH$ is a good hierarchy for $R$, every non-leaf $u \in V(G_\cH)$ lies above a large seed. There are therefore are most $h(\cH) \cdot m(\cH)$ vertices that are either large seeds or non-seeds. Since each small seed is adjacent to a non-seed, and each non-seed is adjacent to at most one small seed, the claimed bound follows.
\end{proof}

We will use Observation~\ref{obs:height:large:seeds} together with the following lemma to bound the number of vertices in a `typical' hierarchy $\cH \in \cH_R$.   

\begin{lem}\label{lem:height:or:vertex}
Let $R$ be a rectangle with $\lgs(R) \le (1/(2q)) \log (1/q)$, and let $\cH \in \cH_R$. Then either\footnote{Recall that $L_2 = L_2(B,C,\delta,L_1)$ is a sufficiently large constant. The lemma also holds with a smaller constant in~\eqref{eq:normal:height}, but this particular tripartition will be convenient in Section~\ref{sec:proof}.} 
\begin{equation}\label{eq:normal:height}
h(\cH) \le \frac{L_2}{\sqrt{q}}\,,
\end{equation}
or there exists a vertex $u \in V(G_\cH)$ such that either
\begin{equation}\label{eq:weird:vertex:lower}
\sh(R_u) \le \frac{B}{q} \qquad \text{and} \qquad \lgs(R_u) \ge 2L_1 \cdot \sh(R_u)\,,
\end{equation}
or
\begin{equation}\label{eq:weird:vertex:upper}
\sh(R_u) \ge \frac{B}{q} \qquad \text{and} \qquad \lgs(R_u) \ge 4 \cdot \sh(R_u)\,.
\end{equation}
\end{lem}

\begin{proof}
Suppose that there exists no vertex $u \in V(G_\cH)$ satisfying either~\eqref{eq:weird:vertex:lower} or~\eqref{eq:weird:vertex:upper}; we will show that $h(\cH) \le L_2 / \sqrt{q}$. To do so, let $v$ be the root of $\cH$, let $P$ be a longest path in $G_{\cH}$ (from $v$ to a seed $w$), and partition (the vertex set of) $P$ into sets 
\[P_1 := \{v\} \cup \big\{ u \in P : \sh(R_u) > B/q \big\}  \qquad \text{and} \qquad P_2 := P \setminus P_1\,.\]
Let $u_1$ be the lowest vertex of $P_1$, and let $u_2$ be the highest vertex of $P_2$.

We first claim that the distance (in $G_{\cH}$) from $u_2$ to $w$ is a most $L_1 / (3\sqrt{q})$. To see this, note that any $w < y \le u_2$ satisfies $\lgs(R_y) < 2L_1 \cdot \sh(R_y)$, and so, by Definition~\ref{def:hier:good}, in the next two consecutive steps up $P$ the semi-perimeter increases by at least 
\[\frac{\delta}{\sqrt{2L_1}} \cdot \sqrt{\lgs(R_y)}\,.\]
Since $L_1$ is large, the claimed bound follows easily. 

Similarly, we claim that the distance (in $G_{\cH}$) from $u_1$ to $v$ is a most $L_1 / (3\sqrt{q})$. To see this, note that any $v > y \ge u_1$ satisfies $4 \cdot\sh(R_y) > \lgs(R_y)$, so in the next two consecutive steps up $P$, either we reach $v$, or the semi-perimeter increases by at least 
\[\frac{\delta}{\sqrt{q}}\exp\big( q \cdot \phi(R_y) / 5 \big)\,.\]
It is again not difficult to see that the claimed bound holds; indeed, the semi-perimeter takes at most $x = L_1 / (6\sqrt{q})$ steps to increase by $5/q$, then at most $x/2$ steps to increase by $5/q$ again, and so on, until it has increased by $(5/q) \log_2 x$ in at most $2x$ steps.
\end{proof}

The next lemma bounds the height of $\cH$ in the case where only~\eqref{eq:weird:vertex:lower} is satisfied. 

\begin{lem}\label{lem:weird:vertex:height:bound}
Let $R$ be a rectangle with $\lgs(R) \le (1/(2q)) \log (1/q)$, and let $\cH \in \cH_R$. Suppose that neither~\eqref{eq:normal:height} nor~\eqref{eq:weird:vertex:upper} holds for any vertex $u \in V(G_\cH)$. Then the vertex $u$ satisfying~\eqref{eq:weird:vertex:lower} may be chosen so that   
\begin{equation}\label{eq:weird:vertex:height:bound}
h(\cH) \le L_1 q^{1/4} \cdot \lgs(R_u)\,.
\end{equation}
\end{lem}

\begin{proof}
Let $u \in V(G_\cH)$ be a vertex satisfying~\eqref{eq:weird:vertex:lower} with $\lgs(R_u)$ maximal, and set $c = \sh(R_u)$ and $d = \lgs(R_u)$. Let $P$ be the longest path in $\cH$, and observe that $P$ contains at most $L_1 / (3\sqrt{q})$ vertices $v$ with $\sh(R_v) > B/q$, as in the proof of Lemma~\ref{lem:height:or:vertex}, since $\cH$ contains no vertex such that~\eqref{eq:weird:vertex:upper} holds. Observe also that $P$ contains at most $(L_1 q^{1/4} / 2) \cdot d$ vertices $v$ with $\lgs(R_v) \le d$, since it follows from Definition~\ref{def:hier:good} that in each two consecutive steps the semi-perimeter increases by at least $\delta q^{-1/4}$. 

Finally, we claim that $P$ contains at most $L_2 / (3\sqrt{q})$ vertices $v$ with $\sh(R_v) \le B/q$ and $\lgs(R_v) > d$. To see this, note that $2L_1 \cdot \sh(R_v) > \lgs(R_v)$, by our choice of $u$, and therefore in each two consecutive steps up $P$, the semi-perimeter increases by at least 
\[\frac{\delta}{\sqrt{2L_1}} \cdot \sqrt{\lgs(R_v)} \ge \frac{\delta}{2\sqrt{L_1}} \cdot \sqrt{\phi(R_v)}\,.\]
It now follows easily that after $L_2 / (3\sqrt{q})$ steps we have $\phi(R_v) \ge 3 L_1 B / q$, and hence $\sh(R_v) \ge B/q$, as claimed. 

Since~\eqref{eq:normal:height} does not hold, it follows that 
\[\frac{L_2}{\sqrt{q}} \le h(\cH) \le \frac{L_1 q^{1/4}}{2} \cdot \lgs(R_u) + \frac{L_2}{2\sqrt{q}}\,,\]
and hence we obtain~\eqref{eq:weird:vertex:height:bound}, as required.
\end{proof}

Define the \emph{upper trunk} of $\cH$ to be the following set of vertices\footnote{Recall that $\tr(\cH)$ denotes the set of \emph{edges} of the trunk; we hope that this minor inconsistency in our notation (which will be quite convenient) will not confuse the reader.} of the trunk:
\[\utr(\cH) := \big\{ u \in V(G_\cH) : u \textup{ is in the trunk of $\cH$ and } \sh(R_u) \ge B/q \big\}\,.\]
The final lemma of this subsection bounds the sum of the semi-perimeters of rectangles in the upper trunk when there does not exist a vertex $u \in V(G_\cH)$ satisfying~\eqref{eq:weird:vertex:upper}. 

\begin{lem}\label{lem:upper:trunk:total:semiperimeter}
Let $R$ be a rectangle with $\lgs(R) \le (1/(2q)) \log (1/q)$, and let $\cH \in \cH_R$. Then either  
\[\sum_{u \in \utr(\cH)} \phi(R_u) \le \frac{L_2}{q^{3/2}},\]
or there exists a vertex $u \in V(G_\cH)$ such that 
\begin{equation}\label{eq:weird:vertex:upper:again}
\sh(R_u) \ge \frac{B}{q} \qquad \text{and} \qquad \lgs(R_u) \ge 4 \cdot \sh(R_u)\,.
\end{equation}
\end{lem}

\begin{proof}
If $u \in \utr(\cH)$, and $u$ does not satisfy~\eqref{eq:weird:vertex:upper:again}, then by Definition~\ref{def:hier:good} (as in the proof of Lemma~\ref{lem:height:or:vertex}), in the next two consecutive steps up the trunk either we reach the root~$v$, or the semi-perimeter increases by at least 
\[\frac{\delta}{\sqrt{q}}\exp\big( q \cdot \phi(R_y) / 5 \big)\,.\]
Set $x = L_1/\sqrt{q}$, and observe (cf. the proof of Lemma~\ref{lem:height:or:vertex}) that there are at most $2^{-k+1} x$ vertices $u \in \utr(\cH)$ with $\phi(R_u) \ge (B+5k)/q$, for each $0 \le k \le \log_2 x$. It follows that
\[\sum_{u \in \utr(\cH)} \phi(R_u) \, \le \, \sum_{k = 0}^{\infty} \frac{B+5k}{q} \cdot \frac{x}{2^{k-1}} \, \le \, \frac{L_2}{q^{3/2}}\,,\]
as required.
\end{proof}

\subsection{The pods of a hierarchy}

To finish this section, let us recall the following important lemma from~\cite{Holroyd03}, which is known as the `pod lemma', and prove a generalization which we be useful in Section~\ref{sec:proof}, below. Recall from~\eqref{eq:def:W} the definition of~$U(S,R)$.

\begin{lem}[Lemma~38 of~\cite{Holroyd03}]\label{lem:pod:Holroyd}
Let $\cH \in \cH_R$. Then there exists a rectangle $S\subset R$ such that
\[\dim(S) \le \sum_{w \in L(\cH)} \dim (R_w)\]
and
\[\sum_{N_{G_{\cH}}(u) = \{v\}} U(R_v,R_u) \ge U(S,R) - 2\big( s(\cH) - 1 \big) q g(\sqrt{q})\,.\]
\end{lem}

Holroyd called the rectangle $S$ the \emph{pod} of $\cH$. Roughly speaking, Lemma~\ref{lem:pod:Holroyd} says that the `cost' of the growth (given the size of the seeds) is minimized by placing all of the seeds near to one another, at the very bottom of the hierarchy. However, when we are in the case corresponding to~\eqref{eq:weird:vertex:lower} (or, more precisely, Lemma~\ref{lem:weird:vertex:height:bound}), we will need to make use of the special rectangle $R_u$, which is somewhere in the middle of $\cH$. In order to use the fact that this rectangle appears in the hierarchy when minimizing the `cost' of growth, we instead form two pods: one corresponding to the growth inside the special rectangle~$R_u$, the other corresponding to the growth of this rectangle to fill $R$.

\begin{lem}\label{lem:pods}
Let $\cH \in \cH_R$ and let $u \in V(G_{\cH})$. Then there exist rectangles $S_1 \subset R_u$ and $R_u \subset S_2 \subset R$, such that
\begin{equation}\label{eq:pods:dims}
\dim(S_1) + \dim(S_2) - \dim(R_u) \le \sum_{w \in L(\cH)} \dim(R_w)
\end{equation}
and
\begin{equation}\label{eq:pods:sum}
\sum_{N_{G_\cH}(v) = \{w\}} U(R_w,R_v) \ge U(S_1,R_u) + U(S_2,R) - 2 \big( s(\cH) - 1 \big) q g(\sqrt{q})\,.
\end{equation}\end{lem}

The proof of Lemma~\ref{lem:pods} is essentially the same as that of Lemma~\ref{lem:pod:Holroyd}, and so we will give only a brief sketch here, and refer the reader to~\cite{Holroyd03} for the details. 

\begin{proof}[Sketch proof of Lemma~\ref{lem:pods}]
We will use induction on the distance from $u$ to the root. Note first that when $u$ is the root of $\cH$, then the claimed conclusion follows from Lemma~\ref{lem:pod:Holroyd} by setting $S_1 = S$ and $S_2 = R$. For the induction step, we divide into cases according to whether the root has one or two neighbours. 

Indeed, suppose first that the root has one neighbour, $x$, and apply the induction hypothesis to the sub-hierarchy of $\cH$ rooted at $x$ to obtain pods $S_1$ and $S_2$. Note that these pods satisfy~\eqref{eq:pods:dims}, and also~\eqref{eq:pods:sum}, since the inequality
\[U(R_x,R) \ge U(S_2,R) - U(S_2,R_x)\,,\]
follows immediately from the definition. On the other hand, if the root has two neighbours, $x$ and $y$, and $u$ is a descendant of $x$, then we apply the induction hypothesis to the sub-hierarchy of $\cH$ rooted at $x$, giving pods $S'_1$ and $S_2'$, and Lemma~\ref{lem:pod:Holroyd} to the sub-hierarchy of $\cH$ rooted at $y$, giving a pod $T$. Set $S_1 := S_1'$, and choose $S_2$, with 
\[\dim(S_2') \le \dim(S_2) \le \dim(S_2') + \dim(T)\]
and
\[U(S_2',R_x) + U(T,R_y) \ge U(S_2,R) - 2q g(\sqrt{q})\]
by applying~\cite[Proposition~15]{Holroyd03}, exactly as in the proof of~\cite[Lemma~38]{Holroyd03}. Noting that $s(\cH) = s(\cH_x) + s(\cH_y)$, the inequalities~\eqref{eq:pods:dims} and~\eqref{eq:pods:sum} follow.   
\end{proof}

\section{Proof of Theorem~\ref{thm:sharp}}\label{sec:proof}

In this section we will put the pieces together and prove Theorem~\ref{thm:sharp}. The main step is the proof of Theorem~\ref{thm:droplet}, which we restate (this time with explicit constants) for convenience. Recall that $I(R)$ denotes the event that $R$ is internally filled by $A$. 

\begin{thm}\label{thm:droplet:again}
Let $R$ be a rectangle with dimensions $\dim(R) = (a,b)$, and suppose that $a \le b$, and 
\begin{equation}\label{eq:droplet:conditions:again}
\frac{3 e^{2B}}{q} \, \le \, b \, \le \, \frac{1}{2q} \log \frac{1}{q}\,.
\end{equation}
Then
\[\Pr_p\big( I(R) \big) \le \exp\left( - \min\left\{ \frac{2\lambda}{q} + \frac{1}{q^{3/4}}, \, ( b - a ) g( aq ) + \frac{2}{q} \int_0^{aq} g(z)\md z - \frac{L_6}{\sqrt{q}} \right\} \right)\,.\]
\end{thm}

We will begin by giving an outline of the proof of Theorem~\ref{thm:droplet:again}, and proving a couple of straightforward technical lemmas. Let us fix a rectangle $R$ as in the theorem until the end of its proof. The first step is to recall that 
\begin{equation}\label{eq:basic:bound:again}
\Pr_p\big( I(R) \big) \le \sum_{\cH \in \cH_R} \bigg( \prod_{j=1}^{2} \prod_{\substack{uv\in G_\cH^{(j)} \\ N_{G_{\cH}}(u)=\{v\}}} \Pr_p\big( D_j^{\x(uv)}(R_v,R_u) \big) \bigg) \bigg(\prod_{u \in L(\cH)} \Pr_p\big( I(R_u) \big) \bigg)\,,
\end{equation}
by Lemma~\ref{lem:basic:bound}, where we used the upper bound on $\lgs(R)$ given by~\eqref{eq:droplet:conditions:again}. Recall that $g(z)$ is decreasing, and (from Definition~\ref{def:hier:good}) that $\sh(R_u) \le q^{-1/2}$ for every leaf $u \in L(\cH)$ of a good hierarchy $\cH$. Therefore, if $\cH \in \cH_R$, then
\begin{equation}\label{eq:seeds:bound}
\Pr_p\big( I(R_u) \big) \le \, 3^{\phi(R_u)} \exp\Big( - \phi(R_u) g(\sqrt{q}) \Big)
\end{equation}
for each leaf $u \in L(\cH)$, by Lemma~\ref{lem:seeds}. Moreover, since $d(R_v,R_u) \le 2f(R_u)$ whenever $N_{G_\cH}(u) = \{v\}$, if $R_u$ is $j$-critical for some $j \in \{1,2\}$, then 
\begin{equation}\label{eq:cor:key:again}
\Pr_p\big( D_j^{\x(uv)}(R_v,R_u) \big) \le C^9 \left(\frac{\delta}{f(R_u)}\right)^{\|\x(uv)\|} \exp\big( - Q(R_v,R_u) + 4 \phi(R_u) q \big)
\end{equation}
by Corollary~\ref{cor:key}, where $Q(R_v,R_u)$ was defined in~\eqref{def:QRS}. Unfortunately, however, some rectangles are neither $1$- nor $2$-critical, and we must deal with these separately. 

\begin{lem}\label{lem:no:long:thin:rectangle}
The probability that there exists an internally filled rectangle $S \subset R$ with 
\begin{equation}\label{eq:no:long:thin:rectangle}
\sh(S) \le \frac{B}{q} \qquad \text{and} \qquad \lgs(S) \ge \frac{3e^{2B}}{q}
\end{equation}
or an internally filled rectangle $S \subset R$ with
\begin{equation}\label{eq:no:small:long:thin:rectangle}
\sh(S) \le \frac{1}{q} \qquad \text{and} \qquad \lgs(S) \ge \frac{B}{2q}
\end{equation}
is at most $e^{- 2 / q}$.
\end{lem}

\begin{proof}
Observe that if $S \subset R$ is internally filled, then it must be crossed from left to right, and from bottom to top. By Lemma~\ref{lem:doublegaps}, it follows that if $S$ satisfies~\eqref{eq:no:long:thin:rectangle} then
\[\Pr_p\big( I(S) \big) \le \exp\Big( - \lgs(S) \cdot g\big( q \cdot \sh(S) \big) \Big) \le \exp\bigg( - \frac{3e^{2B}g(B)}{q} \bigg)\,.\]
Recalling that $g(z) \sim e^{-2z}$ as $z \to \infty$ (and that $B$ is large), and applying the union bound, it follows that the probability that there exists such a rectangle $S$ is at most  
\[\big( \lgs(R) \big)^4  \cdot \exp\left(-\frac{3e^{2B}g(B)}{q}\right) \, \le \, \frac{1}{2} \cdot e^{-2/q}\,.\]
The same bound (with the same proof, noting that $B > 4/g(1)$, since $B$ is sufficiently large) holds if $S$ satisfies~\eqref{eq:no:small:long:thin:rectangle}. The result then follows by the union bound.
\end{proof}

Note that $\lambda = \pi^2 / 18 < 1$, so by Lemma~\ref{lem:no:long:thin:rectangle} we may assume that $R$ contains no internally filled rectangle $S$ satisfying~\eqref{eq:no:long:thin:rectangle} or \eqref{eq:no:small:long:thin:rectangle}. It follows that each rectangle $R_u$ (where $u \in V(G_\cH)$) either satisfies the condition~\eqref{eq:key:small:R} of Lemma~\ref{lem:key:small} (and hence is $1$-critical), or satisfies the condition~\eqref{eq:key:big:R} of Lemma~\ref{lem:key:big} (and hence is $2$-critical). Note also that, since $b \ge 3e^{2B}/q$, by~\eqref{eq:droplet:conditions:again}, we may assume from now on that $a \ge B/q$. 

The next problem is that we would like $R_u$ to be $j$-critical when $uv \in G_\cH^{(j)}$, and this is not necessarily the case. However, since $D_2^{\x(uv)}(R_v,R_u) \subset D_1^{\x(uv)}(R_v,R_u)$, it is not a problem if $uv \in G_\cH^{(2)} = \tr({G_\cH})$ for some $u$ with $\sh(R_u) \le B/q$. The next lemma bounds the probability that there exists $uv \in G_\cH^{(1)}$ with $N_{G_{\cH}}(u)=\{v\}$ and $\sh(R_u) > B/q$.

\begin{lem}\label{lem:two:big:rectangles}
The probability that there exist two disjointly internally filled rectangles $S_1,S_2 \subset R$ with 
\begin{equation}\label{eq:two:big:rectangles}
\min\big\{ \sh(S_1), \sh(S_2) \big\} \ge \frac{B}{q}
\end{equation}
is at most $e^{-2/q + o(1/q)}$.
\end{lem}

This lemma is an almost immediate consequence of Holroyd's theorem and the van den Berg--Kesten lemma. However, for convenience (since the version we need is not explicitly stated in~\cite{Holroyd03}) we will deduce it from the following (very weak) consequence of~\cite[Proposition~15]{Gravner12}, which holds since $2\lambda = \pi^2 / 9 > 1$.

\pagebreak

\begin{prop}\label{prop:GHMresult}
Let $S \subset R$ be a rectangle with $\sh(S) \ge B/q$. Then
\[\Pr_p\big( I(S) \big) \le e^{-1/q}\,.\]
\end{prop}

\begin{proof}[Proof of Lemma~\ref{lem:two:big:rectangles}]
By the van den Berg--Kesten inequality and Proposition~\ref{prop:GHMresult}, the probability that two given rectangles $S_1$ and $S_2$, each with short side at least $B/q$, are disjointly internally filled is at most $e^{-2/q}$. By the union bound, it follows that the probability that two such disjointly internally filled rectangles exist is at most
\[\big( \lgs(R) \big)^8 e^{-2/q} = e^{-2/q + o(1/q)}\,,\]
as claimed. 
\end{proof}

Note that if there exists a vertex $u$ with $\sh(R_u) \ge B/q$ that is not in the trunk, 
then there must exist a split vertex above $u$ whose neighbours are labelled with disjointly internally filled rectangles $S_1$ and $S_2$ satisfying~\eqref{eq:two:big:rectangles}. Hence, by Lemma~\ref{lem:two:big:rectangles}, and recalling that $\lambda < 1$, we may assume that every vertex $u \in V(G_\cH)$ with $\sh(R_u) \ge B/q$ is in the trunk, and hence $uv \in \tr(\cH)$ whenever $N_{G_{\cH}}(u) = \{v\}$ and $\sh(R_u) > B/q$. We may therefore apply the inequality~\eqref{eq:cor:key:again} to bound the probability of the event $D_j^{\x(uv)}(R_v,R_u)$ for each $uv\in G_\cH^{(j)}$ with $N_{G_{\cH}}(u)=\{v\}$. Setting 
\[X(\cH) := \sum_{u \in L(\cH)} \phi(R_u)\,,\]
it follows from~\eqref{eq:basic:bound:again},~\eqref{eq:seeds:bound} and~\eqref{eq:cor:key:again}, and Lemmas~\ref{lem:no:long:thin:rectangle} and~\ref{lem:two:big:rectangles}, that the probability that $R$ is internally filled is bounded from above by $e^{-2/q + o(1/q)}$ plus
\[\sum_{\cH \in \cH^*_R} 3^{X(\cH)} e^{-X(\cH) g(\sqrt{q})} \prod_{N_{G_{\cH}}(u)=\{v\}} C^9 \left(\frac{\delta}{f(R_u)}\right)^{\|\x(uv)\|} \exp\Big( - Q(R_v,R_u) + 4 \phi(R_u) q \Big)\,,\]
where $\cH^*_R$ denotes the set of hierarchies $\cH \in \cH_R$ that contain no rectangle satisfying either~\eqref{eq:no:long:thin:rectangle} or~\eqref{eq:no:small:long:thin:rectangle}, and such that every vertex $u \in V(G_\cH)$ with $\sh(R_u) \ge B/q$ is in the trunk. By Definition~\ref{def:weight}, this is at most
\begin{equation}\label{eq:basic:bound:consequence}
\sum_{\cH \in \cH^*_R} w(\cH) \cdot C^{9 v(\cH)} \cdot 3^{X(\cH)} e^{-X(\cH) g(\sqrt{q})} \prod_{N_{G_{\cH}}(u)=\{v\}} \exp\Big( - Q(R_v,R_u) + 4 \phi(R_u) q \Big)\,.
\end{equation}
The rest of the proof of Theorem~\ref{thm:droplet:again} is just a careful analysis of~\eqref{eq:basic:bound:consequence}.  

\begin{proof}[Proof of Theorem~\ref{thm:droplet:again}]
As explained above, by Lemmas~\ref{lem:seeds},~\ref{lem:basic:bound},~\ref{lem:no:long:thin:rectangle} and~\ref{lem:two:big:rectangles}, and Corollary~\ref{cor:key}, to prove the theorem it will suffice to bound~\eqref{eq:basic:bound:consequence}. Let us set 
\[\Lambda(\cH) := w(\cH) \cdot C^{9 v(\cH)} \cdot 3^{X(\cH)} e^{-X(\cH) g(\sqrt{q})} \prod_{N_{G_{\cH}}(u)=\{v\}} \exp\Big( - Q(R_v,R_u) + 4 \phi(R_u) q \Big)\]
for each $\cH \in \cH^*_R$, and write $\cH_R^{(1)}$ for the set of $\cH \in \cH^*_R$ such that 
\begin{equation}\label{eq:normal:height:properties}
h(\cH) \le \frac{L_2}{\sqrt{q}} \qquad \text{and} \qquad \sum_{u \in \utr(\cH)} \phi(R_u) \le \frac{L_2}{q^{3/2}}\,,
\end{equation}
cf. Lemmas~\ref{lem:height:or:vertex} and~\ref{lem:upper:trunk:total:semiperimeter}. Let us note that this is the most important class of hierarchies, since it will turn out that the remaining hierarchies $\cH^*_R \setminus \cH_R^{(1)}$ contribute only smaller order terms to~\eqref{eq:basic:bound:consequence}. To slightly simplify the formulae below, let us write
\[\cJ(R) := \frac{2}{q} \int_0^{aq} g(z)\md z + (b - a) g(aq)\,,\]
where we recall that $a = \sh(R)$ and $b = \lgs(R)$.

\medskip
\noindent \textbf{Claim 1:} $\ds\sum_{\cH \in \cH_R^{(1)}} \Lambda(\cH) \le \exp\bigg( - \cJ(R) + \frac{L_5}{\sqrt{q}} \bigg) + e^{-2/q}$.

\begin{proof}[Proof of Claim~1] 
The proof is a fairly standard (if somewhat complicated) calculation, similar to, e.g.,~\cite{Gravner12}, the main new ingredient being the weighted counting of Lemma~\ref{lem:weighted:counting}. The first step is to deal with hierarchies with $X(\cH) > 1/q$, and to do so we will first show that 
\begin{equation}\label{eq:lambda:simple}
\Lambda(\cH) \le w(\cH) \cdot C^{10 v(\cH)} \cdot \exp\bigg( - \frac{X(\cH)}{5} \log \frac{1}{q}  + \frac{4 L_2}{\sqrt{q}} \bigg)
\end{equation}
for every $\cH \in \cH_R^{(1)}$. To see this, recall first that every vertex $u$ with $\sh(R_u) \ge B/q$ is in the trunk, and no rectangle that appears in $\cH$ satisfies~\eqref{eq:no:long:thin:rectangle}. It follows that $\sh(R_u) \le B/q$ and $\lgs(R_u) \le 3 e^{2B} / q$ for every $u \not\in \utr(\cH)$, and hence, by~\eqref{eq:normal:height:properties}, we have
\begin{equation}\label{eq:error:bound:claim}
\prod_{N_{G_{\cH}}(u)=\{v\}} \exp\big( 4 \phi(R_u) q \big) \le C^{v(\cH)} \exp\big( 4 L_2 / \sqrt{q} \big)\,,
\end{equation}
since $C = C(B) > 0$ was chosen sufficiently large. Next, observe that
\begin{equation}\label{eq:seeds:bound:claim}
3^{X(\cH)} e^{-X(\cH) g(\sqrt{q})} \le \exp\left( - \frac{X(\cH)}{4} \left( \log \frac{1}{3^4 q} - 4 q^{1/4} \right)\right) \le \exp\left( - \frac{X(\cH)}{5} \log \frac{1}{q} \right)\,,
\end{equation}
since $g(\sqrt{q}) \ge \log (q^{-1/4}) - q^{1/4}$, by~\eqref{obs:g3}. Noting that $Q(R_v,R_u) \ge 0$ for every $R_v \subset R_u$, since $g(z)$ is positive, and using~\eqref{eq:error:bound:claim} and~\eqref{eq:seeds:bound:claim}, we obtain~\eqref{eq:lambda:simple}, as claimed.

Now, recall that $m(\cH)$ denotes the number of large seeds in a hierarchy $\cH$, and that 
\begin{equation*}\label{eq:m:bounds:claim}
v(\cH) \le 2 \cdot h(\cH) \cdot m(\cH)
\end{equation*}
by Observation~\ref{obs:height:large:seeds}, and observe that therefore 
\begin{equation}\label{eq:X:bounds:claim}
X(\cH) \ge \frac{m(\cH)}{3\sqrt{q}} + \frac{\delta( s(\cH) - m(\cH)) }{q^{1/4}} \ge \frac{v(\cH)}{L_3} + \frac{s(\cH)}{q^{1/5}}\,,
\end{equation}
by Observation~\ref{obs:small:seeds} and~\eqref{eq:normal:height:properties}. We claim next that 
\begin{equation}\label{eq:X:large:done}
\sum_{\cH \in \cH_R^{(1)} :\, X(\cH) > 1/q} \Lambda(\cH) \, \le \, e^{-2/q}\,.
\end{equation}
To prove~\eqref{eq:X:large:done}, let us write
\[\cH_R^{(1)}(N,M) := \big\{ \cH \in \cH_R^{(1)} : v(\cH) = N, \, s(\cH) = M \big\}\,,\]
as in~\eqref{def:HNM}, and recall that, by Lemma~\ref{lem:weighted:counting},
\begin{equation}\label{eq:weighted:counting:repeat}
\sum_{\cH \in \cH^{(1)}_R(N,M)} w(\cH) \le \exp\Big( 16 \big( N + M \log \phi(R) \big) \Big)\,.
\end{equation}
Combining~\eqref{eq:lambda:simple} with~\eqref{eq:X:bounds:claim} and~\eqref{eq:weighted:counting:repeat}, it follows that
\begin{align*}
\sum_{\substack{\cH \in \cH_R^{(1)} \\ X(\cH) > 1/q}} \Lambda(\cH) & \, \le \, \sum_{N,M} \sum_{\substack{\cH \in \cH_R^{(1)}(N,M) \\ X(\cH) > 1/q}} w(\cH) \cdot C^{10 v(\cH)} \cdot \exp\bigg( - \frac{X(\cH)}{5} \log \frac{1}{q}  + \frac{4 L_2}{\sqrt{q}} \bigg)\\
& \, \le \, e^{-2/q} \sum_{N,M} \exp\big( - CN - Mq^{-1/5} \big) \sum_{\cH \in \cH_R^{(1)}(N,M)}
w(\cH)\\
& \, \le \, e^{-2/q} \sum_{N,M} \exp\big( - N - M \big) \, \le \, e^{-2/q},
\end{align*}
as claimed, where in the second step we used the bound $X(\cH) > 1/q \gg 1$. 

We will therefore assume from now on that 
\begin{equation}\label{eq:X:small} 
X(\cH) \le \frac{1}{q}\,.
\end{equation}
We next claim that 
\begin{equation}\label{eq:lambda:nextstep} 
\Lambda(\cH) \le w(\cH) \cdot C^{10 v(\cH)} \cdot \exp\bigg(  - \frac{U(S,R)}{q} - \frac{X(\cH)}{4} \log \frac{1}{q} + 13L_2 \cdot X(\cH) \bigg)\,.
\end{equation}
To prove this we repeat the proof of~\eqref{eq:lambda:simple}, being slightly less wasteful in~\eqref{eq:seeds:bound:claim}, and replacing the trivial bound $Q(S,R) \ge 0$ by a more complicated argument. To be more precise, recall that $U(R_v,R_u) \le q \cdot Q(R_v,R_u)$ for every $R_v \subset R_u$, by Lemma~\ref{lem:U:upbound}, and that therefore, by Lemma \ref{lem:pod:Holroyd}, there exists a pod $S$, with $\phi(S) \le X(\cH)$, such that
\begin{equation}\label{eq:Q:nontrivial:bound} 
\sum_{N_{G_{\cH}}(u)=\{v\}} Q(R_v,R_u) \ge \sum_{N_{G_{\cH}}(u)=\{v\}} \frac{U(R_v,R_u)}{q} \, \ge \, \frac{U(S,R)}{q} - 2s(\cH) g(\sqrt{q})\,.
\end{equation}
Now, recall from~\eqref{eq:X:bounds:claim} that we have $X(\cH) \ge s(\cH) \cdot q^{-1/5} \gg s(\cH) g(\sqrt{q})$, and note that $X(\cH) \ge 1/(3\sqrt{q})$, since every hierarchy in $\cH_R$ has at least one large seed. Hence, by~\eqref{eq:error:bound:claim},~\eqref{eq:seeds:bound:claim} and~\eqref{eq:Q:nontrivial:bound}, we obtain~\eqref{eq:lambda:nextstep}, as claimed.

Now, by Lemma~\ref{lem:Ulowerbound}, we have
\[\frac{U(S,R)}{q} \ge \, \frac{2}{q}\int_0^{aq} g(z) \md z + (b - a) g(aq) - \frac{\phi(S)}{2}\log\left(1 + \frac{1}{\phi(S)q}\right) - O\big( \phi(S) \big)\,,\]
since~\eqref{eq:X:small} implies\footnote{Recall that, by Lemma~\ref{lem:no:long:thin:rectangle}, we may assume that $a \ge B/q$.} that $\phi(S) \le X(\cH) \le 1/q \le a$, and hence
\[\frac{U(S,R)}{q} \ge \cJ(R) - \frac{X(\cH)}{2}\log\left(1+\frac{1}{X(\cH)q}\right) - O\big( X(\cH) \big)\,,\]
since the function $x \mapsto x\log\big( 1 + \frac{1}{x} \big)$ is increasing. Combining this with~\eqref{eq:lambda:nextstep}, recalling that $v(\cH) \le L_3 \cdot X(\cH)$, by~\eqref{eq:X:bounds:claim}, and noting that $1 + \frac{1}{X(\cH)q} \le \frac{2}{X(\cH)q}$, by~\eqref{eq:X:small}, we obtain
\begin{equation}\label{eq:one:more:lambda:bound} 
\Lambda(\cH) \le w(\cH) \cdot \exp\left( - \cJ(R) - \frac{X(\cH)}{2}\log\frac{X(\cH)\sqrt{q}}{L_4} \right)
\end{equation}
where $L_4 = C^{O(L_3)}$. Finally, observe that, by~\eqref{eq:X:bounds:claim} and~\eqref{eq:weighted:counting:repeat}, we have 
\[\sum_{\cH \in \cH_R^{(1)} : \, X(\cH) = x} w(\cH) \le \exp\big( O\big( L_3 \cdot x \big) \big)\]
for any $x \in \N$. It follows that
\begin{multline*}
\sum_{\substack{\cH \in \cH_R^{(1)} \\ X(\cH) \le 1/q}} \Lambda(\cH) \, \le \, e^{-\cJ(R)} \sum_{x = 1/(3\sqrt{q})}^{1/q} \exp\left(-\frac{x}{2}\log\frac{x\sqrt{q}}{L_4}\right) \sum_{\substack{\cH \in \cH_R^{(1)} \\ X(\cH) = x}}w(\cH)\\
\le \, e^{-\cJ(R)} \sum_{x = 1/(3\sqrt{q})}^{1/q} \exp\left(-\frac{x}{2}\log\frac{x\sqrt{q}}{L_4} +  O\big( L_3 \cdot x \big) \right)\, \le \, \exp\left( - \cJ(R) + \frac{L_5}{\sqrt{q}} \right),
\end{multline*}
where $L_5 = L_4 \cdot e^{O(L_3)}$, since the summand decreases super-exponentially quickly once $x\sqrt{q}$ is larger than this. This completes the proof of Claim~1. 
\end{proof}

If $\cH \in \cH^*_R \setminus \cH_R^{(1)}$ then, by Lemmas~\ref{lem:height:or:vertex} and~\ref{lem:upper:trunk:total:semiperimeter}, there exists a vertex $u \in V(G_\cH)$ satisfying either~\eqref{eq:weird:vertex:lower} or~\eqref{eq:weird:vertex:upper}. The rest of the proof consists of bounding the contribution to~\eqref{eq:basic:bound:consequence} of hierarchies containing such a vertex. In order to simplify the argument, it will be convenient to first (in Claims~2 and~3) deal with those hierarchies in which either $v(\cH)$ or $X(\cH)$ is unusually large. We then (in Claims~4 and~5) consider the remaining hierarchies with an `abnormal' vertex, i.e., one satisfying either~\eqref{eq:weird:vertex:lower} or~\eqref{eq:weird:vertex:upper}. 

We begin by considering hierarchies with unusually many vertices. Let us write $\cH_R^{(2)}$ for the set of $\cH \in \cH^*_R \setminus \cH_R^{(1)}$ such that
\begin{equation}\label{eq:vertex:number:general:bound}
v(\cH) \ge 8 \cdot s(\cH) + \frac{4L_1}{q^{3/4}}\,.
\end{equation}
For such hierarchies we will prove the following stronger bound. 

\bigskip
\noindent \textbf{Claim 2:} $\ds\sum_{\cH \in \cH_R^{(2)}} \Lambda(\cH) \le e^{-2/q}$. 
\medskip

\begin{proof}[Proof of Claim~2] 
The first step is to prove the following bounds, 
\begin{equation}\label{eq:weak:height:properties}
|\utr(\cH)| \le \frac{L_1}{\sqrt{q}} \log \frac{1}{q} \qquad \text{and} \qquad \sum_{u \in \utr(\cH)} \phi(R_u) \le \frac{L_1}{q^{3/2}} \bigg( \log \frac{1}{q} \bigg)^2 
\end{equation}
which replace those in~\eqref{eq:normal:height:properties}, and hold for any $\cH \in \cH_R$. Both follow immediately from the upper bound on $b = \lgs(R)$ in~\eqref{eq:droplet:conditions:again}, and the observation that in two consecutive steps of $\utr(\cH)$, the semi-perimeter of the corresponding rectangles grows by at least $\delta / \sqrt{q}$. 

Now, for each $\cH \in \cH_R^{(2)}$, consider the set $\cY(\cH)$ of edges $uv \in G_\cH$ such that 
\[\sh(R_u) \le \frac{B}{q}, \qquad N_{G_{\cH}}(u) = \{v\} \qquad \text{and} \qquad |N_{G_{\cH}}(v)| = 1\,.\]
We claim that
\begin{equation}\label{eq:Y:lower:bound}
|\cY(\cH)| \, \ge \, v(\cH) - 4s(\cH) - |\utr(\cH)| \, \ge \, \frac{v(\cH)}{2} + \frac{L_1}{q^{3/4}}\,.
\end{equation}
To see this, recall that $\cH$ has $s(\cH)$ seeds and $s(\cH) - 1$ split vertices, and so there are at most $4s(\cH) - 2$ vertices $u \in V(G_\cH)$ that are either seeds, or split-vertices, or have a single out-neighbour that is a seed or a split vertex. Moreover, $\cH \in \cH^*_R$ implies that every vertex $u \in V(G_\cH)$ with $\sh(R_u) \ge B/q$ is in the trunk. The second inequality follows from~\eqref{eq:vertex:number:general:bound} and~\eqref{eq:weak:height:properties}. 

We next claim that
\begin{equation}\label{eq:another:lambda:bound} 
\Lambda(\cH) \le w(\cH) \cdot C^{10 v(\cH)} \cdot \exp\bigg( \frac{4 L_1}{\sqrt{q}} \bigg( \log \frac{1}{q} \bigg)^2 - \frac{\delta^2 |\cY(\cH)|}{q^{1/4}} \bigg)\,.
\end{equation}
The proof of this is similar to that of~\eqref{eq:lambda:simple}. Indeed, we obtain a slightly weaker bound in place of~\eqref{eq:error:bound:claim} by using~\eqref{eq:weak:height:properties} instead of~\eqref{eq:normal:height:properties}, and~\eqref{eq:seeds:bound:claim} still holds, and the right-hand side is at most $1$. Moreover, $\cH \in \cH^*_R$ implies that $\lgs(R_u) \le 3e^{2B} / q$ for each edge $uv \in \cY(\cH)$, and therefore
\begin{equation*}\label{eq:Ysumbound}
Q(R_u,R_v) \ge \frac{\delta g(3e^{2B})}{q^{1/4}}  \ge \frac{\delta^2}{q^{1/4}}\,,
\end{equation*}
for each such edge, since $d(R_u,R_v) \ge f(R_u)\ge \delta \sqrt{\sh(R_u)} \ge \delta q^{-1/4}$, by Definition~\ref{def:hier:good} and~\eqref{def:f}, and since $\delta = \delta(B)$ was chosen sufficiently small. Plugging these bounds into the definition of $\Lambda(\cH)$, we obtain~\eqref{eq:another:lambda:bound}. 

By~\eqref{eq:Y:lower:bound}, it follows that
\[\Lambda(\cH) \le w(\cH) \cdot \exp\bigg( - \frac{\delta^3 |\cY(\cH)|}{q^{1/4}} \bigg)\,,\]
and hence, by~\eqref{eq:Y:lower:bound} and Lemma~\ref{lem:weighted:counting}, and since $\phi(R) \le 1 / q^{2}$, we obtain
\begin{align*}
\sum_{\cH \in \cH_R^{(2)}} \Lambda(\cH) & \, \le \, \sum_{y \ge L_1 q^{-3/4}} \exp\left( - \frac{\delta^3 y}{q^{1/4}} \right) \sum_{M = 1}^{2y} \sum_{N = M}^{2y} \sum_{\substack{\cH \in \cH_R^{(2)}(N,M) \\ |\cY(\cH)| = y}} w(\cH)\\
& \, \le \, \sum_{y \ge L_1 q^{-3/4}} \exp\left( - \frac{\delta^3 y}{q^{1/4}} \right)\sum_{M = 1}^{2y} \sum_{N = M}^{2y} \exp\bigg( 16 \Big( N + 2M \log\frac{1}{q} \Big) \bigg)\\
& \, \le \,\sum_{y \ge L_1 q^{-3/4}} \exp\left( - \frac{\delta^4 y}{q^{1/4}} \right) \, \le \, e^{-2/q},
\end{align*}
as required.
\end{proof}

We will next deal with those hierarchies for which $X(\cH)$ is unusually large. To be precise, let us define $\cH_R^{(3)}$ to be the set of $\cH \in \cH^*_R \setminus \big( \cH_R^{(1)} \cup \cH_R^{(2)} \big)$ such that
\begin{equation}\label{eq:X:very:big}
X(\cH) \ge \frac{1}{L_2 q^{3/4}}\,.
\end{equation}
For this class of hierarchies we will prove the following bound. 

\bigskip
\noindent \textbf{Claim 3:} $\ds\sum_{\cH \in \cH_R^{(3)}} \Lambda(\cH) \le \exp\left( - \cJ(R) - \frac{1}{q^{3/4}} \right) + e^{-2/q}$.
\medskip

\begin{proof}[Proof of Claim~3] 
Let $\cH \in \cH_R^{(3)}$, and observe that
\begin{equation}\label{eq:more:weak:properties}
v(\cH) \le 8 \cdot s(\cH) + \frac{4L_1}{q^{3/4}} \qquad \text{and} \qquad \sum_{u \in \utr(\cH)} \phi(R_u) \le \frac{L_1}{q^{3/2}} \bigg( \log \frac{1}{q} \bigg)^2\,, 
\end{equation}
where the first inequality holds since $\cH \not\in \cH_R^{(2)}$, and the second holds for any $\cH \in \cH_R$, by~\eqref{eq:weak:height:properties}. We will repeat the proof of Claim~1, using the bounds~\eqref{eq:more:weak:properties} instead of~\eqref{eq:normal:height:properties}. 

Indeed, note (cf.~\eqref{eq:error:bound:claim}) that
\begin{equation}\label{eq:error:bound:claim:weak}
\prod_{N_{G_{\cH}}(u)=\{v\}} \exp\big( 4 \phi(R_u) q \big) \le C^{v(\cH)} \exp\bigg( \frac{4L_1}{\sqrt{q}} \bigg( \log \frac{1}{q} \bigg)^2 \bigg)\,,
\end{equation}
and hence, using~\eqref{eq:seeds:bound:claim}, we obtain
\begin{equation*}\label{eq:lambda:simple:weak}
\Lambda(\cH) \le w(\cH) \cdot C^{10 v(\cH)} \cdot \exp\bigg( - \frac{X(\cH)}{5} \log \frac{1}{q}  + \frac{4 L_1}{\sqrt{q}} \bigg( \log \frac{1}{q} \bigg)^2 \bigg)\,.
\end{equation*}
Now, note that, by Observation~\ref{obs:small:seeds} and the bounds~\eqref{eq:X:very:big} and~\eqref{eq:more:weak:properties}, we have
\begin{equation*}\label{eq:X:bounds:claim:weak}
X(\cH) \ge \frac{1}{L_2} \cdot \max\bigg\{ \frac{s(\cH)}{q^{1/4}}, \, \frac{1}{q^{3/4}} \bigg\} \ge \frac{v(\cH)}{L_3}\,,
\end{equation*}
It follows, exactly as in the proof of Claim~1 (cf. the proof of~\eqref{eq:X:large:done}), that 
\begin{equation*}\label{eq:X:large:done:again}
\sum_{\cH \in \cH_R^{(3)} :\, X(\cH) > 1/q} \Lambda(\cH) \, \le \, e^{-2/q}\,.
\end{equation*}
We will therefore assume from now on that $X(\cH) \le 1/q$.

We now simply repeat the remainder of the proof of Claim~1, using the bounds $1/(L_2 q^{3/4}) \le X(\cH) \le 1/q$, to obtain
\[\Lambda(\cH) \le w(\cH) \cdot \exp\left( - \cJ(R) - \frac{X(\cH)}{2}\log\frac{X(\cH)\sqrt{q}}{L_4} \right)\]
for each $\cH \in \cH_R^{(3)}$, and hence
\begin{multline*}
\sum_{\substack{\cH \in \cH_R^{(3)} \\ X(\cH) \le 1/q}} \Lambda(\cH) \, \le \, e^{-\cJ(R)} \sum_{x = 1/(L_2 q^{3/4})}^{1/q} \exp\left(-\frac{x}{2}\log\frac{x\sqrt{q}}{L_4}\right) \sum_{\substack{\cH \in \cH_R^{(3)} \\ X(\cH) = x}}w(\cH)\\
\le \, e^{-\cJ(R)} \sum_{x = 1/(L_2 q^{3/4})}^{1/q} \exp\left(-\frac{x}{2}\log\frac{x\sqrt{q}}{L_4} +  O\big( L_3 \cdot x \big) \right) \, \le \, \exp\left( - \cJ(R) - \frac{1}{q^{3/4}} \right),
\end{multline*}
as claimed.
\end{proof}

We are now ready to deal with those hierarchies that travel `far from the diagonal', i.e., that  contain a vertex $u$ satisfying either~\eqref{eq:weird:vertex:lower} or~\eqref{eq:weird:vertex:upper}. We will first consider the (easier) case in which $u$ is in the upper trunk, i.e., 
\begin{equation}\label{eq:weird:large:again}
\sh(R_u) \ge \frac{B}{q} \qquad \text{and} \qquad \lgs(R_u) \ge 4 \cdot \sh(R_u)\,.
\end{equation}
Let $\cH_R^{(4)}$ be the set of hierarchies $\cH \in \cH^*_R \setminus \bigcup_{i = 1}^3 \cH_R^{(i)}$ containing a vertex $u$ such that~\eqref{eq:weird:large:again} holds. For these hierarchies we will prove the following bound. 

\pagebreak
\noindent \textbf{Claim 4:} $\ds\sum_{\cH \in \cH_R^{(4)}} \Lambda(\cH) \le \exp\left( - \frac{2\lambda}{q} - \frac{2}{q^{3/4}} \right)$. 
\medskip

\begin{proof}[Proof of Claim~4] 
Given $\cH \in \cH_R^{(4)}$ and $u \in V(\cH)$ satisfying~\eqref{eq:weird:large:again}, we claim that
\begin{equation}\label{eq:lambda:u:upper}
\Lambda(\cH) \le w(\cH) \cdot \exp\left( - \cJ(R_u) + \frac{L_3}{q^{3/4}} \right)\,.
\end{equation}
To prove this, we will repeat the proof of Claim~1, with some minor changes. First, note that~\eqref{eq:X:small} holds, and moreover 
\begin{equation}\label{eq:X:bounds:final:two:claims}
\frac{s(\cH)}{L_1 q^{1/4}} \le X(\cH) \le \frac{1}{L_2 q^{3/4}}\,,
\end{equation}
the first holding by Observation~\ref{obs:small:seeds}, and the second since $\cH \not\in \cH_R^{(3)}$. Now, applying Lemma~\ref{lem:pod:Holroyd} to $\cH_u$, the sub-hierarchy of $\cH$ rooted at $u$, and using~\eqref{eq:error:bound:claim:weak} instead of~\eqref{eq:error:bound:claim}, we obtain (cf. the proof of~\eqref{eq:lambda:nextstep})
\begin{equation*}\label{eq:lambda:nextstep:weaker} 
\Lambda(\cH) \le w(\cH) \cdot C^{10 v(\cH)} \cdot \exp\bigg(  - \frac{U(S,R_u)}{q} - \frac{X(\cH)}{5} \log \frac{1}{q} + \frac{4L_1}{\sqrt{q}} \left( \log \frac{1}{q} \right)^2 \bigg)
\end{equation*}
for some pod $S$ with $\phi(S) \le X(\cH_u) \le X(\cH)$. Using~\eqref{eq:more:weak:properties} again (this time to bound $v(\cH)$ from above), and continuing to follow the proof of Claim~1, we obtain
\begin{equation*}\label{eq:one:more:lambda:bound:modified} 
\Lambda(\cH) \le w(\cH) \cdot \exp\left( - \cJ(R_u) - \frac{X(\cH)}{2}\log\frac{X(\cH)q^{3/5}}{L_4} + \frac{L_2}{q^{3/4}} \right)
\end{equation*}
instead of~\eqref{eq:one:more:lambda:bound}, which implies~\eqref{eq:lambda:u:upper}.  

Now, let $c = \sh(R_u)$, and observe that 
\[\cJ(R_u) \, \ge \, \frac{2}{q} \int_0^{cq} g(z)\md z + 3c g(cq) \, \ge \, \frac{2}{q} \int_0^\infty g(z)\md z + 2c g(cq)\,,\]
by~\eqref{eq:weird:large:again} and~\eqref{obs:g7}, and since $c \ge B/q$, and that
\[cg(cq) \, \ge \, \frac{c}{2} \cdot e^{-2cq} \, \ge \, \frac{1}{4 q^{3/4}} \log \frac{1}{q}\,,\]
where the first inequality holds since $c \ge B/q$, and the second since $4c \le \lgs(R) \le (1/(2q)) \log(1/q)$. Combining this with~\eqref{eq:lambda:u:upper}, it follows that
\[\Lambda(\cH) \le w(\cH) \cdot \exp\left( - \frac{2\lambda}{q} - \frac{L_3}{q^{3/4}} \right)\,.\]
Hence, recalling that $v(\cH) \le 5L_1 / q^{3/4}$ and $s(\cH) \le q^{-1/2}$ for every $\cH \in \cH_R^{(4)}$, by~\eqref{eq:more:weak:properties} and~\eqref{eq:X:bounds:final:two:claims}, and applying Lemma~\ref{lem:weighted:counting}, we obtain
\begin{multline*}
\sum_{\cH \in \cH_R^{(4)}} \Lambda(\cH) \, \le \, \exp\left( - \frac{2\lambda}{q} - \frac{L_3}{q^{3/4}} \right) \sum_{M = 1}^{1/q^{1/2}} \sum_{N = M}^{5L_1 / q^{3/4}} \sum_{\cH \in \cH_R^{(4)}(N,M)}w(\cH)\\
\le \, \exp\left( - \frac{2\lambda}{q} - \frac{L_3}{q^{3/4}} \right) \sum_{M = 1}^{1/q^{1/2}} \sum_{N = M}^{5L_1 / q^{3/4}} \exp\left( \frac{L_2}{q^{3/4}} \right) \, \le \, \exp\left( - \frac{2\lambda}{q} - \frac{2}{q^{3/4}} \right),
\end{multline*}
as claimed.
\end{proof}

Finally, we come to most technically challenging family of hierarchies: those which contain a vertex $u \in V(G_\cH)$ such that
\begin{equation}\label{eq:weird:small:again}
\sh(R_u) \le \frac{B}{q} \qquad \text{and} \qquad \lgs(R_u) \ge 2L_1 \cdot \sh(R_u)\,.
\end{equation}
Let $\cH_R^{(5)}$ denote the set of hierarchies $\cH \in \cH^*_R \setminus \bigcup_{i = 1}^4 \cH_R^{(i)}$ containing a vertex $u$ such that~\eqref{eq:weird:small:again} holds. For this final class of hierarchies we will prove the following bound. 

\medskip
\noindent \textbf{Claim 5:} $\ds\sum_{\cH \in \cH_R^{(5)}} \Lambda(\cH) \le \exp\left( - \cJ(R) - \frac{1}{q^{3/4}} \right)$. 

\begin{proof}[Proof of Claim~5] 
Given a hierarchy $\cH \in \cH_R^{(5)}$, let $u \in V(\cH)$ be a vertex satisfying~\eqref{eq:weird:small:again} with $\lgs(R_u)$ maximal, and set $c = \sh(R_u)$ and $d = \lgs(R_u)$. We will prove that
\begin{equation}\label{eq:final:claim:lambda:bound}
\Lambda(\cH) \le w(\cH) \cdot \exp\left( - \cJ(R) - 3Cd + X(\cH) \log \frac{1}{X(\cH) q^{3/4}} \right)\,,
\end{equation}
from which the claim will follow easily, using Lemma~\ref{lem:weighted:counting}.

In order to prove~\eqref{eq:final:claim:lambda:bound}, we will need various bounds on $c$, $d$, $h(\cH)$, $v(\cH)$ and $X(\cH)$. Note first that $\cH$ does not contain a vertex satisfying~\eqref{eq:weird:large:again} since $\cH \not\in \cH_R^{(4)}$. It follows, by Lemmas~\ref{lem:weird:vertex:height:bound} and~\ref{lem:upper:trunk:total:semiperimeter}, and since $\cH \not\in \cH_R^{(1)}$, that
\begin{equation}\label{eq:final:claim:height:bounds}
\frac{L_2}{\sqrt{q}} \le h(\cH) \le L_1 q^{1/4} d\,.
\end{equation}
Indeed, the lower bound holds since $\cH \not\in \cH_R^{(1)}$ implies that one of the inequalities in~\eqref{eq:normal:height:properties} must fail to hold, and by Lemma~\ref{lem:upper:trunk:total:semiperimeter}, it must be the bound on $h(\cH)$. The upper bound then follows by Lemma~\ref{lem:weird:vertex:height:bound}, and by our choice of $u$ (i.e., with $\lgs(R_u)$ maximal).

We next claim that
\begin{equation}\label{eq:dbounds}
v(\cH) \le d, \qquad c \le \frac{1}{q} \qquad \text{and} \qquad \frac{1}{q^{3/4}} \le d \le \frac{B}{2q}\,.
\end{equation}
Indeed, the lower bound on $d$ follows immediately from~\eqref{eq:final:claim:height:bounds}, since $L_1 \le L_2$. To prove the other bounds, recall first that (since $\cH \in \cH^*_R$) the rectangle $R_u$ does not satisfy~\eqref{eq:no:long:thin:rectangle} or~\eqref{eq:no:small:long:thin:rectangle}. Since $c \le B/q$, by~\eqref{eq:weird:small:again}, it follows that $d \le 3e^{2B} / q$, and hence
\[c \, \le \, \frac{d}{L_1} \, \le \, \frac{3e^{2B}}{L_1q} \, \le \, \frac{1}{q}\,.\]
Now, since $R_u$ does not satisfy~\eqref{eq:no:small:long:thin:rectangle}, it follows that $d \le B/(2q)$, as claimed. Finally, to prove the bound on $v(\cH)$, recall that 
\begin{equation}\label{eq:X:bounds:very:final:claim}
\frac{m(\cH)}{3\sqrt{q}} \, \le \, X(\cH) \, \le \, \frac{1}{L_2 q^{3/4}}\,,
\end{equation}
by the definition of large seeds, and since $\cH \not\in \cH_R^{(3)}$. Hence, by~\eqref{eq:final:claim:height:bounds} and Observation~\ref{obs:height:large:seeds}, we obtain
\[v(\cH) \, \le \, 2 \cdot h(\cH) \cdot m(\cH) \, \le \, 6L_1 q^{3/4} d \cdot X(\cH) \, \le \, d\,,\]
as claimed. 

We now apply Lemma~\ref{lem:pods} to obtain two pods $S_1 \subset R_u$ and $R_u \subset S_2 \subset R$, such that
\begin{equation}\label{eq:pods:semiperimeter}
\phi(S_1) + \phi(S_2) - \phi(R_u) \le X(\cH)
\end{equation}
and
\[\sum_{N_{G_\cH}(v) = \{w\}} U(R_w,R_v) \ge U(S_1,R_u) + U(S_2,R) - 2 s(\cH) q g(\sqrt{q})\,.\]
Let $S_1 \subset S \subset S_2$ be a rectangle with 
\begin{equation}\label{def:S:pod:of:pods}
\dim(S) = \dim(S_1) + \dim(S_2) - \dim(R_u)\,,
\end{equation}
so $\phi(S) \le X(\cH)$, by~\eqref{eq:pods:semiperimeter}, and moreover $U(S_1,R_u) \ge U(S,S_2)$, since $g(z)$ is decreasing. 

Recalling that $U(R_w,R_v) \le q \cdot Q(R_w,R_v)$, by Lemma~\ref{lem:U:upbound}, and that $g(\sqrt{q}) \le \log(1/q)$, by~\eqref{obs:g3}, it follows that
\[\sum_{N_{G_\cH}(v) = \{w\}} Q(R_w,R_v) \ge \frac{1}{q} \Big( U(S,S_2) + U(S_2,R) \Big) - 2 s(\cH) \log \frac{1}{q}\,.\]
Hence, using~\eqref{eq:seeds:bound:claim},~\eqref{eq:error:bound:claim:weak} and~\eqref{eq:X:bounds:final:two:claims}, we obtain
\begin{equation}\label{eq:penultimate:lambda:bound}
\Lambda(\cH) \le w(\cH) \cdot C^{10 v(\cH)} \exp\left( - \frac{1}{q} \Big( U(S,S_2) + U(S_2,R) \Big) - \frac{X(\cH)}{4} \log \frac{1}{q} + \frac{1}{q^{3/4}} \right)\,.
\end{equation}
It only remains to bound $U(S,S_2)$ and $U(S_2,R)$; controlling $U(S,S_2)$ will take some work, but we obtain a suitable bound on $U(S_2,R)$ simply by applying Lemma~\ref{lem:diagonal}. Indeed, by~\eqref{eq:dbounds},~\eqref{eq:X:bounds:very:final:claim} and~\eqref{def:S:pod:of:pods}, we have
\begin{equation}\label{eq:dims:of:Stwo}
\lgs(S_2) \, \le \, \phi(S) + \lgs(R_u) \, \le \, X(\cH) + d \, \le \, \frac{1}{q^{3/4}} + \frac{B}{2q} \, \le \, a \, = \, \sh(R)\,,
\end{equation} 
and therefore, setting $s_2 = \sh(S_2)$ and $t_2 = \lgs(S_2)$, we may apply Lemma~\ref{lem:diagonal}, which gives
\begin{equation}\label{eq:pod:to:top:bound}
\frac{U(S_2,R)}{q} \, \ge \, (t_2 - s_2) g(t_2q) + \frac{2}{q} \int^{aq}_{t_2q} g(z) \md z + (b - a) g(aq)\,.
\end{equation}
As noted above, we will have to work harder to obtain a suitable bound on $U(S,S_2)$; in particular, the bound we obtain will depend on whether or not $\lgs(S) \le \sh(S_2)$. 

\bigskip
\noindent Case 1: $\lgs(S) \le \sh(S_2)$.
\bigskip

This case is also straightforward, since we may apply Lemma~\ref{lem:Ulowerbound}, which gives
\begin{equation}\label{eq:pod:to:pod:bound}
\frac{U(S,S_2)}{q} \, \ge \, \frac{2}{q} \int_0^{s_2q} g(z) \md z + (t_2 - s_2) g(s_2q) - \frac{X(\cH)}{2} \log\frac{2}{X(\cH) q} - O\big( X(\cH) \big)\,,
\end{equation}
where we used the inequalities $\phi(S) \le X(\cH) \le 1/q$. To show that this is sufficient to deduce~\eqref{eq:final:claim:lambda:bound}, we will use Lemma~\ref{lem:leaving:the:diagonal}. Indeed, observe that
\[L_1 \cdot \sh(S_2) \le L_1 \big( c + X(\cH) \big) \le d \le \lgs(S_2) \le \frac{B}{q}\,,\]
where the first inequality follows from~\eqref{eq:pods:semiperimeter} (since $R_u \subset S_2$), the second follows since $2L_1 \cdot c \le d$, by~\eqref{eq:weird:small:again}, and $L_2 \cdot X(\cH) \le q^{-3/4} \le d$, by~\eqref{eq:dbounds} and~\eqref{eq:X:bounds:very:final:claim}, the third since $R_u \subset S_2$, and the last by~\eqref{eq:dims:of:Stwo}. Since $d \le t_2$ it follows, by Lemma~\ref{lem:leaving:the:diagonal}, that
\[\frac{2}{q} \int_{s_2q}^{t_2q} g(z) \md z \, \le \, (t_2 - s_2) \big( g(s_2q) + g(t_2q) \big) - 4Cd\,.\]
Combining this with~\eqref{eq:pod:to:top:bound} and~\eqref{eq:pod:to:pod:bound}, it follows that
\[\frac{1}{q} \Big( U(S,S_2) + U(S_2,R) \Big) \ge \cJ(R) + 4Cd - \frac{X(\cH)}{2} \log\frac{2}{X(\cH) q} - O\big( X(\cH) \big)\,.\]
Hence, by~\eqref{eq:penultimate:lambda:bound}, and recalling from~\eqref{eq:dbounds} that $d \ge \max\{ v(\cH), q^{-3/4} \} \ge X(\cH)$, we obtain
\[\Lambda(\cH) \le w(\cH) \cdot \exp\left( - \cJ(R) - 3Cd + \frac{X(\cH)}{2} \log \frac{1}{X(\cH) \sqrt{q}} \right)\,,\]
which is slightly stronger than~\eqref{eq:final:claim:lambda:bound}.

\medskip
\noindent Case 2: $\lgs(S) > \sh(S_2)$.
\medskip

This case is a bit more tricky, as the easiest path from $S$ to $S_2$ does not reach the diagonal, see Figure~\ref{fig:case2}. As a consequence, we cannot apply Lemma~\ref{lem:Ulowerbound} directly to bound $U(S,S_2)$, nor can we apply Lemma~\ref{lem:leaving:the:diagonal} directly to the dimensions of $S_2$. Instead, we observe that, setting $t := \lgs(S)$, we have
\[\frac{U(S,S_2)}{q} \ge ( t_2 - t ) g (s_2 q) \ge ( t_2 - t ) g(tq)\,,\]
by Lemma~\ref{lem:offdiagonal}, and since $g(z)$ is decreasing. Combining this with~\eqref{eq:pod:to:top:bound}, we obtain
\[\frac{1}{q} \Big( U(S,S_2) + U(S_2,R) \Big) \ge \cJ(R) - \frac{2}{q} \int_0^{t_2q} g(z) \md z + (t_2 - s_2) g(t_2q) + ( t_2 - t ) g(tq)\,.\]

\begin{figure}
\floatbox[{\capbeside\thisfloatsetup{capbesideposition={right,top}}}]{figure}[\FBwidth]
{\begin{tikzpicture}[line cap=round,line join=round,>=triangle 45,x=0.36cm,y=0.4cm]
\clip(-4.1,-0.1) rectangle (17.4,15);
\fill[line width=0pt,fill=black,fill opacity=0.1] (3,3) -- (10,3) -- (10,10) -- cycle;
\draw [line width=0.4pt,domain=0.0:17.0] plot(\x,{(-0--2*\x)/2});
\draw [line width=1.5pt] (3,1)-- (3,2);
\draw [line width=1.5pt] (3,2)-- (10,2);
\draw [line width=1.5pt] (10,2)-- (10,10);
\draw [line width=1.5pt] (10,10)-- (14,14);
\draw [line width=1.5pt] (14,14)-- (16,14);
\draw [line width=1.5pt,dash pattern=on 2pt off 2pt] (3,3)-- (10,3);
\draw [line width=1.5pt,dash pattern=on 2pt off 2pt] (10,10)-- (3,3);
\begin{scriptsize}
\fill [color=black] (3,1) circle (1.5pt);
\draw[color=black] (2.3,1) node {$S$};
\fill [color=black] (3,2) circle (1.5pt);
\fill [color=black] (10,2) circle (1.5pt);
\draw[color=black] (10.9,1.9) node {$S_2$};
\fill [color=black] (10,10) circle (1.5pt);
\draw[color=black] (6.1,10.6) node {$(\lgs(S_2),\lgs(S_2))$};
\fill [color=black] (14,14) circle (1.5pt);
\draw[color=black] (12.7,14.5) node {$(a,a)$};
\fill [color=black] (16,14) circle (1.5pt);
\draw[color=black] (16,13) node {$(b,a)$};
\fill [color=black] (3,3) circle (1.5pt);
\draw[color=black] (-0.5,3.7) node {$(\lgs(S),\lgs(S))$};
\fill [color=black] (10,3) circle (1.5pt);
\draw[color=black] (14,3.5) node {$(\lgs(S_2),\lgs(S))$};
\fill [color=black] (0.1,0.1) circle (1.5pt);
\end{scriptsize}
\end{tikzpicture}}
{\caption{The easiest path via which the rectangle $S$ can grow first to $S_2$, and then to $R$, in Case~2 of Claim~5. In the proof, the lower two thick segments are replaced with the lower dashed one, and Lemma~\ref{lem:leaving:the:diagonal} is applied to the shaded triangle.}
\label{fig:case2}}
\end{figure}

We will now apply Lemma~\ref{lem:leaving:the:diagonal} to the pair $(t,t_2)$. Indeed, we have
\[L_1 \cdot \lgs(S) \le L_1 \cdot X(\cH) \le d \le \lgs(S_2) \le \frac{B}{q}\,,\]
where the first inequality follows since $\phi(S) \le X(\cH)$, and the others follow as in Case~1. By Lemma~\ref{lem:leaving:the:diagonal}, and since $d \le t_2$, it follows that
\[\frac{2}{q} \int_{tq}^{t_2q} g(z) \md z \, \le \, (t_2 - t) \big( g(tq) + g(t_2q) \big) - 4Cd\,.\]
Note also that, by integrating~\eqref{obs:g3}, we have
\[\frac{2}{q} \int_{0}^{tq} g(z) \md z \, \le \, t \log \frac{1}{tq} + O(t)\,.\]
Hence, recalling that $s_2 < t \le \phi(S) \le X(\cH)$, it follows that
\[\frac{1}{q} \Big( U(S,S_2) + U(S_2,R) \Big) \ge \cJ(R) + 4Cd - X(\cH) \log \frac{1}{X(\cH)q} - O\big( X(\cH) \big)\,.\]
Hence, by~\eqref{eq:penultimate:lambda:bound}, and since $d \ge \max\{ v(\cH), q^{-3/4} \} \ge X(\cH)$, by~\eqref{eq:dbounds}, we obtain~\eqref{eq:final:claim:lambda:bound}.

It now only remains to deduce the claim from~\eqref{eq:final:claim:lambda:bound}; we do so using Lemma~\ref{lem:weighted:counting}. Indeed, recalling that $q^{-1/5} s(\cH) \le X(\cH) \le 1 / (L_2 q^{3/4}) \le \max\{ v(\cH), q^{-3/4} \} \le d \le B/q$, by~\eqref{eq:X:bounds:final:two:claims} and~\eqref{eq:dbounds}, we obtain
\begin{multline*}
\sum_{\cH \in \cH_R^{(5)}} \Lambda(\cH) \, \le \, e^{-\cJ(R)} \sum_{x = 1}^{1 / (L_2 q^{3/4})} \bigg( \frac{1}{q^{3/4} x} \bigg)^x \sum_{d = q^{-3/4}}^{B/q} e^{-3Cd} \sum_{M = 1}^{q^{1/5} d} \sum_{N = M}^d \sum_{\cH \in \cH_R^{(5)}(N,M)} w(\cH)\\
\le \, e^{-\cJ(R) + q^{-3/4}} \sum_{d = q^{-3/4}}^{B/q} e^{-3Cd} \sum_{M = 1}^{q^{1/5} d} \sum_{N = M}^d \exp\Big( O \big( N + M \log(1/q) \big) \Big)\\
\le \, e^{-\cJ(R) + q^{-3/4}} \sum_{d = q^{-3/4}}^{B/q} e^{-2Cd} \, \le \, \exp\left( - \cJ(R) - \frac{1}{q^{3/4}} \right),
\end{multline*}
as required. This concludes the proof of Claim~5.
\end{proof}

Now, combining Claims~1--5, it follows that
\[\sum_{\cH \in \cH^*_R} \Lambda(\cH) \, \le \, 3 \cdot \exp\left( - \min\bigg\{ \frac{2\lambda}{q} + \frac{2}{q^{3/4}}, \, \cJ(R) - \frac{L_5}{\sqrt{q}} \bigg\} \right)\,.\]
As was observed before the proof (see the discussion leading up to~\eqref{eq:basic:bound:consequence}), this completes the proof of Theorem~\ref{thm:droplet:again}.
\end{proof}

We are finally ready to deduce Theorem~\ref{thm:sharp} from Theorem~\ref{thm:droplet:again}.

\begin{proof}[Proof of Theorem~\ref{thm:sharp}] 
Recall that the upper bound was proved in~\cite{Gravner08}; we will prove the lower bound. Let $n \in \N$ be sufficiently large and set
\[q := \frac{\lambda}{\log n} - \frac{4e^{4} + L_6}{(\log n)^{3/2}}\,,\]
and note that the same is satisfied by $p = q + \Theta(q^2)$
with a slightly smaller constant. Suppose that $[n]^2$ is (internally) filled. Then by the Aizenman--Lebowitz lemma there exists an internally filled rectangle $R \subset [n]^2$ with 
\[\frac{1}{4q} \log \frac{1}{q} \le \lgs(R) \le \frac{1}{2q} \log \frac{1}{q}\,.\]
There are at most $n^2 (\log n)^3$ rectangles satisfying those conditions, and each one satisfies the conditions of Theorem~\ref{thm:droplet:again}. Hence, by the union bound and Lemma~\ref{obs:lambda} we have
\begin{equation*}
\Pr_p\big( [A] = [n]^2 \big) \le n^2 (\log n)^3 \exp\left( - \frac{2\lambda}{q} + \frac{4e^{4}+L_6}{\sqrt{q}} \right) \rightarrow 0
\end{equation*}
as $n \to \infty$, as required.
\end{proof}

\section{Open problems}\label{sec:open}

The most obvious problem suggested by Theorem~\ref{thm:sharp} is to determine even more precise bounds on $p_c\big( [n]^2,2 \big)$. By a theorem of Balogh and Bollob\'as~\cite{Balogh03}, it is known that the `critical window' in which the probability of percolation increases from $o(1)$ to $1 - o(1)$ has size at most $(\log n)^{-2 + o(1)}$, and it is therefore natural to make the following conjecture.

\begin{conj}
\label{conj:exact}
There exists a constant $\mu > 0$ such that
\[p_c\big( [n]^2,2 \big) = \frac{\pi^2}{18\log n} - \frac{\mu + o(1)}{(\log n)^{3/2}}\]
as $n \to \infty$.
\end{conj}

Another natural direction for future research would be to extend the results of this paper to higher dimensions. The following conjecture was made by by Uzzell~\cite{Uzzell12}, who also established the upper bound.

\begin{conj}[Conjecture~7.1 of~\cite{Uzzell12}]\label{conj:uzzell}
\[p_c\left([n]^d,r\right)= \bigg( \frac{\lambda(d,r)}{\log_{(r-1)}n} - \frac{\Theta(1)}{\big(\log_{(r-1)}n \big)^{3/2}} \bigg)^{d-r+1}\,.\]
\end{conj}

As a first step, it would be interesting to determine whether or not this conjecture holds in the case $r = 2$. In particular, one might hope that the conjecture in this case would follow from a suitable generalization of the technique used in this paper. 

However, perhaps the most interesting avenue for further research would be to prove corresponding `sharp' and `sharper' thresholds for other two-dimensional models, cf.~the discussion of $\cU$-bootstrap percolation in the Introduction. It would be very interesting (and, most likely, very challenging) to determine a sharp threshold for all families with polylogarithmic critical probability, or a `sharper' threshold for either some large class of models (e.g., that studied in~\cite{Duminil12}, or a corresponding class of `unbalanced' models), or for other specific interesting examples, such as the Duarte model (see~\cite{Bollobas17}). The problem in higher dimensions is also extremely interesting, but much more difficult, and proving even much weaker bounds on the critical probability for general $\cU$-bootstrap models (see, for example,~\cite[Conjecture~1.6]{Morris17}) is an important open problem.

\appendix
\section{Proof of Lemmas~\ref{lem:key:small} and~\ref{lem:key:big}}

In this appendix we complete the proofs of Lemmas~\ref{lem:key:small} and~\ref{lem:key:big}, by dealing with the cases that were omitted from the sketch proofs given in Section~\ref{sec:key}. The details are somewhat tedious but, for the convenience of the diligent (or skeptical) reader, we will attempt to spell everything out slowly and carefully.

For the entire appendix we fix the rectangle $R = [(0,0),(a-1,b-1)]$. Recall from Definition~\ref{def:buffers} that if $S \subset R$ is a rectangle, then we write 
\[z(S,R) = |Z(S,R)| = \big|\big\{ \d \in \cI \,:\, B_\d(S,R) \ne \emptyset \big\} \big|\,,\]
where $\cI=\{(-1,0),(1,0),(0,-1),(0,1)\}$ is the set of directions, and $B_\d(S,R)$ is the buffer in direction $\d$. Recall also that if $\x\in\{0,1\}^{\cI}$ and $S \subset R$ is a rectangle, then 
\begin{equation}\label{def:xy:app}
x = x'_{(1,0)}  + x'_{(-1,0)} \qquad \text{and} \qquad y = x'_{(0,1)} + x'_{(0,-1)}\,,
\end{equation}
where $\x' := \x \cdot \mathbbm{1}_{Z(S,R)}$. For convenience, let us begin by restating the two key lemmas.

\begin{lem}[Lemma~\ref{lem:key:small} restated]
\label{lem:key:small:app}
Let $\x\in\{0,1\}^{\cI}$ and $S\subset R$ be a rectangle with $\dim S=(a-s,b-t)$ and set $z=z(S,R)$. If
\begin{equation}
\label{eq:key:small:R:app}
L_1 \le \sh(R) \le \frac{B}{q} \qquad \text{and} \qquad \lgs(R) \le \frac{3e^{2B}}{q}\,,
\end{equation}
and $s,t \le 4\delta \sqrt{\sh(R)}$, then
\[\Pr_p\big( D_1^{\x}(S,R) \big) \le C^{z} \left( \frac{C}{\sqrt{a}} \right)^y \left( \frac{C}{\sqrt{b}} \right)^x \exp\big( - s g(bq) - t g(aq) \big)\,.\]
\end{lem}

\begin{lem}[Lemma~\ref{lem:key:big} restated]\label{lem:key:big:app}
Let $\x\in\{0,1\}^{\cI}$ and $S\subset R$ be a rectangle with $\dim S=(a-s,b-t)$ and set $z=z(S,R)$. If
\begin{equation}
\label{eq:key:big:R:app}
\sh(R) > \frac{B}{q} \qquad \text{and} \qquad \lgs(R) \le \frac{1}{2q} \log \frac{1}{q}
\end{equation}
and $s,t \le \ds\frac{4\delta}{\sqrt{q}} \cdot \exp\big( \sh(R) \cdot q \big)$, then
\[\Pr_p\big( D_2^{\x}(S,R) \big) \le \left( C e^{\sh(R) q} \right)^z \left( C \sqrt{q} e^{-aq} \right)^y \left( C\sqrt{q} e^{-bq} \right)^x \exp\big( - s g(bq) - t g(aq) \big)\,.\]
\end{lem}

We begin with a straightforward technical lemma, which will be required in both proofs.

\begin{lem}
\label{lem:localisation}
Let $15 \le D \le 4\delta/q$, and suppose that $a,b \leq (1/(2q)) \log (1/q)$, and that $s,t \leq D$ and $t \leq \min\{a,b\}$. Then, for any $x,y \in \{0,1,2\}$, we have
\[\exp\big( - 3Dg(tq) \big) \leq \left(\frac{1}{\sqrt{a}}\right)^y\left(\frac{1}{\sqrt{b}}\right)^x \exp\big( - sg(bq) - t g(aq) \big)\]
and 
\[\exp\big( - 3Dg(tq) \big) \leq \left(\sqrt{q}e^{-aq}\right)^y\left(\sqrt{q}e^{-bq}\right)^x\exp\big( - sg(bq) - tg(aq) \big)\,.\]
\end{lem}

\begin{proof}
Observe first that $- 3g(tq) \le \log(tq) \le \log(Dq)$, by~\eqref{obs:g3} and since $t \le D \le 4\delta/q$. Noting that $D\log (Dq)$ is decreasing in $D$, it follows that
\[\exp\big( - Dg(tq) \big) \leq \exp\big(5\log(15q) \big) \leq q^4 \leq\min\left\{\frac{1}{ab}, \, q^2e^{-2q(a+b)}\right\}\,,\]
where in the last step we used the bound $a,b \le (1/(2q)) \log (1/q)$. Moreover, we have 
\[\exp\big( - 2Dg(tq) \big) \le \exp\big( - sg(bq) - tg(aq) \big)\,,\]
since $s,t \le D$ and $t\leq \min\{a,b\}$, and recalling that $g$ is decreasing. Combining these two inequalities we obtain the two claimed bounds.
\end{proof}

\begin{proof}[Proof of Lemma \ref{lem:key:small:app}]
Recall that $D_1^{\x}(S,R)$ denotes the event that 
\[\big[ S \cup \big( A \cap R \setminus \Sbar \big) \big] = R\,.\]
Let $a,b$ satisfy \eqref{eq:key:small:R:app}, and for each $x,y,z$ and each $s,t \le 4\delta\sqrt{\min\{a,b\}}$, set
\[F^{x,y,z}(s,t) := C^{z} \left( \frac{C}{\sqrt{a}} \right)^y \left( \frac{C}{\sqrt{b}} \right)^x \exp\big( - s g(bq) - t g(aq) \big)\,.\]
We will prove, by induction on the pair $(s+t,-(x+y))$, that 
\begin{equation}\label{ih:key:small:app}
\Pr_p\big( D_1^{\x}(S,R) \big) \le F^{x,y,z}(s,t)
\end{equation}
for every $0 \le s,t \le 4\delta\sqrt{\min\{a,b\}}$ and $\x \in \{0,1\}^\cI$, and every $S \subset R$ with $\dim(S) = (a-s,b-t)$, where $x$ and $y$ are as defined in~\eqref{def:xy:app}, and $z = z(S,R)$. 

The base of the induction, the case $\min\{ s, t \} = 0$, was dealt with in Section~\ref{sec:key}, so let us fix $\x \in \{0,1\}^\cI$ and $S \subset R$ with $\dim(S) = (a-s,b-t)$, and assume that~\eqref{ih:key:small:app} holds for all smaller values of the pair $(s+t,-(x+y))$ in lexicographical order. 

Note first that, since $\sh(R) \ge L_1$, the function $F^{x,y,z}(s,t)$ is increasing in $z$ and decreasing in $x$, $y$, $s$ and $t$. Note also that we may assume, without loss of generality, that $\x = \x \cdot \mathbbm{1}_{Z(S,R)}$ (i.e., that $\x_\d = 0$ whenever the buffer $B_\d(S,R)$ is empty), since neither side of the inequality~\eqref{ih:key:small:app} depends on the value of $\x_\d$ if $\d \not\in Z(S,R)$.

We partition into cases, depending on whether or not $z=x+y$. 

\medskip
\noindent \textbf{Case 1:} $z = x + y$, i.e., all of the non-empty buffers are included in $\Sfr$.
\medskip

The key observation in this case is that if the event $D_1^{\x}(S,R)$ holds, then there exists a rectangle $T \subset R$ such that 
\begin{equation}\label{eq:app:T:obs}
\big[ A \cap T \setminus \Sbar \big] = T \qquad  \text{and} \qquad T \cap \Sbar \ne \emptyset
\end{equation}
(see Figure~\ref{fig:cornerap}). In Section~\ref{sec:key} we assumed that $\phi(T) \le 36 f(R) = 36\delta \sqrt{\min\{a,b\}}$, so let us begin by dealing with the other case. To do so, the key observation is that, since $T$ is internally filled by the infected sites in $T \setminus \Sbar$, one of the eight rectangles in Figure~\ref{fig:8regions} has no double gap crossing it in the `short' direction. To be more precise, set $D := 4\delta\sqrt{\min\{a,b\}}$ (so, in particular, $s,t \le D$ and $9D \le \min\{a,b\}$) and consider the $(3D+1) \times (\le t)$ rectangle 
\[\big( [D,4D] \times [0,t-1] \big) \cap \big( R \setminus S \big)\,,\]
which is located at the bottom and to the left of Figure~\ref{fig:8regions}.  

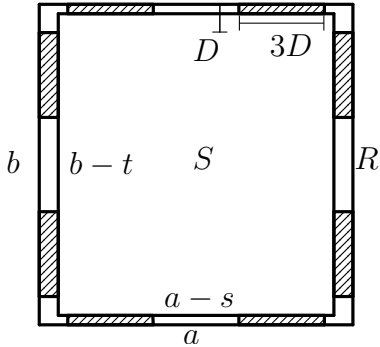
\begin{figure}[h]
\floatbox[{\capbeside\thisfloatsetup{capbesideposition={right,top}}}]{figure}[\FBwidth]
{\begin{tikzpicture}[line cap=round,line join=round,>=triangle 45,x=0.25cm,y=0.25cm]
\clip(-20,-4) rectangle (0,14.5);
\fill[line width=1.2pt,fill=black,pattern=north east lines] (-18,12.5) -- (-17,12.5) -- (-17,8) -- (-18,8) -- cycle;
\fill[line width=1.2pt,fill=black,pattern=north east lines] (-16.5,14) -- (-16.5,13.5) -- (-12,13.5) -- (-12,14) -- cycle;
\fill[line width=1.2pt,fill=black,pattern=north east lines] (-2.5,12.5) -- (-1.5,12.5) -- (-1.5,8) -- (-2.5,8) -- cycle;
\fill[line width=1.2pt,fill=black,pattern=north east lines] (-7.5,14) -- (-7.5,13.5) -- (-3,13.5) -- (-3,14) -- cycle;
\fill[line width=1.2pt,fill=black,pattern=north east lines] (-18,3) -- (-17,3) -- (-17,-1.5) -- (-18,-1.5) -- cycle;
\fill[line width=1.2pt,fill=black,pattern=north east lines] (-12,-3) -- (-12.03,-2.5) -- (-16.5,-2.5) -- (-16.5,-3) -- cycle;
\fill[line width=1.2pt,fill=black,pattern=north east lines] (-2.5,3) -- (-1.5,3) -- (-1.5,-1.5) -- (-2.5,-1.5) -- cycle;
\fill[line width=1.2pt,fill=black,pattern=north east lines] (-7.5,-3) -- (-3,-3) -- (-3,-2.5) -- (-7.5,-2.5) -- cycle;
\draw (-8.5,14)-- (-8.5,12.5);
\draw [line width=1.2pt] (-18,12.5)-- (-17,12.5);
\draw [line width=1.2pt] (-17,12.5)-- (-17,8);
\draw [line width=1.2pt] (-17,8)-- (-18,8);
\draw [line width=1.2pt] (-18,8)-- (-18,12.5);
\draw [line width=1.2pt] (-16.5,14)-- (-16.5,13.5);
\draw [line width=1.2pt] (-16.5,13.5)-- (-12,13.5);
\draw [line width=1.2pt] (-12,13.5)-- (-12,14);
\draw [line width=1.2pt] (-12,14)-- (-16.5,14);
\draw (-10.5,7.) node[anchor=north west] {$S$};
\draw (-10.5,12.7) node[anchor=north west] {$D$};
\draw (-11.,-2.7) node[anchor=north west] {$a$};
\draw (-20.3,6.82) node[anchor=north west] {$b$};
\draw (-6.45,13) node[anchor=north west] {$3D$};
\draw (-17.,6.75) node[anchor=north west] {$b-t$};
\draw (-12.,-0.5) node[anchor=north west] {$a-s$};
\draw (-2,7.) node[anchor=north west] {$R$};
\draw [line width=1.2pt] (-2.5,12.5)-- (-1.5,12.5);
\draw [line width=1.2pt] (-1.5,12.5)-- (-1.5,8);
\draw [line width=1.2pt] (-1.5,8)-- (-2.5,8);
\draw [line width=1.2pt] (-2.5,8)-- (-2.5,12.5);
\draw [line width=1.2pt] (-7.5,14)-- (-7.5,13.5);
\draw [line width=1.2pt] (-7.5,13.5)-- (-3,13.5);
\draw [line width=1.2pt] (-3,13.5)-- (-3,14);
\draw [line width=1.2pt] (-3,14)-- (-7.5,14);
\draw [|-|] (-8.5,14) -- (-8.5,12.5);
\draw [|-|] (-3,13) -- (-7.5,13);
\draw [line width=1.2pt] (-18,3)-- (-17,3);
\draw [line width=1.2pt] (-17,3)-- (-17,-1.5);
\draw [line width=1.2pt] (-17,-1.5)-- (-18,-1.5);
\draw [line width=1.2pt] (-18,-1.5)-- (-18,3);
\draw [line width=1.2pt] (-12,-3)-- (-12.03,-2.5);
\draw [line width=1.2pt] (-12.03,-2.5)-- (-16.5,-2.5);
\draw [line width=1.2pt] (-16.5,-2.5)-- (-16.5,-3);
\draw [line width=1.2pt] (-16.5,-3)-- (-12,-3);
\draw [line width=1.2pt] (-2.5,3)-- (-1.5,3);
\draw [line width=1.2pt] (-1.5,3)-- (-1.5,-1.5);
\draw [line width=1.2pt] (-1.5,-1.5)-- (-2.5,-1.5);
\draw [line width=1.2pt] (-2.5,-1.5)-- (-2.5,3);
\draw [line width=1.2pt] (-7.5,-3)-- (-3,-3);
\draw [line width=1.2pt] (-3,-3)-- (-3,-2.5);
\draw [line width=1.2pt] (-3,-2.5)-- (-7.5,-2.5);
\draw [line width=1.2pt] (-7.5,-2.5)-- (-7.5,-3);
\draw [line width=1.2pt] (-17,13.5)-- (-2.5,13.5);
\draw [line width=1.2pt] (-2.5,13.5)-- (-2.5,-2.5);
\draw [line width=1.2pt] (-2.5,-2.5)-- (-17,-2.5);
\draw [line width=1.2pt] (-17,-2.5)-- (-17,13.5);
\draw [line width=1.2pt] (-18,14)-- (-18,-3);
\draw [line width=1.2pt] (-18,-3)-- (-1.5,-3);
\draw [line width=1.2pt] (-1.5,-3)-- (-1.5,14);
\draw [line width=1.2pt] (-1.5,14)-- (-18,14);
\end{tikzpicture}}
{\caption{The $8$ rectangles, one of which must have a `short' double gap if $\phi(T)$ is large. Each has width at most $D$ and length $3D+1$. Note that $9D \le \sh(R)$, so the rectangles do not overlap.}
\label{fig:8regions}}
\end{figure}

By Lemmas~\ref{lem:doublegaps} and~\ref{lem:localisation}, the probability that this rectangle contains no vertical double gap is at most
\[\exp\big( - 3Dg(tq) \big) \leq \left(\frac{1}{\sqrt{a}}\right)^y\left(\frac{1}{\sqrt{b}}\right)^x \exp\big( - sg(bq) - t g(aq) \big) \le \frac{1}{C} \cdot F^{x,y,z}(s,t)\,.\]
Applying the same argument to the other seven rectangles in Figure~\ref{fig:8regions}, we may assume that each is either empty, or contains a double gap crossing it in the short direction. But in this case any rectangle satisfying~\eqref{eq:app:T:obs} must be contained in a square of size $(4D+1) \times (4D+1)$ in one of the four corners of $R$, and it is therefore not possible that it has semi-perimeter greater than $9 D$, as required.

We will now sum over choices of $T$ with $\phi(T)\le 9D$ the probability that 
\begin{equation}\label{eq:key:small:indep:events:app}
\big[ A \cap T \setminus \Sbar \big] = T \qquad \text{and} \qquad \big[ S \cup T \cup \big( A \cap R \setminus \Sbar \big) \big] = R\,.
\end{equation}
Note that these two events depend on disjoint sets of infected sites, and are therefore independent. To bound the probabilities of these events, we will partition according to $k := \phi(T)$, and the dimensions of $[S \cup T]$,
\[\dim\big( [S \cup T] \big) = (a - s + i, b - t + j)\,.\]
It was proved in Section~\ref{sec:key} that, given $i$, $j$ and $k$, the expected number of rectangles $T$ satisfying the first event in~\eqref{eq:key:small:indep:events:app} is at most $(24kp)^{k/2}$. Bounding the probability of the second event in~\eqref{eq:key:small:indep:events:app} is unfortunately rather more complicated, and will require the induction hypothesis, and 
a careful analysis of the possible positions of elements of $A$ in the buffers of $[S \cup T]$. Recall that the case in which there are no infected sites in the buffers was analyzed in Section~\ref{sec:key}.

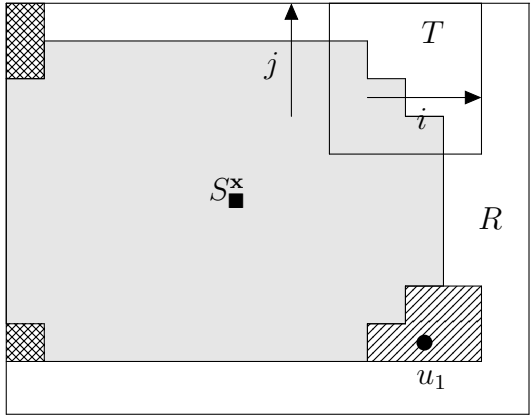
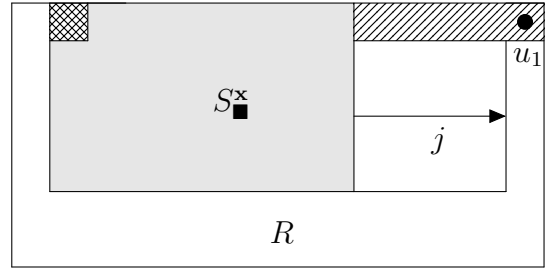
\begin{figure}[h]
\centering
\begin{subfigure}{0.47\textwidth}
\begin{center}
\begin{tikzpicture}[line cap=round,line join=round,>=triangle 45,x=0.25cm,y=0.25cm]
\clip(-20.5,-16.2) rectangle (7.6,6.5);
\fill[fill=black,fill opacity=0.1] (-18,4) -- (-18,2) -- (-20,2) -- (-20,-11) -- (-18,-11) -- (-18,-13) -- (-1,-13) -- (-1,-11) -- (1,-11) -- (1,-9) -- (3,-9) -- (3,0) -- (1,0) -- (1,2) -- (-1,2) -- (-1,4) -- cycle;
\fill[fill=black,pattern=north east lines] (-1,-11) -- (-1,-13) -- (5,-13) -- (5,-9) -- (1,-9) -- (1,-11) -- cycle;
\fill[fill=black,pattern=north east lines] (-20,2) -- (-18,2) -- (-18,6) -- (-20,6) -- cycle;
\fill[fill=black,pattern=north west lines] (-20,2) -- (-18,2) -- (-18,6) -- (-20,6) -- cycle;
\fill[fill=black,pattern=north east lines] (-20,-13) -- (-18,-13) -- (-18,-11) -- (-20,-11) -- cycle;
\fill[fill=black,pattern=north west lines] (-20,-13) -- (-18,-13) -- (-18,-11) -- (-20,-11) -- cycle;
\draw (-20,-13)-- (-18,-13);
\draw (-20,-15.8)-- (7.5,-15.8);
\draw (7.5,-15.8)-- (7.5,6);
\draw (7.53,6)-- (-20,6);
\draw (-20,6)-- (-20,-15.8);
\draw (-9.93,-2.71) node[anchor=north west] {$\Sbar$};
\draw (4.25,-4.24) node[anchor=north west] {$R$};
\draw (-3,6)-- (-3,-2);
\draw (-3,-2)-- (5,-2);
\draw (5,-2)-- (5,6);
\draw (5,6)-- (-3,6);
\draw (1.27,5.64) node[anchor=north west] {$T$};
\draw (1,1) node[anchor=north west] {$i$};
\draw (-7,4) node[anchor=north west] {$j$};
\draw (1,-12.9) node[anchor=north west] {$u_1$};
\draw [->] (-1,1) -- (5,1);
\draw [->] (-5,0) -- (-5,6);
\draw (1,0)-- (3,0);
\draw (-18,4)-- (-18,2);
\draw (-18,2)-- (-20,2);
\draw (-20,2)-- (-20,-11);
\draw (-19.99,-11)-- (-18,-11);
\draw (-18,-11)-- (-18,-13);
\draw (-18,-13)-- (-1,-13);
\draw (-1,-13)-- (-1,-11);
\draw (-1,-11)-- (1,-11);
\draw (1,-11)-- (1,-9);
\draw (1,-9)-- (3,-9);
\draw (3,-9)-- (3,0);
\draw (3,0)-- (1,0);
\draw (1,0)-- (1,2);
\draw (1,2)-- (-1,2);
\draw (-1,2)-- (-1,4);
\draw (-1,4)-- (-18,4);
\draw (-1,-11)-- (-1,-13);
\draw (-1,-13)-- (5,-13);
\draw (5,-13)-- (5,-9);
\draw (5,-9)-- (1,-9);
\draw (1,-9)-- (1,-11);
\draw (1,-11)-- (-1,-11);
\draw (-20,2)-- (-18,2);
\draw (-18,2)-- (-18,6);
\draw (-18,6)-- (-20,6);
\draw (-20,6)-- (-20,2);
\fill [color=black] (2,-12) circle (3pt);
\end{tikzpicture}
\end{center}
\subcaption{The rectangle $T$ is internally filled outside the shaded $\Sbar$ and allows $S$ to grow $i$ to the right and $j$ upwards. Since there is an infected site $u_1 \in J$ (in the hatched region), but not in the double hatched set (so $J(u_1)$ is empty), in this case we have $E = \{u_1\}$.}
\label{fig:cornerap}
\end{subfigure}
\quad
\begin{subfigure}{0.47\textwidth}
\begin{center}
\begin{tikzpicture}[line cap=round,line join=round,>=triangle 45,x=0.25cm,y=0.25cm]
\clip(-22.5,-12.5) rectangle (6.5,2.5);
\fill[fill=black,fill opacity=0.1] (-20,-8) -- (-4,-8) -- (-4,2) -- (-18,2) -- (-18,0) -- (-20,0) -- cycle;
\fill[fill=black,pattern=north east lines] (-4,2) -- (-4,0) -- (6,0) -- (6,2) -- cycle;
\fill[fill=black,pattern=north east lines] (-20,2) -- (-18,2) -- (-18,0) -- (-20,0) -- cycle;
\fill[fill=black,pattern=north west lines] (-20,2) -- (-18,2) -- (-18,0) -- (-20,0) -- cycle;
\draw (-22,-12)-- (6,-12);
\draw (6,-12)-- (6,2);
\draw (6,2)-- (-22,2);
\draw (-22,2)-- (-22,-12);
\draw (-20,2)-- (-20,-8);
\draw (-4,-8)-- (-4,0);
\draw [->] (-4,-4) -- (4,-4);
\draw (-4,2)-- (-4,0);
\draw (-20,-8)-- (-4,-8);
\draw (-20,2)-- (-18,2);
\draw (-20,0)-- (-18,0);
\draw (-18,0)-- (-18,2);
\draw (-18,2)-- (-16,2);
\draw (4,-8)-- (4,0);
\draw (-4,0)-- (6,0);
\draw (6,2)-- (4,2);
\draw (-4,-8)-- (4,-8);
\draw (-12.,-2.) node[anchor=north west] {$\Sbar$};
\draw (-0.5,-4.) node[anchor=north west] {$j$};
\draw (-9.,-9.) node[anchor=north west] {$R$};
\draw (3.8,0.2) node[anchor=north west] {$u_1$};
\begin{footnotesize}
\end{footnotesize}
\begin{scriptsize}
\fill [color=black] (5,1) circle (3pt);
\end{scriptsize}
\end{tikzpicture}
\end{center}
\subcaption{In Algorithm~\ref{alg:app:second:case}, the rectangle $S$ grows $j$ steps to the right until it reaches either $(a)$ the right-hand side of $R$, $(b)$ a double gap, or $(c)$ an infected site $u_1$ in the hatched region (here $x_{(0,1)} = x_{(-1,0)} = 1$ and $x_{(1,0)} = x_{(0,-1)} = 0$). In the figure case $(c)$ occurs, and hence $\ell_{(0,1)} = 1$. The double hatched site in the top-left corner is not occupied, so the set $J_{(-1,0)}$ is empty and hence $\ell_{(-1,0)} = 0$. It follows that $E = \{u_1\}$ in this case.}
\label{fig:sideap}
\end{subfigure}
\caption{The two possible growth mechanisms.}
\end{figure}

In order to deal systematically with all of the possible cases, we run the following algorithm. For simplicity we assume that $T$ contains the top right corner of $\Sfr$ (as in Figure~\ref{fig:cornerap}); the same bound follows in the other three cases by symmetry. 

\begin{alg}\label{alg:app:first:case}
We define a set $E \subset R \setminus [S\cup T]$ of size $0$, $1$ or $2$, as follows:
\begin{itemize}
\item[$1.$] If the set 
\[J := A \cap \big( B_{(-1,0)}([S\cup T],R) \cup B_{(0,-1)}([S\cup T],R) \big) \setminus \Sfr\]
is empty, then set $E := \emptyset$. 
\item[$2.$] Otherwise, choose $\d \in \{ (-1,0), (0,-1)\}$ and $u_1 \in J \cap B_{\d}([S\cup T],R)$.\footnote{Whenever we have a choice to make in either algorithm, we choose the first direction / site of $A$ in some (arbitrary) pre-defined order on $\cI$ / the sites in $R$.} Now, if
\[J(u_1) := A \cap B_{\d'}\big( [S\cup T\cup\{u_1\}],R \big) \setminus \Sfr\]
is empty, where $\{\d,\d'\} = \{(0,-1),(-1,0)\}$, then set $E := \{u_1\}$. 
\item[$3.$] Otherwise, choose $u_2 \in J(u_1)$, and set $E := \{u_1,u_2\}$.  
\end{itemize}
\end{alg} 

We now partition the second event in~\eqref{eq:key:small:indep:events:app} according to the set $E$, and apply the induction hypothesis to the rectangle 
\[\hat{S}(E) := [S \cup T \cup E]\,.\]
Recall that the case $E = \emptyset$ was dealt with in Section~\ref{sec:key}, so we may assume that $|E| \in \{1,2\}$. It is possible to deal with these two cases at the same time, but to simplify the notation we will take them one at a time. 

Indeed, suppose first that $|E| = 2$. Then, by the induction hypothesis, it follows that\footnote{Note that we used here the bound $z(\hat S,R) \le z(S,R)$, and the monotonicity of $F$ in $z,s,t$.}
\[\Pr_p \Big( \big[ \hat{S}(E) \cup \big( A \cap R \setminus \Sbar \big) \big] = R \Big) \le F^{x-2,y-2,z}(s-i-2,t-j-2)\,.\]
Now, recalling that $k = \phi(T)$, note that there are at most $(2k)^2$ choices for the set $E$. Hence, recalling that the expected number of rectangles $T$ satisfying the first event in~\eqref{eq:key:small:indep:events:app} is at most $(24kp)^{k/2}$, it will suffice (in this case) to show (cf.~\eqref{eq:key:small:first:need}) that
\begin{equation}\label{eq:key:small:first:need:Etwo}
\sum_{i + j \ge 4} \sum_{k = i + j}^{36 f(R)} 4k^2 p^2 \cdot (24kp)^{k/2} \cdot F^{x-2,y-2,z}(s-i-2,t-j-2) \ll F^{x,y,z}(s,t)\,.
\end{equation}
To see this, note first that $24kp \le \delta$, since $k \le 36\delta \sqrt{\min\{a,b\}} \le q^{-1/2}$, and that we may therefore assume that $k = i + j$. Now, observe that
\[\frac{F^{x-2,y-2,z}(s - i - 2,t - j - 2)}{F^{x,y,z}(s,t)} = \frac{ab}{C^4} \cdot e^{(i+2) g(bq) + (j+2) g(aq)} \le \frac{ab}{C^4} \left(\frac{C}{bq}\right)^{(i+2)/2}\left(\frac{C}{aq}\right)^{(j+2)/2}\] 
by~\eqref{obs:g5}, since $\max\{a,b\} \le 3e^{2B} / q$. Since $i + j = k$, and recalling that $p \le q$ and that $\min\{a,b\} \ge \max\{L_1,C^3k\}$, we obtain
\[\sum_{k = 4}^{36 f(R)} \sum_{i + j = k} 4 k^2 \cdot (24kp)^{k/2} \cdot \frac{1}{C^2} \left(\frac{C}{bq}\right)^{i/2}\left(\frac{C}{aq}\right)^{j/2} \le \, \sum_{k = 4}^{36 f(R)} \frac{4k^3 \cdot (C^2k)^{k/2}}{\min\{a,b\}^{k/2}} \,\le\, \frac{C^5}{\min\{a,b\}^2}\,,\]
which implies~\eqref{eq:key:small:first:need:Etwo}. 

The argument for the case $|E| = 1$ is almost the same. Observe first that, by symmetry, we may assume that $u_1 \in B_{(0,-1)}([S\cup T],R)$, as in Figure~\ref{fig:cornerap}. Noting that the set $J(u_1)$ (in Algorithm~\ref{alg:app:first:case}) is empty, it follows from the induction hypothesis that
\[\Pr_p \Big( \big[ \hat{S}(E) \cup \big( A \cap R \setminus \Sbar \big) \big] = R \Big) \le F^{x-1,y-2,z}(s-i,t-j-2)\,.\]
There are at most $2k$ choices for the vertex $u_1$, so it will suffice (in this case) to show that
\begin{equation}\label{eq:key:small:first:need:Eone}
\sum_{i + j \ge 4} \sum_{k = i + j}^{36 f(R)} 2k p \cdot (24kp)^{k/2} \cdot F^{x-1,y-2,z}(s-i,t-j-2) \ll F^{x,y,z}(s,t)\,.
\end{equation}
Since $24kp \le \delta$, we may again assume that $k = i + j$. Now, observe that
\[\frac{F^{x-1,y-2,z}(s - i,t - j - 2)}{F^{x,y,z}(s,t)} = \frac{a\sqrt{b}}{C^3} \cdot e^{i g(bq) + (j+2) g(aq)} \le \frac{a\sqrt{b}}{C^3} \left(\frac{C}{bq}\right)^{i/2}\left(\frac{C}{aq}\right)^{(j+2)/2}\]
by~\eqref{obs:g5}, since $\max\{a,b\} \le 3e^{2B} / q$. Since $i \ge 1$, we obtain
\[\sum_{k = 4}^{36 f(R)} \sum_{i + j = k} 2 k \cdot (24kp)^{k/2} \cdot \frac{\sqrt{b}}{C^2} \left(\frac{C}{bq}\right)^{i/2}\left(\frac{C}{aq}\right)^{j/2} \le \, \sum_{k = 4}^{36 f(R)} \frac{2k^2\cdot (C^2k)^{k/2}}{\min\{a,b\}^{(k-1)/2}} \,\le\, \frac{C^5}{\min\{a,b\}^{3/2}}\,,\]
which implies~\eqref{eq:key:small:first:need:Eone}, since $\min\{a,b\} \ge \max\{L_1,C^3k\}$. This completes the proof in Case~1. 

\bigskip
\noindent \textbf{Case 2:} $z > x + y$, i.e., some non-empty buffer is not included in $\Sfr$.
\medskip

Without loss of generality, let $B_{(1,0)}(S,R)$ be a non-empty buffer that is not included in $\Sfr$, so $x_{(1,0)} = 0$. As explained in the sketch proof, the idea is to `grow' $S$ to the right until we find a double gap, an infected site in one of the buffers, or reach the right-hand side of $R$, thus leading either to an increase in $x+y$, or a decrease in $s+t$. In Section~\ref{sec:key} we dealt with the cases in which we find a double gap before reaching the right-hand side, and that before doing so we do not find any infected sites in the buffers above or below $S$. Here we will deal with the other cases.

\begin{alg}\label{alg:app:second:case}
We define a set $E \subset R \setminus \Sfr$ of size $0$, $1$, $2$ or $3$, and also an integer $j \in \{0,\ldots,s\}$ and a variables $\ell_\d \in \{0,1\}$ for each direction $\d \in \cI$, as follows:\footnote{Initially $E = \emptyset$ and $\ell_{\d} = 0$ for every direction $\d\in\cI$.}
\begin{itemize}
\item[$1.$] Set $\tilde{S} := \bigcup_{i = 0}^j \big( S + (i,0) \big)$ and $\tilde{\x} := \x + \mathbbm{1}_{(1,0)}$, where
\[j := \min\big\{ i \ge 0 \,:\, A \cap R \cap \big( S + (i+2,0) \big) \setminus \big( S + (i,0) \big) = \emptyset \big\}\,.\]
If the set $A \cap \tilde{S}_{\square}^{\tilde{\x}} \setminus S_{\square}^\x$ is empty, then go to Step~8.\smallskip
\item[$2.$] Set $\tilde{S} := \bigcup_{i = 0}^j \big( S + (i,0) \big)$, where $j$ is minimal such that $\tilde{S} \setminus S$ is crossed from left to right, and 
\[\big( A \cap \tilde{S}_{\square}^{\tilde{\x}} \big) \setminus \big( S_{\square}^{\x} \cup B_{(1,0)}(\tilde{S},R) \big) \ne \emptyset\,.\]
\item[$3.$] Now, if $x_{(0,1)} = 1$
and the set\footnote{For each $i \in \{1,-1\}$, let $v_{(1,i)}$ denote the unique `corner' site in $R \setminus \tilde{S}$ with a neighbour in each of the buffers  $B_{(1,0)}(\tilde{S},R)$ and $B_{(0,i)}(\tilde{S},R)$.} 
\[J_{(0,1)} := A \cap \Big( \big( B_{(0,1)}( \tilde{S} %[\tilde{S} \cup E]
,R) \setminus B_{(0,1)}(S,R) \big) \cup \{ v_{(1,1)} \} \Big)\]
is non-empty, then add a site $u_{(0,1)} \in J_{(0,1)}$ to $E$, and set $\ell_{(0,1)} := 1$.\smallskip
\item[$4.$] Similarly, if $x_{(0,-1)} = 1$
and the set 
\[J_{(0,-1)} := A \cap \Big( \big( B_{(0,-1)}([\tilde{S}\cup E],R) \setminus B_{(0,-1)}(S,R) \big) \cup \{ v_{(1,-1)} \} \Big)\]
is non-empty, then add a site $u_{(0,-1)} \in J_{(0,-1)}$ to $E$, and set $\ell_{(0,-1)} := 1$.\smallskip
\item[$5.$] If $x_{(0,1)} = 1$ and $E = \{ v_{(1,-1)} \}$, and the set\footnote{Note that this set has at most one element.}
\[J'_{(0,1)} := A \cap \Big( \big( B_{(0,1)}( [\tilde{S} \cup E],R) \setminus B_{(0,1)}(S,R) \big)  \Big)\]
is non-empty, then add a site $u_{(0,1)} \in J'_{(0,1)}$ to $E$, and set $\ell_{(0,1)} := 1$.\smallskip
\item[$6.$] If $x_{(-1,0)} = 0$ then go to Step~8. Otherwise, if the set
\[J_{(-1,0)} := A \cap B_{(-1,0)}( [\tilde{S} \cup E],R ) \setminus \Sfr\]
is non-empty, then add a site $u_{(-1,0)} \in J_{(-1,0)}$ to $E$, and set $\ell_{(-1,0)} := 1$.\pagebreak
\item[$7.$] If either $\ell_{(0,1)} + \ell_{(0,-1)} = x_{(0,1)} + x_{(0,-1)}$ or $\ell_{(-1,0)} = 0$ then go to Step~8. Otherwise:\smallskip
\begin{itemize}
\item[$(a)$] If $\ell_{(0,1)} = 1$ and $\ell_{(0,-1)} = 0$, and the set 
\[J_{(-1,-1)} := A \cap B_{(0,-1)}( [\tilde{S} \cup E],R ) \setminus \Sfr\]
is non-empty, then add a site $u_{(-1,-1)} \in J_{(-1,-1)}$ to $E$ and set $\ell_{(0,-1)}:=1$.\smallskip
\item[$(b)$] If $\ell_{(0,1)} = 0$ and $\ell_{(0,-1)} = 1$, and the set 
\[J_{(-1,1)} := A \cap B_{(0,1)}( [\tilde{S} \cup E],R ) \setminus \Sfr\]
is non-empty, then add a site $u_{(-1,1)} \in J_{(-1,1)}$ to $E$ and set $\ell_{(0,1)}:=1$.\smallskip
\end{itemize}
\item[$8.$] Set $\hat{S} := [\tilde{S} \cup E]$ and STOP.
\end{itemize}
\end{alg} 

Observe that Algorithm~\ref{alg:app:second:case} outputs a set $E$ and an integer $j$ (which together determine the sets $\tilde S$ and $\hat S$, and the variable $\ell_\d$ for each $\d \in \cI$) with the following properties:
\begin{itemize}
\item[$(a)$] $E\subset A$; \smallskip
\item[$(b)$] $[S\cup(A\cap\tilde S)]=\tilde S$;\smallskip
\item[$(c)$] On the event $D_1^\x(S,R)$ the event $D_1^{\hat\x}(\hat S,R)$ occurs, i.e.,
\[\big[ \hat{S} \cup \big( A \cap R \setminus \hat{S}_{\blacksquare}^{\hat\x} \big) \big]=R\,,\]
where $\hat\x=\x-\boldsymbol\ell$, except in the case treated in Section~\ref{sec:key}, in which $\hat\x=\x+\1_{(1,0)}$.
\end{itemize}
We will analyse each case individually, and sum over all possible sets $E$ and $j \in \{0,\ldots,s\}$.\medskip

Suppose first that $E = \emptyset$. In Section~\ref{sec:key} we dealt with this case, under the additional assumption that $B_{(1,0)}(\hat{S},R) \ne \emptyset$ (i.e., we found a double gap before reaching the right-hand side of $R$); here we will deal with the other case (i.e., that $\hat S$ reaches the right-hand side of $R$). To do so, we need to bound the probability that
\begin{equation*}\label{eq:app:small:indep:events}
[ S \cup ( A \cap \tilde{S} ) ] = \tilde{S} \qquad \text{and} \qquad D_1^{\hat\x}(\hat S,R)\, ,
\end{equation*}
where in this case $\hat S = \tilde S$ and $\hat \x = \x$. Note that these two events depend on disjoint sets of sites, and are therefore independent; we will bound the first using Lemma~\ref{lem:doublegaps}, and the second using the induction hypothesis. Indeed, by Lemma~\ref{lem:doublegaps} (and since $g(z)$ is decreasing) we have
\begin{equation*}\label{eq:app:doublegaps:grow:right}
\Pr_p\big( [ S \cup ( A \cap \tilde{S} ) ] = \tilde{S} \big) \le \exp\big( - j g(bq) \big)\,,
\end{equation*}
and by the induction hypothesis (since $B_{(1,0)}(\hat{S},R) = \emptyset$ and $E = \emptyset$),
\[\Pr_p\big( D_1^{\x}(\hat S,R) \big) \le F^{x,y,z-1}(s - j,t)\,.\]
Thus, the probability that $E = \emptyset$, $B_{(1,0)}(\hat{S},R) = \emptyset$ and $D_1^\x(S,R)$ is at most\footnote{Note that there is a unique possible value of $j$ for which $B_{(1,0)}(\hat S,R)=\emptyset$.}
\[\exp\big( - j g(bq) \big) F^{x,y,z-1}(s-j,t) = \frac{1}{C} \cdot F^{x,y,z}(s,t)\]
which suffices since $C$ is sufficiently large.\medskip

We will therefore assume from now on that $E \ne \emptyset$, which means that before reaching a double gap (when trying to grow $S$ rightwards, forming $\tilde{S}$), we find an infected site in either the buffer above or below $\tilde{S}$. More precisely, recall that $\tilde{S} := \bigcup_{i = 0}^j \big( S + (i,0) \big)$, where $j$ is minimal such that $\tilde{S} \setminus S$ is crossed from left to right, and 
\[\big( A \cap \tilde{S}_{\square}^{\tilde{\x}} \big) \setminus \big( S_{\square}^{\x} \cup B_{(1,0)}(\tilde{S},R) \big) \ne \emptyset\,,\]
where $\tilde{\x} = \x + \mathbbm{1}_{(1,0)}$. Note that $x_{(0,1)} + x_{(0,-1)} \ge 1$, and that moreover $\ell_{(0,1)} + \ell_{(0,-1)} \ge 1$  (since the set above is non-empty).\medskip

Suppose first that $|E| = 1$, and therefore (without loss of generality) we have $\ell_{(0,1)} = 1$, $E = \{ u_{(0,1)} \}$ and $\ell_{(0,-1)} = \ell_{(-1,0)} = 0$. This implies that the (independent) events
\begin{equation*}\label{eq:app:small:indep:events2}
u_{(0,1)} \in A, \qquad [ S \cup ( A \cap \tilde S ) ] = \tilde S \qquad \text{and} \qquad D_1^{\hat\x}(\hat S,R)
\end{equation*}
occur, where $\hat{\x} = \x - \mathbbm{1}_{(0,1)}$. Given $j$, it follows from the minimality of $j$ that there are at most four choices for $u_{(0,1)}$, so 
\[\Pr_p\Big( \big\{ u_{(0,1)} \in A \big\} \cap \big\{ [ S \cup ( A \cap \tilde{S} ) ] = \tilde{S} \big\}  \Big) \le 4p \exp\big( - j g(bq) \big)\,.\]
Suppose first that $u_{(0,1)} \ne v_{(1,1)}$. In this case, by the induction hypothesis, we have
\[\Pr_p\big( D_1^{\hat\x}(\hat S,R) \big) \le F^{x,y-1,z}(s - j,t-2)\,,\]
and so the probability in this case can be bounded by 
\[\sum_{j = 0}^s 4p \exp\big( - j g(bq) \big) F^{x,y-1,z}(s - j,t-2) \le \frac{4(s+1)p\sqrt{a}}{C} e^{2 g(aq)} \cdot F^{x,y,z}(s,t)\,.\]
Recalling that $s \le 4\delta\sqrt{a}$ and that $e^{2 g(aq)} \le C / aq$, by~\eqref{obs:g5}, and recalling that $\delta$ is sufficiently small, we obtain a suitably strong bound in this case. Similarly, if $u_{(0,1)} = v_{(1,1)}$, then the induction hypothesis gives
\[\Pr_p\big( D_1^{\hat\x}(\hat S,R) \big) \le F^{x,y-1,z}(s - j - 1,t - 1)\,,\]
and so the probability in this case can be bounded by 
\[\sum_{j = 0}^s 4p \exp\big( - j g(bq) \big) F^{x,y-1,z}(s - j - 1,t - 1) \le \frac{4(s+1)p\sqrt{a}}{C} e^{g(aq)+g(bq)} \cdot F^{x,y,z}(s,t)\,.\]
Since $s \le 4\delta\sqrt{b}$ and $e^{g(aq) + g(bq)} \le C / (q\sqrt{ab})$, we again obtain a suitable bound. \medskip

In all remaining cases we win easily, since each extra infected site is extremely expensive. Nevertheless, we will go carefully through each case. Indeed, suppose next that $|E| = 2$, so (without loss of generality) either 
\begin{itemize}
\item[$(a)$] $\ell_{(0,1)} = \ell_{(0,-1)} = 1$, $\ell_{(-1,0)} = 0$, and $E = \{ u_{(0,1)}, u_{(0,-1)} \}$, or \smallskip
\item[$(b)$] $\ell_{(0,1)} = \ell_{(-1,0)} = 1$, $\ell_{(0,-1)} = 0$, and $E = \{ u_{(0,1)}, u_{(-1,0)} \}$.
\end{itemize}
In case~$(a)$, the (independent) events
\begin{equation*}\label{eq:app:small:indep:events2}
E \subset A, \qquad [ S \cup ( A \cap \tilde S ) ] = \tilde S \qquad \text{and} \qquad D_1^{\hat\x}(\hat S,R)
\end{equation*}
occur, where $\hat{\x} = \x - \mathbbm{1}_{(0,1)} - \mathbbm{1}_{(0,-1)}$. Given $j$, there are at most five choices for each element of $E$, so 
\[\Pr_p\big( E \subset A \, \text{ and } \, [ S \cup ( A \cap \tilde{S} ) ] = \tilde{S}  \big) \le (5p)^2 \exp\big( - j g(bq) \big)\,.\]
Suppose first that the set $E \cap \{ v_{(1,1)}, v_{(1,-1)} \}$ is empty. In this case, by the induction hypothesis, we have
\[\Pr_p\big( D_1^{\hat\x}(\hat S,R) \big) \le F^{x,y-2,z}(s - j,t - 4)\,,\]
and so the probability in this case can be bounded by 
\[\sum_{j = 0}^s (5p)^2 e^{- j g(bq)} \cdot F^{x,y-2,z}(s - j,t - 4) \le \frac{25(s+1) p^2 a}{C^2} e^{4 g(aq)} \cdot F^{x,y,z}(s,t)\,.\]
Since $s \le 4\delta\sqrt{a}$ and $e^{2 g(aq)} \le C / aq$, we win (easily) in this case. Similarly, if the set $E \cap \{ v_{(1,1)}, v_{(1,-1)} \}$ is non-empty, then the induction hypothesis gives
\[\Pr_p\big( D_1^{\hat\x}(\hat S,R) \big) \le F^{x,y-2,z}(s - j - 1,t - 3)\,,\]
and so the probability in this case can be bounded by 
\[\sum_{j = 0}^s (5p)^2 e^{- j g(bq)} \cdot F^{x,y-2,z}(s - j - 1,t - 3) \le \frac{25(s+1) p^2 a}{C^2} e^{3g(aq)+g(bq)} \cdot F^{x,y,z}(s,t)\,.\]
Since $s \le 4\delta\sqrt{b}$ and $e^{3g(aq) + g(bq)} \le C^2 / (q^2\sqrt{a^3b})$, we again obtain a suitable bound.

In case~$(b)$, the (independent) events
\begin{equation*}\label{eq:app:small:indep:events2}
E \subset A, \qquad [ S \cup ( A \cap \tilde S ) ] = \tilde S \qquad \text{and} \qquad D_1^{\hat\x}(\hat S,R)
\end{equation*}
occur, where $\hat{\x} = \x - \mathbbm{1}_{(0,1)} - \mathbbm{1}_{(-1,0)}$. Given $j$, there are at most four choices for each element of $E$, so 
\[\Pr_p\big(E \subset A \, \text{ and } \, [ S \cup ( A \cap \tilde{S} ) ] = \tilde{S} \big) \le (4p)^2 \exp\big( - j g(bq) \big)\,.\]
Suppose first that $u_{(0,1)} \ne v_{(1,1)}$. In this case, by the induction hypothesis, we have
\[\Pr_p\big( D_1^{\hat\x}(\hat S,R) \big) \le F^{x-1,y-1,z}(s - j - 2,t - 2)\,,\]
and so the probability in this case can be bounded by 
\[\sum_{j = 0}^s (4p)^2 e^{- j g(bq)} \cdot F^{x-1,y-1,z}(s - j - 2,t - 2) \le \frac{16(s+1)p^2 \sqrt{ab}}{C^2} e^{2 g(aq)+2 g(bq)} \cdot F^{x,y,z}(s,t)\,.\]
Since $s \le 4\delta\sqrt{a}$ and $e^{2 g(aq)+2 g(bq)} \le C^2 / (abq^2)$, we are done in this case, as before. Similarly, if $u_{(0,1)} = v_{(1,1)}$, then the induction hypothesis gives
\[\Pr_p\big( D_1^{\hat\x}(\hat S,R) \big) \le F^{x-1,y-1,z}(s - j - 3,t - 1)\,,\]
and so the probability in this case can be bounded by 
\[\sum_{j = 0}^s (4p)^2 e^{- j g(bq)} \cdot F^{x-1,y-1,z}(s - j - 3,t - 1) \le \frac{16(s+1)p^2 \sqrt{ab}}{C^2} e^{g(aq)+3g(bq)} \cdot F^{x,y,z}(s,t)\,.\]
Since $s \le 4\delta\sqrt{b}$ and $e^{g(aq) + 3g(bq)} \le C^2 / (q^2\sqrt{ab^3})$, we again obtain a suitable bound. 

\medskip
Finally, suppose that $|E| = 3$, and observe that $\ell_{(0,1)} = \ell_{(0,-1)} = \ell_{(-1,0)} = 1$ and, without loss of generality, either
\begin{itemize}
\item[$(a)$] $E = \{ u_{(0,1)}, u_{(0,-1)}, u_{(-1,0)} \}$, or \smallskip
\item[$(b)$] $E = \{ u_{(0,1)}, u_{(-1,0)}, u_{(-1,-1)} \}$.
\end{itemize}
In either case, the (independent) events
\begin{equation*}\label{eq:app:small:indep:events2}
E \subset A, \qquad [ S \cup ( A \cap \tilde S ) ] = \tilde S \qquad \text{and} \qquad D_1^{\hat\x}(\hat S,R)
\end{equation*}
occur, where $\hat{\x} = \x - \mathbbm{1}_{(0,1)} - \mathbbm{1}_{(0,-1)} - \mathbbm{1}_{(-1,0)} = \0$. Given $j$, there are at most six choices for each element of $E$, so 
\[\Pr_p\big(E \subset A \, \text{ and } \, [ S \cup ( A \cap \tilde{S} ) ] = \tilde{S} \big) \le (6p)^3 \exp\big( - j g(bq) \big)\,.\]
Suppose first that the set $E \cap \{ v_{(1,1)}, v_{(1,-1)} \}$ is empty. In this case, by the induction hypothesis, we have
\[\Pr_p\big( D_1^{\hat\x}(\hat S,R) \big) \le F^{x-1,y-2,z}(s - j - 2,t - 4)\,,\]
and so the probability in this case can be bounded by 
\[\sum_{j = 0}^{s-1} (6p)^3 e^{- j g(bq)} \cdot F^{x-1,y-2,z}(s - j - 2,t - 4) \le \frac{6^3 s p^3 a\sqrt{b}}{C^3} e^{4 g(aq) + 2g(bq)} \cdot F^{x,y,z}(s,t)\,.\]
Since $s \le 4\delta\sqrt{a}$ and $e^{4 g(aq) + 2g(bq)} \le C^3 / (q^3a^2b)$, we again win easily in this case. Similarly, if $E \cap \{ v_{(1,1)}, v_{(1,-1)} \}$ is non-empty, then the induction hypothesis gives
\[\Pr_p\big( D_1^{\hat\x}(\hat S,R) \big) \le F^{x-1,y-2,z}(s - j - 3,t - 3)\,,\]
and so the probability in this case can be bounded by 
\[\sum_{j = 0}^{s-1} (6p)^3 e^{- j g(bq)} \cdot F^{x-1,y-2,z}(s - j - 3,t - 3) \le \frac{6^3 s p^3 a \sqrt{b}}{C^3} e^{3g(aq)+3g(bq)} \cdot F^{x,y,z}(s,t)\,.\]
Since $s \le 4\delta\sqrt{b}$ and $e^{3g(aq) + 3g(bq)} \le C^3 / (q^3 \sqrt{a^3b^3})$, we again obtain a suitable bound.\medskip

Summing over the various cases completes the proof of Lemma~\ref{lem:key:small:app}.
\end{proof}

The proof of Lemma~\ref{lem:key:big:app} is very similar to that of Lemma~\ref{lem:key:small:app}, and we will be able to reuse large parts of the proof (in particular Algorithms~\ref{alg:app:first:case} and~\ref{alg:app:second:case}). 

\begin{proof}[Proof of Lemma~\ref{lem:key:big:app}]
Recall that $D_2^{\x}(S,R)$ denotes the event that 
\[\big[ S \cup \big( A \cap R \setminus \Sbar \big) \big] = R \qquad \text{and} \qquad A \cap \Sfr = \emptyset\,.\]
As in the proof of Lemma~\ref{lem:key:small:app}, we use induction on the pair $(s+t,-(x+y))$, this time to prove that 
\begin{equation*}\label{ih:key:big:app}
\Pr_p\big( D_2^{\x}(S,R) \big) \le \hat{F}^{x,y,z}(s,t)\,,
\end{equation*}
where
\[\hat{F}^{x,y,z}(s,t) := \left( C e^{\sh(R) q} \right)^z \left( C \sqrt{q} e^{-aq} \right)^y \left( C\sqrt{q} e^{-bq} \right)^x \exp\big( - s g(bq) - t g(aq) \big)\,,\]
for every $0 \le s,t \le 4\delta \cdot q^{-1/2} \cdot \exp\big( \min\{a,b\} \cdot q \big)$ and $\x \in \{0,1\}^\cI$, and every $S \subset R$ with $\dim(S) = (a-s,b-t)$, where $x$ and $y$ are as defined in~\eqref{def:xy:app}, and $z = z(S,R)$. The base of the induction remains unchanged from Lemma~\ref{lem:key:small:app}.

As in the proof of Lemma~\ref{lem:key:small:app}, we partition into cases depending on whether or not $z=x+y$, the function $\hat{F}^{x,y,z}(s,t)$ is increasing in $z$ and decreasing in $x$, $y$, $s$ and $t$, and we may assume without loss of generality that $\x = \x \cdot \mathbbm{1}_{Z(S,R)}$.  

\pagebreak
\noindent \textbf{Case 1:} $z = x + y$, i.e., all of the non-empty buffers are included in $\Sfr$.
\medskip

As in Lemma~\ref{lem:key:small:app}, the event $D_2^\x(S,R)$ requires the existence of a rectangle $T$ such that
\[\big[A\cap T\setminus\Sbar\big]=T \qquad\text{and}\qquad T\cap\Sbar\neq \emptyset\,.\]
The first step is to apply Lemma~\ref{lem:localisation}, as in the proof of Lemma~\ref{lem:key:small:app}, to exclude rectangles $T$ with $\phi(T) > 9D$, where this time we set $D := 4\delta \cdot q^{-1/2} \cdot \exp\big( \min\{a,b\} \cdot q \big)$. It follows from~\eqref{eq:key:big:R:app} that $s,t \le D \le 4\delta/q$ and $9D \le \min\{a,b\}$, and we may therefore argue exactly as before, except using the second inequality in Lemma~\ref{lem:localisation}, which gives
\[\exp\big( - 3Dg(tq) \big) \leq \left(\sqrt{q}e^{-aq}\right)^y\left(\sqrt{q}e^{-bq}\right)^x\exp\big( - sg(bq) - tg(aq) \big) \le \frac{1}{C} \cdot \hat{F}^{x,y,z}(s,t)\,.\]

We will therefore assume from now on that $\phi(T)\le 9D$, and sum over choices of $T$ with $\phi(T)\le 9D$ the probability that
\begin{equation}\label{eq:key:big:indep:events:app}
\big[ A \cap T \setminus \Sbar \big] = T, \quad \big[ S \cup T \cup \big( A \cap R \setminus \Sbar \big) \big] = R \quad \text{and}\quad A\cap\Sfr=\emptyset\,.
\end{equation}
Note that these events depend on disjoint sets of sites and are therefore independent. It was proved in Section~\ref{sec:key} that, given $k := \phi(T)$ and the dimensions of $[S \cup T]$, the expected number of rectangles $T$ satisfying the first event in~\eqref{eq:key:big:indep:events:app} is at most $(24kp)^{k/2}$. For the intersection of the second and third events, we will partition the space according to the set $E$ given by Algorithm~\ref{alg:app:first:case}, and apply the induction hypothesis to the set 
\[\hat{S}(E) := [S \cup T \cup E]\,.\]

Suppose first that $E = \emptyset$, and recall that this means that the set
\[\big( B_{(-1,0)}([S\cup T],R) \cup B_{(0,-1)}([S\cup T],R) \big) \setminus \Sfr\]
contains no element of $A$. Together with the second and third events in~\eqref{eq:key:big:indep:events:app}, this implies that the event $D_2^{\hat\x}( \hat{S}(E),R )$ occurs, where $\hat{\x} = \x - \mathbbm{1}_{(1,0)} - \mathbbm{1}_{(0,1)}$. By the induction hypothesis, we have
\[\Pr_p \big( D_2^{\hat\x}( \hat{S}(E),R ) \big) \le \hat F^{x-1,y-1,z}(s-i,t-j)\,,\]
where $\dim( [S \cup T] ) = (a - s + i, b - t + j)$. Now, since $B/q \le a,b \le (1/(2q)) \log(1/q)$, we have
\[\frac{\hat F^{x-1,y-1,z}(s - i,t - j)}{\hat F^{x,y,z}(s,t)} = \frac{e^{(a+b)q}}{C^2 q} \exp\big( i g(bq) + j g(aq) \big) \le \frac{2^{i+j}}{C^2 q^2}\]
since $e^{g(aq) + g(bq)} \le 2$. Recalling that $k \le 9D \le 36\delta/q$, it follows that the probability of this case is at most
\[\sum_{i + j \ge 4} \sum_{k = i + j}^{9D} (24kp)^{k/2} \cdot \frac{2^{i+j}}{C^2 q^2} \cdot \hat F^{x,y,z}(s,t) \le \frac{1}{C} \cdot \hat F^{x,y,z}(s,t)\,,\]
as required.

Suppose next that $|E| = 1$, and observe that, by symmetry, we may assume that $u_1 \in B_{(0,-1)}([S\cup T],R)$, as in Figure~\ref{fig:cornerap}. Recalling that the set $J(u_1)$ (in Algorithm~\ref{alg:app:first:case}) is empty, it follows that the event $D_2^{\hat\x}( \hat{S}(E),R )$ occurs, where $\hat{\x} = \x - \mathbbm{1}_{(1,0)} - \mathbbm{1}_{(0,1)} - \mathbbm{1}_{(0,-1)}$. By the induction hypothesis, we have
\[\Pr_p \big( D_2^{\hat\x}( \hat{S}(E),R ) \big) \le \hat F^{x-1,y-2,z}(s-i,t-j-2)\]
and, since $B/q \le a,b \le (1/(2q)) \log(1/q)$, we have (as before)
\[\frac{\hat F^{x-1,y-2,z}(s - i,t - j - 2)}{\hat F^{x,y,z}(s,t)} = \frac{e^{(2a+b)q}}{C^3 q^{3/2}} \exp\big( i g(bq) + (j+2) g(aq) \big) \le \frac{2^{i+j+2}}{C^3 q^3}\,.\]
Noting that there are at most $2k$ choices for the vertex $u_1$, and recalling that $k \le 9D \le 36\delta/q$, it follows that the probability of this case is at most
\[\sum_{i + j \ge 4} \sum_{k = i + j}^{9D} 2kp \cdot (24kp)^{k/2} \cdot \frac{2^{i+j+2}}{C^3 q^3} \cdot \hat F^{x,y,z}(s,t) \le \frac{1}{C^2} \cdot \hat F^{x,y,z}(s,t)\,,\]
as required.

Finally, suppose that $|E| = 2$, and observe that in this case the event $D_2^{\hat\x}( \hat{S}(E),R )$ occurs, where $\hat{\x} = \0$, and that $\x = \mathbf{1}$. By the induction hypothesis, we have
\[\Pr_p \big( D_2^{\hat\x}( \hat{S}(E),R ) \big) \le \hat F^{x-2,y-2,z}(s-i-2,t-j-2)\]
and, since $B/q \le a,b \le (1/(2q)) \log(1/q)$, we have (as before)
\[\frac{\hat F^{x-2,y-2,z}(s - i - 2,t - j - 2)}{\hat F^{x,y,z}(s,t)} = \frac{e^{(2a+2b)q}}{C^4 q^{2}} \exp\big( (i+2) g(bq) + (j+2) g(aq) \big) \le \frac{2^{i+j+4}}{C^4 q^4}\,.\]
Noting that there are at most $4k^2$ choices for $E$, and since $k \le 9D \le 36\delta/q$, it follows that the probability of this case is at most
\[\sum_{i + j \ge 4} \sum_{k = i + j}^{9D} (2kp)^2 \cdot (24kp)^{k/2} \cdot \frac{2^{i+j+4}}{C^4 q^4} \cdot \hat F^{x,y,z}(s,t) \le \frac{1}{C^3} \cdot \hat F^{x,y,z}(s,t)\,,\]
as required. This completes the proof in Case~1.
 
\bigskip
\noindent \textbf{Case 2:} $z > x + y$.
\medskip

As in the proof of Lemma~\ref{lem:key:small:app}, let $B_{(1,0)}(S,R)$ be a non-empty buffer that is not included in $\Sfr$, so $x_{(1,0)} = 0$, and define a set $E$ using Algorithm~\ref{alg:app:second:case}.

Suppose first that $E = \emptyset$, and recall that $\tilde S = \bigcup_{i = 0}^j \big( S + (i,0) \big)$, where
\[j = \min\big\{ i \ge 0 \,:\, A \cap R \cap \big( S + (i+2,0) \big) \setminus \big( S + (i,0) \big) = \emptyset \big\},\]
and that $\tilde S_{\square}^{\tilde{\x}} \setminus S_{\square}^\x$ contains no elements of $A$, where $\tilde{\x} = \x + \mathbbm{1}_{(1,0)}$. There are two sub-cases, depending on whether or not $B_{(1,0)}(\tilde S,R) = \emptyset$, that is, whether or not we reached the right-hand side without finding a double gap. Suppose first that we did find a double gap (i.e., $B_{(1,0)}(\tilde{S},R) \ne \emptyset$). We will sum over choices of $j$ the probability that 
\begin{equation}\label{eq:key:small:indep:events:again:app}
[ S \cup ( A \cap \tilde S  ) ] = \tilde{S}, \qquad \big[ \tilde{S} \cup \big( A \cap R \setminus \tilde{S}_{\blacksquare}^{\tilde \x} \big) \big] = R \qquad \text{and}\qquad A \cap \tilde S_\square^{\tilde \x} = \emptyset\,.
\end{equation}
Note that these three events depend on disjoint sets of sites, and are therefore independent; we will bound the first using Lemma~\ref{lem:doublegaps}, and the intersection of the second and third using the induction hypothesis. Indeed, by Lemma~\ref{lem:doublegaps} (and since $g(z)$ is decreasing) we have
\[\Pr_p\big( [ S \cup ( A \cap \tilde{S} ) ] = \tilde{S} \big) \le \exp\big( - j g(bq) \big)\,.\]
Moreover, the second and third events imply that the event $D_2^{\tilde \x}( \tilde{S},R )$ occurs, and by the induction hypothesis we have
\[\Pr_p \big( D_2^{\tilde \x}( \tilde{S},R ) \big) \le \hat F^{x+1,y,z}(s-j,t)\,.\]
It follows that the probability that there exists $j \ge 0$ such that the events in~\eqref{eq:key:small:indep:events:again:app} all hold is at most
\[\sum_{j=0}^{s-1} e^{- j g(bq)} \cdot \hat F^{x+1,y,z}(s-j,t) = Cs \sqrt{q} e^{-bq} \cdot \hat F^{x,y,z}(s,t) \le 4C\delta \cdot \hat F^{x,y,z}(s,t)\]
as required, since $s \le 4\delta \cdot q^{-1/2} \cdot \exp\big( \min\{a,b\} \cdot q \big)$ and $\delta = \delta(C) > 0$ is sufficiently small. 

We next deal with the case $E = \emptyset$ and $B_{(1,0)}(\tilde{S},R) = \emptyset$ (i.e., we reached the right-hand side without finding a double gap). In this case, the event $D_2^{\x}( \tilde{S},R )$ occurs, and by the induction hypothesis we have
\[\Pr_p \big( D_2^{\x}( \tilde{S},R ) \big) \le \hat F^{x,y,z-1}(s-j,t)\,,\]
since $z(\tilde S,R) = z(S,R) - 1$. It follows that the probability that the events in~\eqref{eq:key:small:indep:events:again:app} all hold, and $B_{(1,0)}(\tilde{S},R) = \emptyset$, is at most
\[\exp\big( - j g(bq) \big) \hat F^{x,y,z-1}(s-j,t) \le \frac{1}{C} \cdot \hat F^{x,y,z}(s,t)\]
which suffices since $C$ is sufficiently large.\medskip

We will therefore assume from now on that $E \ne \emptyset$, so $\tilde{S} = \bigcup_{i = 0}^j \big( S + (i,0) \big)$, where $j$ is minimal such that $\tilde{S} \setminus S$ is crossed from left to right, and 
\[\big( A \cap \tilde{S}_{\square}^{\tilde{\x}} \big) \setminus \big( S_{\square}^{\x} \cup B_{(1,0)}(\tilde{S},R) \big) \ne \emptyset\,.\]
In other words, before reaching a double gap we found an infected site in either the buffer above or below $\tilde{S}$. In particular, note that $\ell_{(0,1)} + \ell_{(0,-1)} \ge 1$.\medskip

Suppose first that $|E| = 1$, and therefore that (without loss of generality) we have $\ell_{(0,1)} = 1$, $E = \{ u_{(0,1)} \}$ and $\ell_{(0,-1)} = \ell_{(-1,0)} = 0$. Then the events
\begin{equation*}\label{eq:app:small:indep:events2}
u_{(0,1)} \in A, \quad [ S \cup ( A \cap \tilde S ) ] = \tilde S, \quad A \cap B_{(0,1)}(S,R) = \emptyset \quad \text{and} \quad D_2^{\hat \x}( \hat{S},R )
\end{equation*}
occur, where $\hat{S} = [\tilde{S} \cup E]$ and $\hat \x = \x - \mathbbm{1}_{(0,1)}$. There is an important subtlety in this case, since these events might not be independent: the buffer $B_{(0,1)}(S,R)$ might `stick out' of the top of $\hat S$, and therefore intersect the set of sites that the event $D_2^{\hat \x}( \hat{S},R )$ depends on. However, the only dependence is between the decreasing event $\{ A \cap B_{(0,1)}(S,R) = \emptyset \}$ and the increasing part of the event $D_2^{\hat \x}( \hat{S},R )$ (since $\hat x_{(0,1)} = 0$), so by Harris' inequality~\cite{Harris60}\footnote{Harris' inequality states that increasing events in a product space are positively correlated. It is often referred to as the FKG inequality, which is a generalization that was proved somewhat later.} the probability that all four events occur is at most the product of their probabilities.

Given $j$, there are at most four choices for $u_{(0,1)}$, so 
\[\Pr_p\Big( \big\{ u_{(0,1)} \in A \big\} \cap \big\{ [ S \cup ( A \cap \tilde{S} ) ] = \tilde{S} \big\}  \Big) \le 4p \exp\big( - j g(bq) \big)\,.\]
Suppose first that $z(\hat S,R) < z(S,R)$. In this case, by the induction hypothesis, we have
\[\Pr_p\Big( \big\{ A \cap B_{(0,1)}(S,R) = \emptyset \big\} \cap D_2^{\hat \x}( \hat{S},R ) \Big) \le (1 - p)^{a - s} \cdot \hat F^{x,y-1,z-1}(s - j - 1,t - 2)\,,\]
and so the probability in this case can be bounded by\footnote{We remark that this is the only point in the proof where we will need the term $e^{\sh(R) qz}$ in the bound in Lemma~\ref{lem:key:big}. This term gives rise to the term $\utr(\cH)$ in the proof of Theorem~\ref{thm:droplet} and the corresponding precision needed in Lemma~\ref{lem:upper:trunk:total:semiperimeter}.}
\begin{multline*}
\sum_{j = 0}^s 4p(1 - p)^{a - s} e^{- j g(bq)} \cdot \hat F^{x,y-1,z-1}(s - j - 1,t - 2) \\
\le 4(s+1) p\cdot e^{- (a-s)q} \cdot \frac{e^{aq - \min\{a,b\} q}}{C^2 \sqrt{q}} \cdot e^{2 g(aq) +g(bq)} \cdot \hat F^{x,y,z}(s,t) \le \frac{\delta}{C} \cdot \hat F^{x,y,z}(s,t)\,,
\end{multline*}
since $s \le 4\delta \cdot q^{-1/2} \cdot \exp\big( \min\{a,b\} \cdot q \big) \le 4\delta / q$, and hence $e^{sq + 2g(aq) + g(bq)} \le 2$. 

On the other hand, if $z(\hat S,R) = z(S,R)$ then the buffer $B_{(0,1)}(S,R)$ must have height at least two, and hence
\[\Pr_p\Big( \big\{ A \cap B_{(0,1)}(S,R) = \emptyset \big\} \cap D_2^{\hat \x}( \hat{S},R ) \Big) \le (1 - p)^{2(a - s)} \cdot \hat F^{x,y-1,z}(s - j - 1,t-2)\,.\]
The probability in this case can therefore be bounded, as above, by 
\begin{multline*}
\sum_{j = 0}^s 4p(1 - p)^{2(a - s)} e^{- j g(bq)} \cdot \hat F^{x,y-1,z}(s - j - 1,t - 2) \\
\le 4(s+1) p\cdot e^{- 2(a-s)q} \cdot \frac{e^{aq}}{C \sqrt{q}} \cdot e^{2 g(aq) +g(bq)} \cdot \hat F^{x,y,z}(s,t) \le \delta \cdot \hat F^{x,y,z}(s,t)\,,
\end{multline*}
as required.

\medskip
The remaining cases are similar but easier, since each extra infected site is extremely expensive. We will therefore be able to be use slightly weaker bounds, which simplifies the analysis somewhat. Suppose next that $|E| = 2$, so either 
\begin{itemize}
\item[$(a)$] $\ell_{(0,1)} = \ell_{(0,-1)} = 1$, $\ell_{(-1,0)} = 0$, and $E = \{ u_{(0,1)}, u_{(0,-1)} \}$, or \smallskip
\item[$(b)$] $\ell_{(0,1)} = \ell_{(-1,0)} = 1$, $\ell_{(0,-1)} = 0$, and $E = \{ u_{(0,1)}, u_{(-1,0)} \}$.
\end{itemize}
In either case, given $j$ there are at most five choices for each element of $E$, so 
\[\Pr_p\big( E \subset A \, \text{ and } \, [ S \cup ( A \cap \tilde{S} ) ] = \tilde{S} \big) \le (5p)^2 \exp\big( - j g(bq) \big)\,.\]
Moreover, in case~$(a)$, by the induction hypothesis and Harris' inequality, we have
\[\Pr_p\Big( \big\{ A \cap B = \emptyset \big\} \cap D_2^{\hat \x}( \hat{S},R ) \Big) \le (1 - p)^{2(a - s)} \cdot \hat F^{x,y-2,z}(s - j - 1,t - 4)\,,\]
where $B = B_{(0,1)}(S,R) \cup B_{(0,-1)}(S,R)$ and $\hat{\x} = \x - \mathbbm{1}_{(0,1)} - \mathbbm{1}_{(0,-1)}$.\footnote{Note that here (and also below) we could have gained substantially by dividing into cases, as above, depending on whether or not $z(\hat S,R) < z(S,R)$. In this case, however, this weaker bound will suffice.} The probability in this case can therefore be bounded by 
\[\sum_{j = 0}^s (5p)^2 (1 - p)^{2(a - s)} e^{- j g(bq)} \cdot \hat F^{x,y-2,z}(s - j - 1,t - 4) \le \frac{\delta}{C} \cdot \hat F^{x,y,z}(s,t)\,,\]
since $s \le 4\delta / q$ and $e^{2sq + 4g(aq) + g(bq)} \le 2$. Similarly, in case~$(b)$ we have
\[\Pr_p\Big( \big\{ A \cap B = \emptyset \big\} \cap D_2^{\hat \x}( \hat{S},R ) \Big) \le (1 - p)^{a - s + b - t} \cdot \hat F^{x-1,y-1,z}(s - j - 3,t - 2)\,,\]
where $B = B_{(0,1)}(S,R) \cup B_{(-1,0)}(S,R)$ and $\hat{\x} = \x - \mathbbm{1}_{(0,1)} - \mathbbm{1}_{(-1,0)}$, which allows us to bound the probability in this case by
\[\sum_{j = 0}^{s} (5p)^2 (1 - p)^{a - s + b - t} e^{- j g(bq)} \cdot \hat F^{x-1,y-1,z}(s - j - 3,t - 2) \le \frac{\delta}{C} \cdot \hat F^{x,y,z}(s,t)\,,\]
exactly as before. The calculation when $|E| = 3$ is almost the same. Recall that $\ell_{(0,1)} = \ell_{(0,-1)} = \ell_{(-1,0)}=1$ and, without loss of generality, either
\begin{itemize}
\item[$(a)$] $E = \{ u_{(0,1)}, u_{(0,-1)}, u_{(-1,0)} \}$, or \smallskip
\item[$(b)$] $E = \{ u_{(0,1)}, u_{(-1,0)}, u_{(-1,-1)} \}$.
\end{itemize}
In either case, given $j$ there are at most six choices for each element of $E$, so 
\[\Pr_p\big( E \subset A \, \text{ and } \, [ S \cup ( A \cap \tilde{S} ) ] = \tilde{S} \big) \le (6p)^3 \exp\big( - j g(bq) \big)\,.\]
Moreover, in either case, by the induction hypothesis and Harris' inequality, we have
\[\Pr_p\Big( \big\{ A \cap B = \emptyset \big\} \cap D_2^{\hat \x}( \hat{S},R ) \Big) \le (1 - p)^{2(a - s) + b - t} \cdot \hat F^{x-1,y-2,z}(s - j - 3,t - 4)\,,\]
where $B = B_{(0,1)}(S,R) \cup B_{(0,-1)}(S,R) \cup B_{(-1,0)}(S,R)$ and $\hat{\x} = \0$. The probability in this case can therefore be bounded by 
\[\sum_{j = 1}^s (6p)^3 (1 - p)^{2(a - s) + b - t} e^{- j g(bq)} \cdot \hat F^{x-1,y-2,z}(s - j - 3,t - 4) \le \frac{\delta}{C^2} \cdot \hat F^{x,y,z}(s,t)\,,\]
since $s \le 4\delta / q$ and $e^{2sq + tq + 4g(aq) + 3g(bq)} \le 2$. This completes the proof of Lemma~\ref{lem:key:big:app}.
\end{proof}

\section*{Acknowledgements}

The authors would like to thank B\'ela Bollob\'as for introducing them to the problem (several years apart), and for many interesting and inspiring conversations on the topic of bootstrap percolation. The second author would also like to thank Ander Holroyd for a useful conversation about the ideas behind the proof of our main theorem. 

This research was partly carried out while the authors were visiting IMT Lucca, and partly while the first author was visiting the University of Cambridge, and we are grateful to both institutions for their hospitality, and for providing a wonderful working environment. Our main theorem was first announced by the second author in 2009, and he would like to thank Murray Edwards College, Cambridge, the Japan Society for the Promotion of Science, and Keio University, which supported his research and hosted him during the early stages of this project. 

\bibliographystyle{amsplain}
%\bibliography{HMrefs}
\bibliography{D:/Master/Study/LaTeX/Bib/Bib}

\end{document}